\documentclass[11pt]{amsart}

 \usepackage{amsmath,amssymb}
\usepackage[margin=3.cm]{geometry}
\usepackage{hyperref}
\usepackage{comment}
\usepackage{bigints}
\usepackage{amssymb}%
\usepackage[all]{xy}
\usepackage{tikz-cd}
\usepackage{graphicx}
\graphicspath{ {./image/} }

\usepackage{xcolor}

\newcommand{\poubelle}[1]{}

\usepackage{colordvi}
\usepackage[all]{xy}
\usepackage{colortbl}

\usepackage{enumerate,wasysym,stmaryrd}

\usepackage{color}
\usepackage{fancybox}
\usepackage{colordvi}
\usepackage{multicol}
\usepackage{colordvi}
\usepackage{geometry}
\usepackage{lscape}
\usepackage{amsthm}
\usepackage{eqnarray,amsmath}
\usepackage{array}
\usepackage{booktabs}   
  \newcolumntype{x}[1]{>{\centering\hspace{0pt}}p{#1}} 
\def\R{\mathbb{R}}
\def\K{\mathbb{K}}

\def\{{\left\lbrace}
\def\}{\right\rbrace}

\newtheorem{theorem}{Theorem}[section]

\newtheorem{corollary}[theorem]{Corollary}
\newtheorem{define}[theorem]{Definition}

\newtheorem{lemma}[theorem]{Lemma}
\newtheorem{prop}[theorem]{Proposition}
\newtheorem{remark}[theorem]{Remark}

\renewcommand{\v}{\mathrm{v}}
\newcommand{\w}{\mathrm{w}}
\newcommand{\aut}{\mathrm{Aut}}
\newcommand{\red}{\mathrm{red}}

\newcommand{\M}{\mathbf {M}}

\newcommand{\J}{\mathbf {J}}

\newcommand{\T}{\mathbb{T}}
\renewcommand{\K}{\mathbb{K}}

\newcommand{\C}{\mathbb{C}}

\newcommand{\ext}{\mathrm{ext}}

\newcommand{\ent}{{\rm Ent}}

\newcommand{\Ent}{{\rm { Ent}}}

\newcommand{\vol}{{\rm Vol}}
\newcommand{\Tr}{\Lambda_0}

\newcommand{\uu}{\mathrm{v}}
\renewcommand{\K}{\mathcal{K}(X,\omega_0)^\T}
\newcommand{\Kn}{\mathring{\mathcal{K}}(X,\omega_0)^\T}
\newcommand{\TC}{\mathbb{T}^{\mathbb{C}}}

\usepackage[utf8]{inputenc}

\title[Weighted cscK metric (II): the continuity method]{Weighted cscK metric (II): the continuity method}

\author{Eleonora Di Nezza, Simon Jubert and Abdellah Lahdili}

\begin{document}

\begin{abstract}
In this paper we investigate the existence of metrics with weighted constant scalar curvature (wcscK for short) on a compact Kähler manifold $X$: this notion include constant scalar curvature K\"ahler metrics, weighted solitons, Calabi’s extremal Kähler metrics and extremal metric on semisimple principal fibrations. We prove that the coercivity of the weighted Mabuchi functional implies the existence of a wcscK metric, thereby achieving the equivalence. 

We then give several applications in Kähler and toric geometry, such as a weighted version of the toric Yau–Tian–Donaldson correspondence, and the characterization of the existence of wcscK metric on total space of semisimple principal fibration $Y$ in term of existence of wcscK metric on its fiber $X$.
\end{abstract}

\maketitle

\tableofcontents

\section{Introduction}
Let $(X, \omega_0)$ be a compact K\"ahler manifold of complex dimension $n$, endowed with an a action of a compact torus $\T$ in the group $\aut_{\red}(X)$ of reduced automorphisms of $X$ and let $\omega_0$ be a $\mathbb{T}$-invariant K\"ahler form on $X$. We denote $\mathcal{K}(X,\omega_0)^{\mathbb{T}}$ the set of smooth $\omega_0$-relative K\"ahler potentials which are $\mathbb{T}$-invariant. For $\varphi \in \mathcal{K}(X,\omega_0)^{\mathbb{T}}$ we denote by $\omega_{\varphi}=\omega_0+dd^c\varphi$  the corresponding $\T$-invariant K\"ahler metric. In our convention, $dd^cf=2i \partial \bar{\partial}f$, for any smooth function $f$ on $X$. It is well-known that the $\mathbb{T}$-action on $X$ is $\omega_{\varphi}$-Hamiltonian (see e.g. \cite[Section 2.5]{Gau}) and  that $P_{\varphi}:= \mu_{\varphi}(X)$ is a convex polytope in $\mathfrak{t}^{*}$ \cite{Ati, GS, Lah19}. Here $\mu_\varphi: X \rightarrow \mathfrak{t}^{*}$ is the moment map associated to $\omega_\varphi$ and $\mathfrak{t}^{*}$ stands for the dual vector space for the Lie algebra $\mathfrak{t}$ of $\mathbb{T}$. We normalize $\mu_{\varphi}$ by

\begin{equation}{\label{normalizing-moment-map}}
\mu_{\varphi}=\mu_{0} + d^c\varphi,
\end{equation}

\noindent in such a way that $P=P_{\varphi}$ is $\varphi$-independent, see \cite[Lemma 1]{Lah19}. Sometimes, when confusion may arise, we denote by $\mu_\omega$ the moment map of a given K\"ahler metric $\omega$. 
 
For a given positive weight function $\v \in {C}^{\infty}(P,\R_{>0})$, we define the $\v$-weighted scalar curvature of a $\T$-invariant K\"ahler metric $\omega_\varphi \in [\omega_0]$  by

\begin{equation}{\label{def:scalv2}}
 {\mathrm{Scal}_\v(\omega_\varphi)}:= 2\v(\mu_\varphi) \Lambda_{\varphi,\v}(\mathrm{Ric}_\v(\omega_\varphi)),
\end{equation}

\noindent where $\Lambda_{\varphi,\v}$ is the $\v$-weighted trace with respect to $\omega_\varphi$ (see \cite[Appendix A]{DJL}) and $ \mathrm{Ric}_\v(\omega_\varphi) := \mathrm{Ric}(\omega_\varphi) -\frac{1}{2}dd^c \log \v(\mu_\varphi)$ is the $\v$-weighted Ricci form of $\omega_\varphi$.

Given a second weight $\w \in {C}^{\infty}(P,\R)$, a $\T$-invariant K\"ahler metric $\omega_\varphi \in [\omega_0]$ is called \emph{$(\v,\w)$-weighted cscK} if its $\v$-weighted scalar curvature $\mathrm{Scal}_\v(\omega_\varphi)$ satisfies

\begin{equation}{\label{wcsck}}
    \mathrm{Scal}_\v(\omega_\varphi)=\w(\mu_\varphi).
\end{equation}

The significance of \eqref{wcsck} in relation to various geometric conditions is thoroughly examined in \cite{Lah19} (see also \cite{AJL} for $\v$-solitons and semisimple principal fibration).  However, we shall mention a few specific cases below:

\begin{enumerate}
    \item[(a)] for $\v = 1$ and $\w=c$ is a constant, \eqref{wcsck} corresponds to the classical cscK equation;
    \item[(b)]  for $\v = 1$ and $\w=\ell_{\ext}$, where $\ell_{\ext}$ is the affine extremal function, the weighted cscK metrics are Calabi's extremal K\"ahler metrics \cite{EC82};
    \item[(c)] for $\v>0$ is any smooth function and $\w=2\v(x)(n+ \langle d\log\v(x),x \rangle )$, \eqref{wcsck} corresponds to the weighted $\v$-soliton examined in \cite{AJL, HL}, generalizing the well-studied K\"ahler-Ricci solitons (see e.g. \cite{BW, Blo, TZ, TZ02});
    \item[(d)] for $\v$ and $\w$ polynomials, then $(\v,\w)$-cscK metrics on $X$ correspond to Calabi's extremal K\"ahler metrics on the total space of an holomorphic fibration $Y$ with fiber $X$, called semisimple principal fibration \cite{ACGT11, AJL, Jub23};
    \item[(e)] for $\v=\ell^{-n-1}, \w=a \ell^{-n-2}$,  $[\omega_0]=c_1(L)$, where $\ell$ is a positive affine-linear function on the polytope, $a$ is a constant,  and $L$ is  a polarization of $X$, \eqref{wcsck} describes  a scalar flat cone K\"ahler metric on the affine cone $(L^{-1})^{\times}$ polarized by the lift of $\xi = d\ell$ to $L^{-1}$ via $\ell$ \cite{AC,ACL}.
\end{enumerate}

In \cite{Lah19}, an extension of Calabi's extremal K\"ahler metrics \cite{EC82} is proposed. Suppose that $\v, \w_0>0$ are given positive smooth functions on $P$. One can then find a unique affine-linear function $\ell^{\ext}_{\v,\w_0}$ on $\mathfrak{t}^*$, called \emph{the weighted extremal function}, such that 

\begin{equation*}
\int_X \Big(\mathrm{Scal}_\v(\omega) - \ell^{\ext}_{\v,\w_0}(\mu_{\omega})\w_0(\mu_{\omega})\Big) \ell(\mu_{\omega})\omega^{n} =0
\end{equation*}
any affine-linear function $\ell$ on $\mathfrak{t}^*$.  In this case, a solution of the $(\v, \w)$-cscK problem \eqref{wcsck} is called \emph{$(\v,\w)$-extremal K\"ahler metric}. We emphasize that $(\v, \w)$-extremal K\"ahler metrics are  $(\v, \w)$-cscK metrics in the particular case $\w = \ell^{\ext}_{\v,\w_0} \w_0$ with $\w_0>0$ on $P$ and $\ell^{\ext}_{\v,\w_0}$ affine-linear. 

\smallskip

Our main result states as follows:

\begin{theorem}{\label{t:exi}} Let $\v>0$, $\w$ be two weight functions on $P$ such that $\v$ is $\log$-concave and $\v,\w$ satisfy
\begin{equation}\label{norm_intro}
   \int_X \big(\mathrm{Scal}_{\mathrm{v}}(\omega_{\varphi})-\mathrm{w}(\mu_{\varphi})\big) \omega_\varphi^{[n]} = 0,
\end{equation}
for any $\varphi\in \mathcal{K}(X,\omega_0)^{\mathbb{T}}$.
Assume that the weighted Mabuchi energy $\M_{\v,\w}$ is coercive, i.e. there exist positive constants $A, B\in \mathbb{R}$ such that
\[\M_{\v,\w}(\varphi) \ge A \inf_{\gamma \in \T^{\C}} d_1(0, \gamma \cdot \varphi) - B, \]
for any $\varphi\in \mathcal{K}(X,\omega_0)^{\mathbb{T}}$. Then there exists a $(\v,\w)$-cscK in $[\omega_0]$.
\end{theorem} 
Theorem \ref{t:exi} recovers the main result of \cite{HL25}, in which the authors establish the existence of weighted extremal metrics -a special class of weighted cscK metrics. Moreover, our proof differs in several aspects, particularly in the techniques used to obtain the $C^0$-estimates and in the application of the continuity method.
\smallskip

The weighted Mabuchi energy was first introduced in \cite{Lah19} and it is defined by its variation

\begin{equation*}
    D_{\varphi}\left(\mathbf{M}_{\v,\w}\right)(\dot{\varphi})=-\int_X \dot{\varphi}\big(\mathrm{Scal}_\v(\omega_\varphi)-\w(\mu_\varphi)\big)\v(\mu_\varphi)\omega_\varphi^{[n]}, \quad \mathbf{M}_{\v,\w}(0)=0,
\end{equation*}
so that its critical points are exactly $(\v, \w)$-cscK metrics. The coercivity condition links the behavior of $\mathbf{M}_{\v,\w}$ with respect to the distance $d_1$. This distance was defined and studied in \cite{Dar}. The tangent space of $ \mathcal{K}:=\mathcal{K}(X,\omega_0)^{\mathbb{T}}$ can be identified with the space of smooth functions and Darvas considered the norm 
$$\|\dot{\psi}\|_\varphi:= \int_X |\dot{\psi}| \omega_\varphi^n, \qquad \dot{\varphi} \in T_\varphi \mathcal{K}.$$ Such a norm gives a way to measure the length of smooth paths $\gamma:[0,1]\rightarrow \mathcal{K}$, given as
$$\ell_1(\gamma):= \int_0^1 \|\dot \gamma(t)\|_{\gamma(t)} \, dt$$
Given $\varphi_0, \varphi_1 \in \mathcal{K}$, the distance $d_1$ is defined as 
$$d_1(\varphi_0, \varphi_1):= \inf\{ \ell_1(\gamma)\quad \gamma:[0,1]\rightarrow \mathcal{K}, \gamma(0)=\varphi_0, \gamma(1)=\varphi_1\}.$$

\smallskip
It is worth mentioning that the condition on the weights \eqref{norm_intro} is very natural. Indeed, if a $(\v,\w)$-cscK metric exists, then the weighted Futaki invariant $\mathbf{F}_{\v,\w}$ defined on the space of affine functions on $P$

\begin{equation*}
\mathbf{F}_{\v,\w}(\ell):=\int_X \big(\mathrm{Scal}_\v(\omega) - {\w} (\mu_{\omega})\big) \ell(\mu_{\omega}) \omega^{[n]},
\end{equation*}
identically vanishes; and in particular the weights $(\v,\w)$ satisfy the normalization condition \eqref{norm_intro}. 

\smallskip

We also remark that the weight $\v$ associated to example \textnormal{(a)}, \textnormal{(b)}, \textnormal{(d)} and the one for the K\"ahler-Ricci solitons in \textnormal{(c)} are $\log$-concave.

Since $(\v,\w)$-extremal metric are particular examples of weighted $(\v,\w)$-cscK metric (with $\w=\ell \w_0$) our main result (Theorem \ref{t:exi}) combined with the main result in \cite[Theorem 1]{AJL} gives

\begin{theorem}
Assume $\T\subset {\rm Aut}_{\red}(X)$ is a maximal torus and $\v$ is $\log$-concave. Then there exists a $(\v,\w)$-extremal metric in $[\omega_0]$ if and only if the weighted Mabuchi energy $\M_{\v,\w}$ is $\TC$-coercive.
\end{theorem}

The strategy consists in resolving a weighted version of the continuity path introduced by Chen in \cite{Ch17}:

\begin{equation}{\label{int:weighted:cont}}
  \frac{t}{\v(\mu_{\varphi})}\big({\mathrm{Scal}}_{\v}(\omega_{\varphi})-\w(\mu_{\varphi})\big) = (1-t)(\Lambda_{\varphi,\v}(\omega_0)-\tilde{\v}(\mu_\varphi)), \text{ } \text{ } t\in[0,1],
  \end{equation}
 where $\varphi=\varphi_t$ and $  \tilde{\v}(x):= n + \langle d \log(\v(x)), x \rangle$. We note that this weight satisfies $\Lambda_{\varphi_0,\v}(\omega_0)=\tilde{\v}(\mu_0)$.
 
The openness and the existence of a solution (for small time) of the continuity path is established through a generalization of Hasimoto's result \cite{Hash}, and an application of a quantitative version of the Implicit Function Theorem.

The proof for closeness is divided into two parts: we first prove closeness up to a time $t < 1$, primarily relying on the a priori estimates for weighted cscK metrics established in our previous work \cite{DJL}. To extend the closeness to $t = 1$, we derive $W^{1,2p}$-estimates for solutions of a $(\v,\w)$-cscK type equation:

 \begin{equation}{\label{w:type:eq}}
     \begin{cases}
     F &= \log\left( \v(\mu_\varphi) \frac{\omega_\varphi^n}{\omega^n_0}\right) \\
    \Delta_{\varphi,\v} F  &=   \w ( \mu_\varphi) + 2 \Lambda_{\varphi,\v} ( \mathrm{Ric}(\omega_0 ) - \beta_\varphi).
     \end{cases}
\end{equation}
Here $\beta_\varphi$ is a given positive form, potentially depending on $\varphi$, and $\Delta_{\varphi, \v}$ represents the $\v$-weighted Laplacian with respect to $\varphi$ (see e.g. \cite[App. A]{DJL}). When $\beta_\varphi = 0$, the solutions of equation \eqref{w:type:eq} correspond to $(\v,\w)$-cscK metrics \eqref{wcsck} (c.f. \cite[Section 3]{DJL}).
\smallskip

Theorem \ref{t:exi} can be combined with previous known results to get new existence theorems.

\subsubsection*{Stability properties w.r.t. weights} We define the space of weights for which the weighted Futaki invariant vanishes:

\begin{equation*}
    F(X,[\omega_0]):= \{ (\v,\w) \in \mathcal{C}_{\log\text{-}\mathrm{c}}^{\infty}(P,\R_{>0}) \times \mathcal{C}^{\infty}(P,\R) \text{ s.t. } \mathbf{F}_{\v,\w} \equiv 0\}.
\end{equation*}
Here $C^\infty_{\log\text{-}c}$ is the set of smooth functions which are log-concave.

\begin{corollary} 
Set $$
 S(X,[\omega_0]):= \{ (\v,\w) \in \mathcal{C}_{\log\text{-}\mathrm{c}}^{\infty}(P,\R_{>0}) \times \mathcal{C}^{\infty}(P,\R) \text{ s.t. } \exists \, \omega \text{ } (\v,\w)\text{-cscK} \text{ in } [\omega_0]\}. $$ Then $ S(X,[\omega_0]) \subset F(X,[\omega_0])$. Moreover,
$S(X,[\omega_0])$ is open in $F(X,[\omega_0])$ for the smooth topology.
\end{corollary}

This generalizes \cite[Theorem 1.1]{ALL} which holds for $\v$-soltions.
\smallskip

\subsubsection*{Yau-Tian-Donaldson correspondence in the toric case}
We address the existence problem in the toric setting:
\begin{corollary}
Let $(X,\omega_0,\T)$ be a compact toric K\"ahler manifold and let $(P,\mathbf{L})$ be the labeled Delzant polytope. Assume $\v>0$ is $\log$-concave. Then there exists a $(\v,\w)$-cscK metric in $[\omega_0]$ if and only if $(P,\mathbf{L})$ is $(\v,\w)$-uniformly K-stable.
\end{corollary}
The notion of weighted stability is the one introduced in \cite[Section 7]{Jub23}, which extends the framework originally proposed by Donaldson \cite{SKD} for cscK metrics. 
For completeness, we mention that the converse direction ``existence implies stability" was previously proved in \cite{LLS}.

\smallskip
\subsubsection*{Semisimple principal fibration} 
We consider a semisimple principal fibration $(Y, \T)$ over a product of cscK manifold $(B,\omega_B):=\prod_{a=1}^k(B_a,\omega_{B_a})$ with fiber a K\"ahler manifold $(X, \T)$, as constructed in \cite{AJL} (see Section \ref{s:fibration}). In \cite[Theorem 2]{AJL}, it was proved that the existence of an extremal metric on the total space $Y$ is equivalent to the existence of a  weighted cscK metric on $X$ for an appropriate weight function. We extend this result to the weighted extremal setting:

\begin{corollary}{\label{c:fibration}}
Let $\tilde{\omega}_0$ be the compatible K\"ahler metric on $Y$ induced by a K\"ahler metric $\omega_0$ on $X$ and $(\v,\w)$ be weight functions such that $\v>0$ and $\log$-concave. The followings are equivalent:

\begin{itemize}
\item[(i)] $Y$ admits a $(\v,\w)$-extremal metric in $[\tilde{\omega}_0]$;
\item[(ii)] $X$ admits a  $\T$-invariant $(\mathrm{p}\v , \tilde \w)$-cscK metric in the K\"ahler class $[\omega_0]$,  where
\begin{equation*}
\mathrm{p}(x)=\prod_{a=1}^k(\langle p_a, x\rangle+ c_a)^{m_a}, \qquad \tilde \w(x)=  \w-\sum_{a=1}^k \frac{\mathrm{Scal}(\omega_{B_a})}{\left(\langle p_a, x \rangle + c_a\right)};
\end{equation*}
\end{itemize}
Moreover, suppose that $(X,\omega_0,\T)$ is a toric manifold with labeled Delzant polytope $(P,\mathbf{L})$. Then all conditions above are equivalent to:
\begin{itemize}
\item[(iii)] $(P,\mathbf{L})$ is $(\mathrm{p}\v ,\tilde{\w})$-uniformly $K$-stable.
\end{itemize}
\end{corollary}  
We refer to Section \ref{s:fibration} for the notations.

\subsection{Organisation of the paper}
In Section \ref{s:esti} we study a twisted weighted system which will appear naturally in Section \ref{s:close} when we take into account the action of the torus.
More precisely we derive a priori $W^{1,2p}$-estimates for solutions of the system \eqref{w:type:eq}. In Section \ref{s:open}, we show openess of the weighted continuity path \eqref{int:weighted:cont}, and in Section \ref{s:close} we prove its closeness. Section \ref{s:app} is devoted to the applications.


\section*{Acknowledgements}
The first and second author are funded by the ERC SiGMA - 101125012 (PI: Eleonora Di Nezza). The third-named author gratefully acknowledges insightful discussions with Chung-Ming Pan and the hospitality of CIRGET-UQAM during the preparation of this work. Part of this material is based upon work done while the first author was supported by the National Science Foundation under Grant No. DMS-1928930, while the first author was in residence at the Simons Laufer Mathematical Sciences Institute
(formerly MSRI) in Berkeley, California, during the Fall 2024 semester. \\
We thank the referee for their comments who helped to improve the paper.

\section{Estimates for a weighted cscK type equation}{\label{s:esti}}
As it will be explained in Section \ref{s:close}, in order to run the continuity method to solve the weighted cscK equation we need to consider the following elliptic PDEs system

 \begin{equation}\label{wsystem1}
     \begin{cases}
     F &= \log\left( \v(\mu_\varphi) \frac{\omega_\varphi^n}{\omega^n_0}\right) \\
    \Delta_{\varphi,\v} F  &=   \w ( \mu_\varphi) + 2 \Lambda_{\varphi,\v} ( \mathrm{Ric}(\omega_0 ) - \beta_\varphi),
     \end{cases}
\end{equation}
where $\w\in C^{\infty}(P, \R)$, $\beta_\varphi$ is a  $\T$-invariant K\"ahler form possibly depending on $\varphi$. We write $\beta_\varphi:=\beta_0 + dd^cf$, for $\beta_0$ independent on $\varphi$ and $f$ normalize such that $\sup_Xf=0$. We let $\mu_{\beta_0}$ the moment map of $\beta_0$ and we normalize the moment map $\mu_{\beta_\varphi}$ by

\begin{equation}{\label{moment:beta}}
    \mu_{\beta_{\varphi}} = \mu_{\beta_0} + d^cf.
\end{equation}
The image of $\mu_{\beta_{\varphi}}$ is a compact polytope in $\mathfrak{t}^*$ independent of $\varphi$. In particular, for any vector $\xi \in \mathfrak{t}$, $  d^cf(\xi)$ is bounded independently of $\varphi$. Then \eqref{wsystem1} is equivalent to

 \begin{equation}\label{wsystem2}
     \begin{cases}
     F &= \log\left( \v(\mu_\varphi) \frac{\omega_\varphi^n}{\omega^n_0}\right) \\
    \Delta_{\varphi,\v}( F +f)  &=   \w ( \mu_\varphi) + 2 \Lambda_{\varphi,\v} ( \alpha_0),
     \end{cases}
\end{equation}
where $\alpha_0:= \mathrm{Ric}(\omega_0)-\beta_0$.   Observe that for $\beta_\varphi$=0, we recover the system  studied in \cite{DJL} corresponding to the weighted cscK equation. The goal to this section is to show the following estimates:

\begin{theorem}{\label{t:estimates:w12p}}
Suppose $\v$ is log-concave. Let $\varphi$ be a smooth solution of \eqref{wsystem1}. Let $p_0 \geq \kappa_n>1$ for some constant $\kappa_n$ depending only on $n$. Then for any $q<p_0$,

$$
\|F+f\|_{W^{1,2 q}(\omega^n_0)} \leq C, \quad\|n+\Delta \varphi\|_{L^q\left(\omega^n_0\right)} \leq C,
$$

\noindent where $C$ depends on $\omega_0$, $\|\w\|_{C^0}$, $\|\v\|_{C^0}$, $\max _X\left|\beta_0\right|_0$, $\Ent(\varphi)$, $p_0$, and $\int_X e^{-p_0 f} \omega_0^{[n]}$.
\end{theorem}

\subsection{Estimates for uniform norm}

In this section, we prove a-priori $C^0$-estimates for $\varphi$ solution of \eqref{wsystem2}.

\begin{theorem}\label{t:c0:estimates}
The functions $\varphi, F+f$ are uniformly bounded by a constant that only depends on $\omega_0$, $\ent(\varphi)$ and $\int_Xe^{-p_0 f} \omega_0^{[n]}$.
\end{theorem}

In order to prove the above theorem we introduce an auxiliary function. Let $\psi$ be the unique solution of the Monge-Amp\`ere equation
$$\omega_\psi^{[n]}= b^{-1} \sqrt{F^2+1}  \omega_\varphi^{[n]}, \quad \sup_X \psi = 0, $$
\noindent $b:= \int_X \sqrt{F^2+1}\,\omega_\varphi^{[n]}$. We normalize $\omega_0$ such that $
    \int_X \omega_0^{[n]}=1 $.
As observed in \cite[Section 4]{DJL}, $0<b$ is dominated by the weighted entropy $\Ent_\v(\varphi):= \int_X \log\left( \frac{\v(\mu_\varphi) \omega_\varphi^n}{m_\v \omega_0^n} \right) \frac{\v(\mu_\varphi) \omega_\varphi^{[n]}}{{m_\v}}$, which is comparable to the classical entropy $\Ent(\varphi):= \int_X \log\left( \frac{\omega_\varphi^n}{\omega_0^n} \right)  \omega_\varphi^{[n]}$.

The same computations in \cite[Theorem 4.2]{DJL} applied to $H:=F+f+\varepsilon \psi- A\varphi$ give

\begin{prop}\label{step 1 C0 weighted}
Given $\varepsilon\in(0,1)$, there exists $C=C(\varepsilon, n, \omega_0, \w, \v, b)$ such that

\begin{equation*}
    F+f+\varepsilon \psi- A\varphi\leq C,
\end{equation*}
where $A>0$ is a uniform constant depending only on the lower bound of $\mathrm{Ric}(\omega_0)$. 
\end{prop}

\begin{corollary}\label{coro-covid-19}
The functions $\psi, \varphi, f+F$ are uniformly bounded by a constant that only depends on $\omega_0$, $\v$, $\w$ and $\Ent(\varphi)$ and $\int_Xe^{-p_0 f} \omega_0^{[n]}$. In particular, $F$ is bounded from below by a constant $C$ with the same dependence.
\end{corollary}
\begin{proof}
From  Proposition \ref{step 1 C0 weighted} we know that $F+f\leq C-\varepsilon\psi+A\varphi \leq C-\varepsilon \psi$,
since $\sup_X \varphi=0$. Therefore
\begin{equation*}
\begin{split}
 \int_X e^{2F} \omega_0^{[n]}= & \int_X e^{2(F+f)} e^{-2f} \omega_0^{[n]} \\
 \leq &   \left(\int_X e^{2q_0(F+f)}\omega_0^{[n]} \right)^{1/q_0} \left(\int_X e^{-2p_0f}  \,\omega_0^{[n]} \right)^{1/p_0}   \\
  \leq & C' \left(\int_X e^{-2q_0\varepsilon \psi} \omega_0^{[n]} \right)^{1/q_0} \left(\int_X e^{-p_0f} \omega_0^{[n]} \right)^{2/p_0},  
\end{split}
\end{equation*}
where $q_0$ is the conjugate exponent of $p_0/2$.
Choosing $\varepsilon $ such that $\varepsilon q_0< \nu_{\omega_0}^{-1}$, by \cite[Theorem 8.11]{GZ17} we get a bound for $\|e^F\|_{L^2}$ which depends on $\int_X e^{-p_0f} \omega_0^{[n]}$. It follows from Ko{\l}odziej uniform estimates \cite{Kolo} applied to the equation $\omega_\varphi^{[n]}=e^F\omega_0^{[n]}$, that $\varphi \geq -C(\|e^{F}\|_{L^2},\omega_0)$. In particular, since $\sup_X\varphi=0$ we do get a uniform control on $\|\varphi\|_{L^{\infty}}.$ Also, the same type of computations as before give

\begin{equation*}
\begin{split}
  \int_X e^{2F} (F^2+1) \,\omega_0^{[n]} \leq  & \int_X e^{4F}  \,\omega_0^{[n]}   \\
\leq  & \left( \int_X e^{4q_1(F+f)}\omega_0^{[n]}\right)^{1/q_1}  \left(\int_X e^{-p_0f}  \,\omega_0^{[n]} \right)^{4/p_0}   \\
  \leq & C' \int_X e^{-4q_1\varepsilon\psi} \omega_0^{[n]},
\end{split}  
\end{equation*}
where now $q_1$ is the H\"older conjugate of $p_0/4$. Once again, thanks to Theorem \cite[Theorem 8.11]{GZ17}, choosing $\varepsilon$ small enough we get a bound for $\|e^F\sqrt{F^2+1}\|_{L^2}$ which depends on $\omega_0, n, \int_X e^{-p_0f} \omega_0^{[n]}$. \cite[Theorem 2.2]{DJL}  then gives a uniform control for $\|\psi\|_{L^{\infty}}.$
We can then infer that $$F+f\leq C-\varepsilon\psi +A\varphi\leq -\varepsilon\inf_X \psi \leq C_4$$ for some uniform positive constant $C_4$.

The proof of the lower bound of $F+f$ follows the arguments in \cite[Lemma 2.3.]{CC21c} by applying the weighted Laplacian to $e^{-a(F+bf+C \varphi)}.$
\end{proof}

\subsection{Integral second order estimates }

The proof is similar to that of \cite[Theorem 5.1]{DJL} (where $f=0$) but it presents crucial modifications.  

\begin{lemma}{\label{l:trace:bound}}
Let $\varphi$ be solution of \eqref{wsystem2}. Then there exist a positive constants $C$,   depending on $\|F\|_{C^0}$, $\|\v\|_{C^0}$, $\inf_P(\v)$, and $n$ such that
\begin{equation*}
    \Lambda_0(\omega_\varphi) \leq C e^{-f} \Lambda_\varphi (\omega_0)^{n-1}, \qquad   \Lambda_\varphi(\omega_0)\leq n \sup_{P}|\v| e^{-F} \Lambda_0(\omega_\varphi)^{n-1} \leq C \Lambda_0(\omega_\varphi)^{n-1},
\end{equation*}
and
\begin{equation*}
    \Lambda_0(\omega_\varphi) C_1 \geq 1  \hspace*{2cm}   e^{-f/n} \Lambda_\varphi(\omega_0) C_2 \geq 1,
\end{equation*}
\end{lemma}

\begin{proof}
Recall that for smooth positive $(1,1)$-forms $\alpha$, $\beta$
\begin{equation*}
  n\left(\frac{\alpha^n}{\beta^n}\right)^{1/n} \leq  \Lambda_\beta(\alpha) \leq n \frac{\alpha^n}{\beta^n}\Lambda_\alpha(\beta)^{n-1}.
\end{equation*}
By Corollary \ref{coro-covid-19} and the second inequality above, we infer 
$$\Lambda_0(\omega_\varphi) \leq n \frac{e^F}{\v(\mu_\varphi)}  \Lambda_\varphi (\omega_0)^{n-1}=n \frac{e^{F+f}e^{-f}}{\v(\mu_\varphi)}  \Lambda_\varphi (\omega_0)^{n-1} \leq Ce^{-f}\Lambda_\varphi (\omega_0)^{n-1}. $$
Similarly, using the fact that $F$ is uniformly bounded form below we get
$$\Lambda_\varphi(\omega_0) \leq n \v(\mu_\varphi) {e^{-F}}  \Lambda_0 (\omega_\varphi)^{n-1} \leq C \Lambda_0 (\omega_\varphi)^{n-1} $$\\
For the other side we find 
\begin{equation*}
  \Lambda_0(\omega_\varphi) \geq n \left( \frac{e^F}{\v(\mu_\varphi) } \right)^{1/n} \geq C \quad \text{ and } \quad  \Lambda_\varphi(\omega_0) \geq n \left( {\v(\mu_\varphi)  e^{-F}} \right)^{1/n} \geq C' e^{f/n}.
\end{equation*}

\end{proof}

\begin{theorem}{\label{t:int:c2:estimates}}
Assume $\v$ is log-concave and that $\int_X e^{-p_0f} \omega_0^{[n]}< +\infty$ for some $p_0>1$. Then, for any $p\geq1$, there exists a constant $C>0$, depending on $\|F + f\|_{C^0}$, $\|\varphi\|_{C^0}$, $\v$, $\w$,  $\sup_X|\beta_0|_0$, $\omega_0$, a bound for $\int_X e^{-p_0f} \omega_0^{[n]}$, and $p$, such that

\begin{equation*}
    \int_Xe^{(p-1)f}\Lambda_0(\omega_\varphi)^p\omega_0^{[n]}\leq C.
\end{equation*}

\end{theorem}

\begin{proof}

Consider $$u:= e^{-\gamma (F+\lambda \varphi+{\delta f})} \Tr\, (\omega_\varphi) \geq 0,$$ where $\gamma,\lambda>1$ and $\delta\in(0,1)$ are constants to be chosen later in a suitable way. Standard computations (see also the beginning of the proof of \cite[Theorem 5.1]{DJL}) give

\begin{equation*}
\begin{split}
    \Delta_{\varphi,\v} u  
    \geq & u \, \Delta_{\varphi,\v} \log u \\
    =&  -\gamma u\Delta_{\varphi,\v} (F {+ f} + \lambda \varphi +{(\delta-1)} f) + u \Delta_{\varphi,\v} \left( \log \Lambda_0(\omega_\varphi)\right).
\end{split}
\end{equation*}
From \eqref{wsystem2} and the fact that $\Lambda_{\varphi,\v}(\omega_0) \leq \Lambda_\varphi(\omega_0)+C$, $C>0$ uniform, we deduce that
 
 \begin{equation*}
     \begin{split}
         \Delta_{\varphi,\v} (F +{f} + \lambda \varphi + {(\delta-1) f)}
    \leq   C_3 + (A - \lambda) \Lambda_\varphi(\omega_0) +  (\delta-1) \Delta_{\varphi,\v} f,
     \end{split}
 \end{equation*}
where $C_3$ is a constant independent of $\varphi$ and $A$ is such that $2\alpha_0\leq A \omega_0$. Since $\v$ is $\log$-concave, it follows from \cite[Lemma 5.6]{DJL} that

\begin{equation*}
\begin{split}
\Delta_{\varphi,\v} u  \geq & e^{-\gamma (F+\lambda \varphi + \delta f)} \bigg(-\gamma C_3   \Tr( \omega_\varphi) +  \Delta_{0}F+(\lambda \gamma -A\gamma - B)  \Tr( \omega_\varphi) \Lambda_{\varphi}(\omega_0) \bigg) \\
&+ \gamma  u(1-\delta)\Delta_{\varphi,\v} f,
\end{split}
\end{equation*}
where $B>0$ is a lower bound for the holomorphic bisectional curvature of $\omega_0$ and $\Delta_0=\Delta_{\omega_0}$.

We now choose $\lambda \geq 4 \max (A, B)$ (in order to have $\lambda \gamma -A\gamma -B\geq \frac{\lambda \gamma }{2}$) so that
\begin{equation*}
\begin{split}
\Delta_{\varphi,\v} u  \geq &  -\gamma C_3  u + \frac{\lambda \gamma}{2}  \Lambda_{\varphi}(\omega_0) \,u +e^{-\gamma (F+\lambda \varphi+\delta f)} \Delta_{0} F + \gamma u (1-\delta)\Delta_{\varphi,\v}f.
\end{split}
\end{equation*}
The same arguments in \cite[eq. (36)]{DJL} show that
\begin{equation}{\label{integral bound}}
\begin{split}
 0  &= \frac{1}{2p+1}\int_X \Delta_{\varphi, \v} (u^{2p+1}) \, \v(\mu_\varphi) \omega_\varphi^n\\
&\geq   2p \int_X u^{2p-2}|d u |_0^2  e^{-\gamma (F+\lambda \varphi + {\delta f)}+F} \omega_0^{[n]} - \gamma C_3 \int_X u^{2p+1} e^F \omega_0^{[n]}\\
&\quad+\frac{\gamma \lambda}{2}  \int_X u^{2p} \Tr(\omega_\varphi) \Lambda_{\varphi}(\omega_0) e^{-\gamma (F+\lambda \varphi + {\delta f)}+F}\omega_0^{[n]} \\
&\quad+\int_X u^{2p} e^{-\gamma (F+\lambda \varphi+ \delta f)+F} \Delta_{0} F\omega_0^{[n]} \\
&\quad+ \int_X u^{2p} \gamma (1-\delta)\Delta_{\varphi,\v} f\Lambda_{0}(\omega_\varphi) e^{-\gamma (F+\lambda \varphi+ {\delta f)}+F}\omega_0^{[n]}.
\end{split}
\end{equation}

Next, we focus on finding a suitable lower bound for the term involving the Laplacian of $F$. Set $G:=(1-\gamma)F-\gamma\lambda \varphi{ - \gamma \delta f}$. A formal trick gives that
\begin{equation*}
\begin{split}
I:=& -\int_X u^{2p} \Delta_{0} F e^{G} \, \omega_0^{[n]}\\
=& \frac{1}{\gamma-1} \int_X u^{2p} \,\Delta_{0} G e^{G} \, \omega_0^{[n]}+\frac{\gamma \lambda }{\gamma-1} \int_X u^{2p} \Delta_{0}\varphi\, e^{G} \, \omega_0^{[n]} + \frac{\gamma\delta}{ \gamma-1}\int_Xu^{2p}\Delta_{0}f\, e^{G} \omega_0^{[n]} \\
:=&   I_1+I_2  + \frac{\gamma\delta}{\gamma-1} \int_X u^{2p}\Delta_{0}f \, e^{G} \omega_0^{[n]} 
\end{split}
\end{equation*}
By \cite[eq. (37)]{DJL}

\begin{equation}{\label{bound I1}}
 I_1 
\leq  \frac{ 2p^2}{\gamma -1} \int_X u^{2p-2} |du|_0^2 e^G \,\omega_0^{[n]}.
\end{equation}
 Choosing $\delta=\frac{\gamma-1}{\gamma}$ (i.e. $1-\delta=1/\gamma$) and using Corollary \ref{coro-covid-19}, we deduce that $G 
=(1-\gamma)(F+f)-\gamma \lambda \varphi \leq C$ and that $u \leq C \Lambda_0(\omega_\varphi)e^{\gamma(1-\delta)f}=C \Lambda_0(\omega_\varphi)e^{f}$. Hence
\begin{equation}{\label{bound I222}}
\begin{split}
I_2 = &\frac{\gamma \lambda }{\gamma-1} \int_X u^{2p} \Lambda_{0}(\omega_\varphi) e^{G}\, \omega_0^{[n]} -\frac{n\gamma \lambda }{\gamma-1} \int_X u^{2p}  e^{G} \, \omega_0^{[n]} \\
\leq & C \int_X u^{2p}  \Lambda_{0}(\omega_\varphi)\, \omega_0^{[n]} \\
\leq &  C_5 \int_X e^{2p{f}} \Lambda_0(\omega_\varphi)^{2p+1} \omega_0^{[n]}.
\end{split}
\end{equation}
Also,

\begin{eqnarray*}
\nonumber  \gamma(1-\delta)\Lambda_0(\omega_\varphi) \Delta_{\varphi,\v} f- \frac{\gamma \delta}{\gamma-1}\Delta_{0}f &=&  \Lambda_0(\omega_\varphi) \Delta_{\varphi,\v}f-\Delta_{0}f 
\nonumber    \\
& = & \Lambda_0(\omega_\varphi)\Delta_{\varphi}f  +\Lambda_0(\omega_\varphi) \langle d\log(\v(\mu_\varphi)), d^cf \rangle - \Delta_{0}f \\
 \nonumber       &\geq& -C_6 \Lambda_0(\omega_\varphi)  -\Lambda_0(\omega_\varphi)\Delta_\varphi f - \Delta_0f \\
 & \geq &- C_6 \Lambda_0(\omega_\varphi) - \max_X|\beta_0|_{0} \Lambda_0(\omega_\varphi)\Lambda_\varphi(\omega_0).
\end{eqnarray*}

For the first inequality we use that $d^cf(\xi)$ is bounded independently of $\varphi$, see the discussion below \eqref{moment:beta}. We use \cite[(2.33)]{CC21c} for the last inequality. 
 We then obtain:

\begin{equation}{\label{comp:int}}
    \begin{split}
&\frac{\gamma \lambda}{2}  \int_X u^{2p} \Tr(\omega_\varphi) \Lambda_{\varphi}(\omega_0)e^G\omega_0^{[n]}+ \int_X u^{2p}\left(\gamma(1-\delta)\Lambda_0(\omega_\varphi) \Delta_{\varphi,\v}f- \frac{\gamma \delta}{\gamma-1}\Delta_{0}f\right)e^G\omega_0^{[n]}        \\
\geq&   \int_X \left(  \frac{\gamma \lambda}{2}   -   \max_X|\beta_0|_{0}\right) \Lambda_0(\omega_\varphi)\Lambda_\varphi(\omega_0)u^{2p}e^G\omega_0^{[n]} -C_6 \int_X u^{2p} \Lambda_0(\omega_\varphi)e^G\omega_0^{[n]} \\
  \geq & C_8 \int_X u^{2p}  \Tr( \omega_\varphi) \Lambda_{\varphi}(\omega_0) e^G\omega_0^{[n]} -C_7 \int_X  e^{2pf}\Lambda_0(\omega_\varphi)^{2p+1}\omega_0^{[n]}\\
    \geq & C_9 \int_X u^{2p+1}  e^{\left(\frac{n-2}{n-1}\right)F} (\Tr\,(\omega_\varphi))^{\frac{1}{n-1}}\omega_0^{[n]} -C_7 \int_X  e^{2pf} \Lambda_0(\omega_\varphi)^{2p+1}\omega_0^{[n]},
    \end{split}
\end{equation}
where we choose $\lambda$ big enough such that $C_9:= \frac{\gamma \lambda}{2} -\max|\beta_0|_0>0$. For the last inequality we use that $ \Tr(\omega_\varphi) \Lambda_\varphi (\omega_0) \geq  C  e^{\frac{-F}{n-1}} \Tr(\omega_\varphi)^{1+\frac{1}{n-1}}$ via similar arguments to those used in the proof of Lemma \ref{l:trace:bound}.

Combining  \eqref{integral bound}, \eqref{bound I1}, \eqref{bound I222}, \eqref{comp:int} and choosing $\gamma$ big enough (say $\gamma= a p$, with $a>>1$) we obtain that

\begin{equation*}
\begin{split}
0&\geq  2\left( p-\frac{p^2}{\gamma -1}\right)   \int_X u^{2p-2} |d u|_0^2 e^{G} \, \omega_0^{[n]} - (\gamma C_3+C_5+C_7) \int_X e^{2pf} \Lambda_0(\omega_\varphi)^{2p+1} \omega_0^{[n]} \\
&\quad +C_9 \int_X u^{2p+1} e^{\left(\frac{n-2}{n-1}\right)F} (\Tr\, \omega_\varphi)^{\frac{1}{n-1}} \omega_0^{[n]}  \\
&\geq  -C_{10} \int_X  e^{2pf}\Tr\,  (\omega_\varphi)^{2p+1}\,\omega_0^{[n]} + C_{11} \int_X e^{\left(2p+1- \frac{n-2}{n-1}\right)f} (\Tr\, \omega_\varphi)^{2p+1+\frac{1}{n-1}}\,\omega_0^{[n]}. 
\end{split}
\end{equation*}
Note that $p-\frac{p^2}{\gamma -1}>0$ thanks to the choice of $\gamma$. Also, by the choice of $\delta$ we do have $ u= e^{-\lambda \gamma \varphi}\Lambda_0(\omega_\varphi)e^{-\gamma (F+f)+f} 
 \geq  C \Lambda_0(\omega_\varphi) e^f$, hence the last inequality.\\
By taking $2p=\frac{k}{n-1}$ with $k\geq0$, we find that
$I(k):= \int_X  e^{\frac{k}{n-1}f}\Tr\,  (\omega_\varphi)^{\frac{k}{n-1}+1}\,\omega_0^{[n]} $ satisfies $I(k+1)\leq C' I(k).$ In particular, for any $k\in \mathbb{N}$, $I(k)$ is bounded by $I(0)= \int_X \Tr\,  (\omega_\varphi)\,\omega_0^{[n]} = \frac{1}{(n-1)!}\int_X   \omega_\varphi \wedge \omega_0^{n}=\frac{\vol(\omega_0)}{(n-1)!}$. This concludes the proof.
\end{proof}

The proof of the corollary below is essentially \cite[Corollary 2.4]{CC21c} but we write down the details for reader’s convenience. 

\begin{corollary}{\label{c:int:estim}}
Assume $\v$ is log-concave and that $\int_X e^{-p_0f} \omega_0^{[n]}< +\infty$ for some $p_0>1$. For any $1 < q < p_0$, there exists a constant $C_q$, depending only on
the bound of $\int_Xe^{-p_0f}\omega_0^{[n]}$, $\|\w\|_{C^0}$,  $\|\v\|_{C^0}$, $\max_X|\beta_0|_0$, $\|F+f\|_{C^0}$, $\|\varphi_0\|_{C^0}$, $\omega_0$ such that

\begin{equation*}
    \int_X\Lambda_0(\omega_\varphi)^q\omega_0^{[n]} \leq C_q.
\end{equation*}

\end{corollary}

\begin{proof}
Choose $s = \frac{(q-1)p_0}{p_0 - 1}$. We compute

\begin{equation*}
    \begin{split}
        \int_X \Lambda_0(\omega_\varphi)^q \omega_0^{[n]} =& \int_X e^{-sf} \cdot e^{sf} \Lambda_0(\omega_\varphi)^q \omega_0^{[n]} \\
        \leq &\left( \int_X e^{-p_0 f} \omega_0^{[n]} \right)^{\frac{s}{p_0}}  
\left( \int_X e^{\frac{s p_0}{p_0 - s} f} (\Lambda_0(\omega_\varphi))^{\frac{p_0 q}{p_0 - s}} \omega_0^{[n]} \right)^{1 - \frac{s}{p_0}}.
    \end{split}
\end{equation*}
By definition of $s$ we have that 
\[
\frac{s p_0}{p_0 - s} = \frac{p_0 q}{p_0 - s} - 1,
\]
Hence the result follows from Theorem \ref{t:int:c2:estimates}.
\end{proof}

\subsection{Integral estimates for the gradient}


We start with the following lemma:

\begin{lemma}{\label{l:w:eq:F:plus:f}}

Let $H:=F+f$. Then
\begin{equation*}
\begin{split}
\Delta_{\varphi, \v}\left(e^{H/2} |dH|_\varphi^{2}\right)\geq &  2e^{H/2} g_\varphi( d\Delta_{\varphi, \v} H,dH)  \\
&- C e^{H/2} |dH|_\varphi^{2} \left(   e^{-2F}\Lambda_0(\omega_\varphi)^{2n-2}+e^{-F}\Lambda_0(\omega_\varphi)^{n-1}+1 \right) .
\end{split}
\end{equation*}
\end{lemma}

\begin{proof}
By \cite[eq. (42)]{DJL} we know that
\begin{equation}\label{ineq:c2:0}
\begin{split}
e^{-H/2}\Delta_{\varphi,\v}\Big(e^{H/2} |dH|_\varphi^{2}\Big) &\geq  2 g_\varphi( d\Delta_{\varphi,\v} H,dH) +2|\nabla^{\varphi,+}dH|_\varphi^{2}\\
&+2\left(\mathrm{ric}_{\varphi}+\nabla^{\varphi,+}dH-\nabla^{\varphi}d\log \v(\mu_\varphi)\right)(\nabla^{\varphi}H,\nabla^{\varphi}H)\\
&+\frac{1}{2}\Delta_{\varphi,\v} H |dH|_\varphi^{2},
\end{split}    
\end{equation}
where $\nabla^{\varphi, +}\alpha$ for a 1-form $\alpha$ is defined by

\begin{equation}{\label{nabla:pm}}
\nabla_V^{\varphi, \pm}\alpha (W):=\frac{1}{2}\left(\nabla^\varphi_V\alpha(W) {\pm}\nabla^\varphi_{JV}\alpha(JW)\right)
\end{equation}
for any vector fields $V$, $W$ and $\mathrm{ric}_\varphi$ is the symmetric Ricci-$2$ tensor associate to the K\"ahler metric $\omega_\varphi$. We estimate the second line of \eqref{ineq:c2:0}. Using the fact that $
    \mathrm{ric}_{0}=\mathrm{ric}_{\varphi}+\nabla^{\varphi,+}dF -\nabla^{\varphi,+}d\log\v(\mu_\varphi) $ (cf \cite[page 25]{DJL}) and \cite[Lemma 6.2]{DJL} (recall that $\beta_\varphi=\beta_0+dd^c f>0$) we arrive at

\begin{equation*}
\begin{split}
    \mathrm{ric}_{\varphi}+\nabla^{\varphi,+}dH-\nabla^{\varphi}d\log \v(\mu_\varphi)& = \mathrm{ric}_{\varphi}+\nabla^{\varphi,+}dF +\nabla^{\varphi,+}df -\nabla^{\varphi}d\log \v(\mu_\varphi) \\
    &= \mathrm{ric}_0 + \nabla^{\varphi,+}df -\nabla^{\varphi,-}d\log \v(\mu_\varphi) \\
    & = \mathrm{ric}_0 + \frac{1}{2}(\beta_\varphi -   \beta_0)(\cdot, J\cdot)-\nabla^{\varphi,-}d\log \v(\mu_\varphi) \\
    & \geq \mathrm{ric}_0 - \frac{1}{2} \beta_0(\cdot, J\cdot)-\nabla^{\varphi,-}d\log \v(\mu_\varphi).
\end{split}    
\end{equation*}
Substituting back in \eqref{ineq:c2:0}, we obtain

\begin{equation*}
\begin{split}
e^{-H/2}\Delta_{\varphi,\v}\Big(e^{H/2} |dH|_\varphi^{2}\Big) &\geq  2 g_\varphi( d\Delta_{\varphi,\v} H,dH) +2|\nabla^{\varphi,+}dH|_\varphi^{2}\\
&+2\left( \mathrm{ric}_0 -  \frac{1}{2} \beta_0(\cdot, J\cdot)-\nabla^{\varphi,-}d\log \v(\mu_\varphi)\right)(\nabla^{\varphi}H,\nabla^{\varphi}H)\\
&+\frac{1}{2}\Delta_{\varphi,\v} H |dH|_\varphi^{2}.
\end{split}    
\end{equation*}
We also have

\begin{equation*}
\begin{split}
\left|\big(\mathrm{ric}_0-\frac{1}{2}\beta_0(\cdot, J\cdot)\big) (\nabla^{\varphi}H,\nabla^{\varphi}H)\right|=&\left|g_\varphi\left(\mathrm{ric}_0-\frac{1}{2}\beta_0(\cdot, J\cdot),dH\otimes dH\right)\right|\\
\leq& \left|\mathrm{Ric}_0-\frac{1}{2}\beta_0  \right|_\varphi |dH|^{2}_\varphi\\
\leq& C |dH|^{2}_\varphi\, \Lambda_\varphi (\omega_0)^2 \\
\leq& C' |dH|^{2}_\varphi e^{-2F}\, \Lambda_0(\omega_\varphi)^{2(n-1)}.
\end{split}
\end{equation*}
In the second line we apply Cauchy-Schwarz inequality for $2$-tensors. In the third line the constant $C$ is such that $\mathrm{Ric}_0-\frac{1}{2}\beta(\cdot, J\cdot) \leq C\omega_0$ and we use the well known inequality (see e.g. \cite[(1.12.5)]{Gau})
\begin{equation*}
    |\omega_0|^2_\varphi = \Lambda_\varphi(\omega_0)^2 - \frac{\omega_0^2\wedge \omega_\varphi^{[n-2]}}{\omega_0^{[n]}} \leq \Lambda_\varphi(\omega_0)^2.
\end{equation*}
The last inequality follows from Lemma \ref{l:trace:bound}. By \cite[(45)]{DJL} and Lemma \ref{l:trace:bound} we deduce
 
 \begin{equation*}
    \Delta_{\varphi, \v} H |dH|_\varphi^2 \geq  -C \left( \Lambda_\varphi(\omega_0)+ 1\right) |dH|^2\geq  - C\big( e^{-F} \Lambda_0(\omega_\varphi)^{n-1}+1\big) |dH|_\varphi^2,
 \end{equation*}
which conclude the proof.
\end{proof}

Now we show this key Proposition which provides an estimate for the norm of the differential of $F+f$ with respect to $\varphi$.

\begin{prop}{\label{p:estim:c1}}
There exists $k_n$, depending only on $n
$, such that as long as $p_0 > k_n$, we
have
\begin{equation*}
    | d (F+f) |^2_{\varphi} \leq C, 
\end{equation*}
where $C$ is as in Theorem \ref{t:int:c2:estimates}.
\end{prop}

\begin{proof}
Denote $H= F+f$, $u:=e^{\frac{1}{2} H}\left|d H\right|_{\varphi}^{2}+1, \tilde{G}=C\left(2 e^{-2 F}\Lambda_0(\omega_\varphi)^{2 n-1}+e^{-F}\Lambda_0(\omega_\varphi)^{n-1}+1\right)$. By Lemma \ref{l:w:eq:F:plus:f}, we have
\begin{equation*}
\Delta_{\varphi,\v} u \geq 2 e^{\frac{1}{2} H} g_\varphi(d H, d \Delta_{\varphi,\v} H)-u \tilde{G}.
\end{equation*}
 By definition of $\Delta_{\varphi,\v}$ and $d^*_\varphi$, for any $\T$-invariant functions $f_1$ and $f_2$ on $X$ we have that (see e.g. \cite[Appendix A]{DJL})

\begin{equation}\label{def d*}
\begin{split}
\int_X(- \Delta_{\varphi,\v} f_1) f_2  \, \v(\mu_\varphi) \omega_\varphi^{[n]}=& \int_X  d^*_\varphi\big(\v(\mu_\varphi)df_1\big) f_2  \,\omega_\varphi^{[n]} \\
=&\int_X g_{\varphi}(df_1,df_2)  \v(\mu_\varphi) \omega_\varphi^{[n]}.
\end{split}
\end{equation} 
Hence, for any $p \geq 1$

\begin{eqnarray}\label{ineq. dH}
\nonumber  \int_{X}(p-1) u^{p-2}\left|d u\right|_{\varphi}^{2} \v(\mu_\varphi) \omega^{[n]}_{\varphi}
 &=&\int_{X} u^{p-1}\left(-\Delta_{\varphi,\v} u\right) \v(\mu_\varphi)\omega^{[n]}_{\varphi}\\
& \leq& \int_{X} u^{p} \tilde{G} \v(\mu_\varphi)\omega^{[n]}_{\varphi}  \\
\nonumber & &- \int_X2 u^{p-1} e^{\frac{1}{2} H} g_\varphi(d H, d \Delta_{\varphi,\v} H)\v(\mu_\varphi)\omega^{[n]}_{\varphi}.
\end{eqnarray}
 We now estimate the last term of the above inequality. By definition of $u$ we have

\begin{equation}{\label{eq:u:lapH}}
   u^{p-1} e^{\frac{1}{2} H}\left|d H\right|_{\varphi}^{2} \Delta_{\varphi,\v} H=u^{p} \Delta_{\varphi,\v} H-u^{p-1} \Delta_{\varphi,\v} H. 
\end{equation} 
Using \eqref{def d*} and the above identity we obtain

\begin{equation*}
\begin{split}
- \int_X2 u^{p-1} e^{\frac{1}{2} H}& g_\varphi(d H, d \Delta_{\varphi,\v} H)\v(\mu_\varphi)\omega^{[n]}_{\varphi}\\ =&\int_{X} 2 u^{p-1} e^{\frac{1}{2} H}(\Delta_{\varphi,\v}H)^{2} \v(\mu_\varphi) \omega^{[n]}_{\varphi} +\int_{X} u^{p-1} e^{\frac{1}{2} H}\left|d H\right|_{\varphi}^{2} \Delta_{\varphi,\v} H \v(\mu_\varphi) \omega^{[n]}_{\varphi}\\
 &+\int_{X} 2(p-1) u^{p-2} e^{\frac{1}{2} H} g_\varphi(d u , d H) \Delta_{\varphi,\v} H \v(\mu_\varphi) \omega^{[n]}_{\varphi} \\
 \leq& \int_{X} 2 u^{p-1} e^{\frac{1}{2} H}{\left(\Delta_{\varphi,\v} H\right)^{2}}\v(\mu_\varphi) \omega^{[n]}_{\varphi} +\int_{X} u^{p} \Delta_{\varphi,\v}( H )\v(\mu_\varphi)\omega^{[n]}_{\varphi}\\
 &-\int_{X} u^{p-1} \Delta_{\varphi,\v}( H )\v(\mu_\varphi)\omega^{[n]}_{\varphi}  +\int_{X} \frac{p-1}{2} u^{p-2}\left|d u\right|_{\varphi}^{2}\v(\mu_\varphi)\omega^{[n]}_{\varphi}\\
&+\int_{X} 2(p-1) u^{p-2} e^{H}\left|d H\right|_{\varphi}^{2}\left(\Delta_{\varphi,\v} H\right)^{2} \v(\mu_\varphi) \omega^{[n]}_{\varphi}\\
 \leq &\int_{X} 2 {p} u^{p-1} e^{\frac{1}{2} H}\left(\Delta_{\varphi, {\v}} H\right)^{2} \v(\mu_\varphi) \omega^{[n]}_{\varphi}+\int_{X} u^{p}\left(\left({\Delta_{\varphi,\v}} H\right)^{2}+1\right) \v(\mu_\varphi) \omega^{[n]}_{\varphi} \\
 &+\int_{X} \frac{p-1}{2} u^{p-2}\left|d u\right|_{\varphi}^{2} \v(\mu_\varphi) \omega^{[n]}_{\varphi}.
\end{split}
\end{equation*}
In the first inequality we used Young inequality which gives

$$
2(p-1) u^{p-2} e^{\frac{1}{2} H} g_{\varphi}(d u, dH ) \Delta_{\varphi,\v} H \leq \frac{p-1}{2} u^{p-2}\left|d u\right|_{\varphi}^{2}+2(p-1) u^{p-2} e^{H}\left|d H\right|_{\varphi}^{2}\left(\Delta_{\varphi,\v} H\right)^{2}.
$$
In the second inequality we use the inequality $ab-cb\leq |a-c| |b|\leq |a+c| \frac{(1+b^2)}{2}$. Together with the fact that $u\geq 1$, this gives

$$
u^{p} \Delta_{\varphi,\v} H-u^{p-1} \Delta_{\varphi,\v} H \leq \frac{1}{2}\left(u^{p}+u^{p-1}\right)\left(1+\left(\Delta_{\varphi,\v} H\right)^{2}\right) \leq u^{p}\left(1+\left(\Delta_{\varphi,\v} H\right)^{2}\right).
$$
We also use \eqref{eq:u:lapH} applied to $$2(p-1) u^{p-2} e^{H}\left|d H\right|_{\varphi}^{2}\left(\Delta_{\varphi,\v} H\right)^{2}= 2(p-1)  e^{H/2} \left( u^{p-2} e^{H/2}\left|d H\right|_{\varphi}^{2} \Delta_{\varphi,\v} H \right) \Delta_{\varphi,\v} H.$$ Injecting in \eqref{ineq. dH} and using the fact that $u \geq 1$ we get

\begin{equation*}
\begin{split}
 \int_{X} \frac{p-1}{2} u^{p-2}\left| du\right|_{\varphi}^{2} \v(\mu_\varphi) \omega^{[n]}_{\varphi} \leq &\int_{X} u^{p}\left(\tilde{G}+\left(\Delta_{\varphi,\v} H\right)^{2}+1\right) \v(\mu_\varphi) \omega^{[n]}_{\varphi} \\
 &+\int_{X} 2 {p} u^{p} e^{\frac{1}{2} H}\left(\Delta_{\varphi, {\v}} H\right)^{2} \v(\mu_\varphi) \omega^{[n]}_{\varphi}.
\end{split}
\end{equation*}
 Let $G:=\tilde{G}+\left(\Delta_{\varphi,\v} H\right)^{2}+1+2 e^{\frac{1}{2} H}\left(\Delta_{\varphi,\v} H\right)^{2}$. Using \eqref{wsystem2}, we deduce that

\begin{equation*}
\int_{X} \frac{p-1}{2} u^{p-2}\left|d u\right|_{\varphi}^{2} \v(\mu_\varphi) \omega_\varphi^{[n]} \leq  \int_{X} {p} u^{p} G e^{F}  \omega^{[n]}_{0}.
\end{equation*}
Standard computations give $\left|d\left(u^{\frac{p}{2}}\right)\right|= \frac{p^2}{4} u^{p-2} |du|^2_\varphi$. We then infer 

\begin{equation*}
\int_{X} \frac{p-1}{2} u^{p-2}\left|d u\right|_{\varphi}^{2}\v(\mu_\varphi) \omega^{[n]}_{\varphi} \geq C \int_{X} \frac{2(p-1)}{p^{2}}\left|d\left(u^{\frac{p}{2}}\right)\right|_{\varphi}^{2} \omega^{[n]}_{0}.
\end{equation*}
In the above inequality we use that $\v(\mu_\varphi) \omega^{n}_{\varphi}= e^F\omega_0^n$ and that 

\begin{equation*}
  e^F= e^{F+f-f}\geq e^{-\| F+f \|_{C^0}}e^{-f} \geq e^{-\| F+f \|_{C^0}}>C,  
\end{equation*}
where $C>0$ is uniform thanks to Theorem \ref{t:c0:estimates}. The same arguments in \cite[(2.62)]{CC21c} show that for $\varepsilon$ small enough,

\begin{equation}{\label{bound:u}}
\left\|u^{\frac{p}{2}}\right\|_{L^{\frac{2 n(2-\varepsilon)}{2 n-2+\varepsilon}}}^{2} \leq C \frac{p^{3}}{p-1}\left(K_{\varepsilon} L_{\varepsilon}+1\right)\left\|u^{\frac{p}{2}}\right\|_{L^{\frac{4}{2-\varepsilon}}}^{2}, 
\end{equation}
where

\begin{equation*}
L_{\varepsilon}:=\left(\int_{X} G^{\frac{2}{\varepsilon}} e^{\frac{2 F}{\varepsilon}} \omega_0^{[n]}\right)^{\frac{\varepsilon}{2}} \quad \text{ and } \quad  K_{\varepsilon}:=\left(\int_{M}(\Lambda_0(\omega_\varphi))^{\frac{2}{\varepsilon}-1} \omega^{[n]}_{0}\right)^{\frac{\varepsilon}{2-\varepsilon}}.
\end{equation*}
\\
We now establish a bound for $K_{\varepsilon}, L_{\varepsilon}$ for $\varepsilon=\frac{1}{2n}$. Hence, we need to show that $\int_{X}\Lambda_0(\omega_\varphi)^{4 n-1} \omega_0^{[n]}$ and  $\int_{X} G^{4 n} e^{4 n F} \omega_0^{[n]}$ are bounded. \\
From Corollary \ref{c:int:estim} we know that $K_{1/2n}=\int_{X}\Lambda_0(\omega_\varphi)^{4 n-1} \omega_0^{[n]}$ is bounded as soon as $p_0>4n$. In order to bound the second quantity we observe that $$\tilde{G}:=C\left(2 e^{-2 F}\Lambda_0(\omega_\varphi)^{2 n-1}+e^{-F}\Lambda_0(\omega_\varphi)^{n-1}+1\right)\leq C_1\Lambda_0(\omega_\varphi)^{2 n-1}$$ thanks to the fact that $F$ is uniformly bounded (Corollary \ref{coro-covid-19}). Also, we recall that $H:=F+f$ is uniformly bounded. Hence by \eqref{wsystem2}
\begin{equation*}
\begin{split}
 G=& \tilde{G}+ (2e^{\frac{1}{2}(F+f)}+1)\left(\Delta_{\varphi,\v}(F+f)\right)^{2}+1 \\
\leq &  C_1 \Lambda_0(\omega_\varphi)^{2 n-1} + C_2 (\w(\mu_\varphi)+2 \Lambda_{\varphi,\v}(\alpha_0))^2 \\
\leq& C_{3}(\Lambda_0(\omega_\varphi)^{2 n-1}+ \Lambda_{\varphi}(\omega_0)^{2})  \\
 \leq &C_{4}\Lambda_0(\omega_\varphi)^{2 n-1},
\end{split}
\end{equation*}
where the last inequality follows from Lemma \ref{l:trace:bound}. 

Thus, the above bound together with H\"older inequality and Theorem \ref{t:c0:estimates} give
\begin{eqnarray*}
 \int_{X} G^{4n} e^{4nF}\omega_0^{[n]}&=& \int_{X} G^{4n} e^{4n(F+f)} e^{-4nf}\omega_0^{[n]}\\
 &\leq & C_4 \int_{X} \Lambda_0(\omega_\varphi)^{4n(2n-1)} e^{4n(F+f)} e^{-4nf}\omega_0^{[n]}\\
  &\leq & C_5 \left(\int_{X} \Lambda_0(\omega_\varphi)^{ \frac{4n(2n-1) p_0}{p_0-4n}} \omega_0^{[n]} \right)^{\frac{p_0-4n}{p_0}} 
   \left(\int_{X} e^{-p_0f}\omega_0^{[n]}\right)^{\frac{4n}{p_0}}.
\end{eqnarray*}
Choosing $p_0$ big enough  (e.g. $p_{0} \geq 8n^2$) we get a bound of $K_{\varepsilon}$ and $L_{\varepsilon}$.\\
Once \eqref{bound:u} is in hand, we use the same arguments in \cite[page 10]{DD21} to ensure that the $L^\infty$-norm of $u$ is controlled by the $L^1$-norm. To conclude we then need to establish an estimate for $\|u\|_{L^1}$. We then observe that 

\begin{equation*}
\Delta_{\varphi,\v}\left(e^{\frac{1}{2} H}\right)=\frac{1}{4} e^{\frac{1}{2} H}\left|d H\right|_{\varphi}^{2}+\frac{1}{2} e^{\frac{1}{2} H} \Delta_{\varphi,\v} H,
\end{equation*}
and

\begin{equation*}
\begin{split}
 \int_{X} e^{\frac{1}{2} H} \left|d H\right|_{\varphi}^{2} \omega_0^{[n]} &=\int_{X} e^{\frac{1}{2} H}\left|d H\right|_{\varphi}^{2}e^{-F} \v(\mu_\varphi)\omega_\varphi^{[n]}\\
 &\leq C \int_{X} e^{\frac{-1}{2} H}\left|d H\right|_{\varphi}^{2} \v(\mu_\varphi)\omega_\varphi^{[n]} \\
 &= C \int_{X} 2 e^{\frac{1}{2} H}\left(-\Delta_{\varphi,\v} H\right) \v(\mu_\varphi)\omega_\varphi^{[n]}  \\
  &\leq C_1 \int_{X}\left(-\w(\mu_\varphi) -2 \Lambda_{\varphi,\v}(\alpha_0)\right) \v(\mu_\varphi)\omega_\varphi^{[n]} \\
&  \leq C_2\sup \v\int_{X} \left(\Lambda_\varphi(\omega_0) +1\right)\, \omega_\varphi^{[n]} \\
&\leq C_{3} \operatorname{Vol}(\omega_0),
\end{split}
\end{equation*}
where in the third line we use \eqref{wsystem2} and that $H$ is uniformly bounded. The last inequality simply follows from $\int_{X} \Lambda_\varphi(\omega_0) \omega_\varphi^{n} = n\int_{X}  \omega_0\wedge  \omega_\varphi^{n-1}=n\int_{X}  \omega_0^n$. This concludes the $L^1(\omega_0)$-bound of $u$, hence the proof.
\end{proof}

We can now conclude the proof of Theorem \ref{t:estimates:w12p}.

\begin{proof}
For any function $u$ we have $|du|^2_0\leq |du|^2_\varphi \Lambda_0(\omega_\varphi)$, thus by Proposition \ref{p:estim:c1} we have

\begin{equation*}
        |d(F+f)|^2_{0}  \leq    |d (F+f)|^2_{\varphi} \Lambda_0(\omega_\varphi) \leq C \Lambda_0(\omega_\varphi).
\end{equation*}
Theorem \ref{t:int:c2:estimates} allows to conclude.
\end{proof}

\section{Openness and existence of a twisted solution}{\label{s:open}}

\subsection{The weighted continuity path of Chen}

\noindent We suppose that $\v$, $\w$ satisfy

\begin{equation}{\label{norm:scalv}}
   \int_X \big(\mathrm{Scal}_{\mathrm{v}}(\omega_{\varphi})-\mathrm{w}(\mu_{\varphi})\big) \omega_\varphi^{[n]} = 0,
\end{equation}
for any $\varphi \in \mathcal{K}(X,\omega_0)^\T$. We observe that if the weighed Futaki invariant vanishes, then \eqref{norm:scalv} is satisfied. We consider the continuity path for $ \varphi \in \mathcal{K}(X,\omega_0)^{\mathbb{T}}$ given by

\begin{equation}{\label{continuity-path-weithed}}
  \frac{t}{\v(\mu_\varphi)}\big(\mathrm{Scal}_{\v}(\omega_{\varphi})-\w(\mu_{\varphi})\big) = (1-t)(\Lambda_{\varphi,\v}(\omega_0)-\tilde{\v}(\mu_\varphi)), \text{ } \text{ } t\in[0,1],
  \end{equation}
where $\varphi=\varphi_t$ and $\tilde{\v}\in C^{\infty}(P,\mathbb{R})$ is the weight function
\begin{equation}\label{eq:tilde-v}
    \tilde{\v}(x):= n + \langle d \log(\v(x)), x \rangle, \quad x \in P.
\end{equation}
    By definition, a solution $\varphi_t$ at $t=1$ is $(\v,\w)$-cscK and we call a solution at $t\in (0,1)$ a $\omega_0$-twisted $(\v,\w)$-cscK metric. For $t_1 \in (0,1]$, we define

\begin{equation}{\label{set-solution}}
    S_{t_1}:=\{ t \in (0,t_1] \text{ }| \text{ } \eqref{continuity-path-weithed} \text{ has a solution } \varphi_t \in \mathcal{K}(X,\omega_0)^{\mathbb{T}}\}.
\end{equation}

\noindent The goal is to show that $S_{t_1}$ is open, closed and non-empty in $(0,1]$. This will imply that $S_{t_1}=S_1=(0,1]$. In particular \eqref{continuity-path-weithed} admits a solution for $t=1$.

Let $r=\frac{(1-t)}{t}$ and $\theta\in [\omega_0]$ a $\T$-invariant K\"ahler form. We consider the nonlinear operator $\mathcal{S}^{(r)}_\theta : \mathcal{K}(X,\omega_0)^\T \longrightarrow \mathcal{C}^{\infty}(X)^{\T}$ given by
\begin{equation}\label{eq:S_r}
   \mathcal{S}^{(r)}_\theta(\varphi):=\frac{1}{\v(\mu_{\varphi})}(\mathrm{Scal}_{\v}(\omega_\varphi)-\w(\mu_{\varphi}))-r\left(\Lambda_{\varphi, \v}(\theta)-\tilde{\v}(\mu_{\varphi})\right),
\end{equation}
Since $\mathcal{S}^{(r)}_\theta$ is translation invariant, i.e. $\mathcal{S}^{(r)}_\theta(\varphi)=\mathcal{S}^{(r)}_\theta(\varphi +C)$ for any constant $C$, we will sometimes use both notations $\mathcal{S}^{(r)}_\theta(\omega_\varphi)$, $\mathcal{S}^{(r)}_\theta(\varphi)$.

We observe that for any $\varphi \in \mathcal{K}(X,\omega_0)^\T$, the image $\mathcal{S}^{(r)}_\theta(\varphi)$ is $L^2\left(\v(\mu_\varphi)\omega_\varphi^{[n]}\right)$-orthogonal to constants. Indeed it follows from \eqref{norm:scalv} and Lemma \ref{cohomo:const:trace} below that

\begin{equation*}
    \int_X (\Lambda_{\varphi,\v}(\theta)-\tilde{\v}(\mu_\varphi))\v(\mu_{\varphi}) {\omega_\varphi}^{[n]}= \int_X (\Lambda_{\varphi,\v}(\omega_\varphi)-\tilde{\v}(\mu_\varphi))\v(\mu_\varphi) \omega_\varphi^{[n]}=0.
\end{equation*}

\begin{lemma}{\label{cohomo:const:trace}}
Let $\omega$ be a $\T$-invariant K\"ahler form and $\theta$ be a  $\T$-invariant closed real $(1,1)$-form on $X$. The integral
\begin{equation}\label{Const}
\int_X\Lambda_{\omega,\v}(\theta)\, \v(\mu_\omega)\omega^{[n]}
\end{equation}
is constant and it does not dependent on the choice of a $\T$-invariant representative $\omega,\theta$ respectively in $[\omega]$, $[\theta]$.
\end{lemma}
\begin{proof}
The fact that \eqref{Const} does not depend on the choice of a $\T$-invariant representative of $\omega_\varphi\in [\omega]$ is proved in \cite[Lemma 2]{Lah19}. Now, let $\theta+dd^c u=:\theta_u \in [\theta]$ be a $\T$-invariant form. We then have 
\[
\begin{split}
&\int_X\Lambda_{\omega,\v}(\theta)\v(\mu_\omega)\omega^{[n]}\\
&=\int_X\Lambda_{\omega,\v}(\theta)\v(\mu_\omega)\omega^{[n]} +\int_X \Delta_{\omega,\v}(u) \, \v(\mu_{\omega})\omega^{[n]}\\
&= \int_X\Lambda_{\omega,\v}(\theta)\v(\mu_\omega)\omega^{[n]}+\int_X\Lambda_{\omega,\v}(dd^c u) \v(\mu_{\omega})\omega^{[n]}\\
&=\int_X\Lambda_{\omega,\v}(\theta_{u})\v(\mu_{\omega})\omega^{[n]}.
\end{split}
\]
where $\Lambda_{\omega,\v}(\theta)+\Lambda_{\omega,\v}(dd^c u)= \Lambda_{\omega,\v}(\theta_u)$ since $\mu_{\theta_u}=\mu_\theta+d^{c}u$ (see \cite[Lemma 1]{Lah19}).
\end{proof}

\subsection{Linearization of the continuity path}
Following  \cite{Hash, Jub23} we introduce the operator $\mathbb{H}^{\theta}_{\varphi,\v} : \mathcal{C}^{\infty}(X)^{\mathbb{T}} \longrightarrow \mathcal{C}^{\infty}(X)^{\mathbb{T}}$  defined by

\begin{equation}{\label{hash}}
  \mathbb{H}^{\theta}_{\varphi,\v}(\dot{\varphi}):= g_{\varphi}\big(\theta, dd^c \dot{\varphi} \big) + g_{\varphi}\big(d \Lambda_{\varphi} \theta, d\dot{\varphi}\big)  + g_{\varphi}\big(\theta,d\log(\v(\mu_{\varphi})) \wedge d^c\dot{\varphi}\big).
\end{equation}

\begin{lemma}{\label{lem-F-elliptic}}
Suppose that $\theta$ is K\"ahler. The operator $\mathbb{H}^\theta_{\varphi, \v}$ is a $\v(\mu_\varphi)\omega_\varphi^{[n]}$-self-adjoint, $\T$-invariant, second order elliptic linear operator, satisfying
\begin{equation}\label{eq-(Fphi,psi)}
\int_X \mathbb{H}^\theta_{\varphi, \v}(\dot{\varphi})\dot{\psi} \,\v(\mu_\varphi)\omega_{\varphi}^{[n]}=-\int_Xg_{\varphi}\left(\theta,d\dot{\varphi}\wedge d^c\dot{\psi}\right)\uu(\mu_\varphi)\omega_{\varphi}^{[n]}.
\end{equation}
In particular $\ker( \mathbb{H}^\theta_{\varphi, \v})$ consists of constant functions.
\end{lemma}
\begin{proof}
From the definition it is clear that $\mathbb{H}^\theta_{\varphi, \v}$ is a second order linear operator. Since $\theta$ is a $\T$-invariant K\"ahler form, it follows that $\mathbb{H}^\theta_{\varphi, \v}$ is a $\T$-invariant elliptic operator. We now show \eqref{eq-(Fphi,psi)}. By definition we have
\begin{align}
\begin{split}\label{(F-v-w,psi)}
\int_X\dot{\psi}\mathbb{H}^\theta_{\varphi, \v}\dot{\varphi})\v(\mu_{\varphi})\omega_{\varphi}^{[n]}
=&\int_X\left( g_{\varphi}(\theta,dd^c\dot{\varphi}\right)\dot{\psi}\uu(\mu_\varphi)\omega_{\varphi}^{[n]}+\int_Xg_{\varphi}\left(d\Lambda_\varphi(\theta),d\dot{\varphi}\right)\dot{\psi}\uu(\mu_\varphi)\omega_{\varphi}^{[n]}\\
&+ \int_X\dot{\psi} g_{\varphi}\big(\theta,d\log(\v(\mu_{\varphi})) \wedge d^c\dot{\varphi}\big)\v(\mu_\varphi)\omega_{\varphi}^{[n]}.
\end{split}
\end{align}
We recall the well-know identity  (see e.g. \cite[(1.12.5)]{Gau})
\begin{equation*}
g_\varphi\left(\theta,dd^c\dot{\varphi}\right)\omega_{\varphi}^{[n]}=-\theta\wedge dd^{c}\dot{\varphi}\wedge \omega^{[n-2]}_\varphi+(\Delta_\varphi\dot{\varphi})(\Lambda_\varphi\theta)\omega^{[n]}_\varphi.
\end{equation*} 
Integrating by parts the above gives
\begin{align}
\begin{split}\label{(theta,ddc)}
&\int_X g_\varphi\left(\theta,dd^c\dot{\varphi}\right)\dot{\psi}\uu(\mu_\varphi)\omega_{\varphi}^{[n]}=\\
&+\int_X \dot{\psi}(\Delta_\varphi\dot{\varphi})(\Lambda_\varphi\theta)\, \uu(\mu_\varphi)\omega^{[n]}_\varphi-\int_X\dot{\psi}\uu(\mu_\varphi)\theta\wedge dd^{c}\dot{\varphi}\wedge \omega^{[n-2]}_\varphi=\\
&-\int_Xg_\varphi\left(d\Lambda_\varphi\theta,d\dot{\varphi}\right)\dot{\psi}\uu(\mu_\varphi)\omega_{\varphi}^{[n]}-\int_Xg_\varphi\left(d\dot{\varphi},d\dot{\psi}\right)(\Lambda_\varphi\theta)\uu(\mu_\varphi)\omega_{\varphi}^{[n]}\\
& -\int_Xg_\varphi\left(d\dot{\varphi},d\log(\v(\mu_\varphi))\right)(\Lambda_\varphi\theta)\dot{\psi}\v(\mu_\varphi)\omega_{\varphi}^{[n]}+\int_X\uu(\mu_\varphi)\theta\wedge d\dot{\psi}\wedge d^{c}\dot{\varphi}\wedge \omega^{[n-2]}_\varphi\\
& +\int_X\dot{\psi}\theta\wedge d\uu(\mu_\varphi)\wedge d^{c}\dot{\varphi}\wedge \omega^{[n-2]}_\varphi.
\end{split}
\end{align}
Using again \cite[(1.12.5)]{Gau} we find
\begin{align*}
\theta\wedge d\dot{\psi}\wedge d^{c}\dot{\varphi}\wedge \omega^{[n-2]}_\varphi=&\left(-g_\varphi\left(\theta,d\dot{\psi}\wedge d^{c}\dot{\varphi}\right)+g_\varphi\left(d\dot{\varphi},d\dot{\psi}\right)(\Lambda_\varphi\theta)\right)\omega^{[n]}_\varphi,
\end{align*}
and
\begin{eqnarray*}
\theta\wedge d(\v(\mu_\varphi))\wedge d^{c}\dot{\varphi}\wedge \omega^{[n-2]}_\varphi
&=&\left(-g_\varphi\left(\theta,(d(\v(\mu_\varphi))\wedge d^{c}\dot{\varphi}\right)+g_\varphi\left(d\dot{\varphi},d(\v(\mu_\varphi))\right)(\Lambda_\varphi\theta)\right)\omega^{[n]}_\varphi, 
\end{eqnarray*}
where the last equality follows from the proof of \cite[Lemma 4]{Lah19}. \\
Thus \eqref{(theta,ddc)} writes as
\begin{align*}
\int_Xg_\varphi\left(\theta,dd^c\dot{\varphi}\right)\dot{\psi}\uu(\mu_\varphi)\omega_{\varphi}^{[n]}
=&-\int_X g_\varphi\left(d\Lambda_\varphi\theta,d\dot{\varphi}\right)\dot{\psi}\uu(\mu_\varphi)\omega_{\varphi}^{[n]}-\int_X g_\varphi\left(\theta,d\dot{\psi}\wedge d^{c}\dot{\varphi}\right)\uu(\mu_\varphi)\omega^{[n]}_\varphi\\
&-\int_X\dot{\psi} g_\varphi\left(\theta,(d(\log \v(\mu_\varphi))\wedge d^{c}\dot{\varphi}\right) \v(\mu_\varphi) \omega^{[n]}_\varphi.
\end{align*}
Substituting back in \eqref{(F-v-w,psi)}, and noting that the $J$-invariance of $g_\varphi$ ensures $g_\varphi(\theta, d\dot{\psi}\wedge d^c \dot{\varphi})=g_\varphi(\theta, d\dot{\varphi}\wedge d^c \dot{\psi})$, we obtain \eqref{eq-(Fphi,psi)}. Using that $\dot{\varphi}, \dot{\psi}$ plays symmetric roles in \eqref{eq-(Fphi,psi)}, it follows that $\mathbb{H}^{\theta}_{\varphi,\v}$ is a self-adjoint operator. 

The last statement follows from \eqref{eq-(Fphi,psi)}. Indeed, if $\mathbb{H}^\theta_{\varphi, \v}(\dot{\varphi})=0$, taking $\dot{\psi}=\dot{\varphi}$ implies $\int_X |\nabla^\varphi \dot{\varphi}|^2_{\theta}\uu(\mu_\varphi)\omega_{\varphi}^{[n]}=0$. Thus $|\nabla^\varphi \dot{\varphi}|^2_{\theta}=0$, meaning that $\dot{\varphi}$ is a constant.
\end{proof}

 \begin{lemma}{\label{l:line:op}}
 The linearization $D_\varphi\mathcal{S}_{\theta}^{(r)}(\dot{\varphi})$ of $ \mathcal{S}_{\theta}^{(r)}$ in the direction $\dot{\varphi}$ at a potential $\varphi$ satisfying $\mathcal{S}_{\theta}^{(r)}(\varphi)=0$ is given by 

\begin{equation*}
D_\varphi\mathcal{S}_{\theta}^{(r)}(\dot{\varphi})=-2\mathbb{L}_{\varphi,\v}(\dot{\varphi})+r\mathbb{H}^{\theta}_{\varphi,\v}(\dot{\varphi}).
\end{equation*}
In particular, $D_\varphi\mathcal{S}_{\theta}^{(r)}$ is an elliptic and self-adjoint operator with respect to $\v(\mu_\varphi)\omega_\varphi^{[n]}$.
\end{lemma}

 \noindent We recall that $\mathbb{L}_{\varphi,\v}$ is the $\v$-weighted Lichnerowicz operator introduced in \cite{Lah19}, defined on $\T$-invariant smooth functions $f$ by

\begin{equation}{\label{define-lichne-pondéré}}
    \mathbb{L}_{\varphi,\v}(f):=\frac{1}{\v(\mu_{\varphi})}\bigg((\nabla^{\varphi,-}d)^{*_\varphi} \big(\v(\mu_{\varphi})(\nabla^{\varphi,-}df)\big) \bigg),
\end{equation}

\noindent where $\nabla^{\varphi,-}$ is introduced in \eqref{nabla:pm} and $(\nabla^{\varphi,-}d)^{*_{\varphi}}$ denotes its adjoint operator with respect to $\omega_\varphi$. 

\begin{proof}
We start recalling the following classical variational formulas  which will be useful in the proof below

\begin{equation*}
    D_{\varphi}\big(\Lambda_{\varphi}(\theta)\big)(\dot{\varphi})= -g_{\varphi}(\theta,dd^c\dot{\varphi}) \quad \text{and  } \quad D_{\varphi}(\mathrm{u}(\mu_\varphi))(\dot{\varphi})=g_{\varphi}(d\dot{\varphi},d\mathrm{u}(\mu_\varphi)),
\end{equation*}
for any weight $\mathrm{u} \in \mathcal{C}^{\infty}(P,\mathbb{R})$. We refer to \cite[Section 3.2]{Gau} for the first one. For the second, we add some details below for reader's convenience. Let $\varphi_t$ be a path in $\mathcal{K}(X,\omega_0)^\T$ and denote $\dot{\varphi}:= \frac{d}{dt}|_{t=0}(\varphi_t)$. Then

\begin{equation*}
\begin{split}
    D_{\varphi}(\mathrm{u}(\mu_\varphi))(\dot{\varphi}) &= \frac{d}{dt}|_{t=0}(\mathrm{u}(\mu_{\varphi_t})) =\sum_{a=1}^r \mathrm{u}_{,a}(\mu_\varphi) d^c\dot{\varphi}(\xi_a) \\
    &= g_{\varphi}( d\dot{\varphi}, d\mathrm{u}(\mu_\varphi)).
\end{split}    
\end{equation*}
The second equality is due to the fact $\mu_{\varphi_t}= \mu_0 + d^c\varphi_t$ and the third follows from the fact that the gradient of $\mu_\varphi^a$ is $-J\xi_a$ by definition of the moment map.\\
Now, by \cite[Lemma B.1]{Lah19} 

\begin{equation}{\label{comp:lichne}}
\begin{split}
D_{\varphi}\left(\frac{1}{\v(\mu_\varphi)}\big({\mathrm{Scal}}_{\v}(\omega_{\varphi})-\w(\mu_{\varphi})\big)\right)(\dot{\varphi})=&-2\mathbb{L}_{\varphi,\v}(\dot{\varphi}) -  g_{\varphi}\left(d\left(\frac{\w(\mu_{\varphi})}{\v(\mu_{\varphi})}\right),d\dot{\varphi}\right) \\ 
&+ g_{\varphi}\left( d \left( \frac{{\mathrm{Scal}}_{\v}(\omega_{{\varphi}})}{\v(\mu_{\varphi})}\right), d\dot{\varphi}  \right).
\end{split}
\end{equation}
Moreover 
\begin{equation*}
    \begin{split}
D_{\varphi}\bigg( &(\Lambda_{\varphi,\v}(\theta)-\tilde{\v}(\mu_\varphi)) \bigg)(\dot{\varphi})\\   
= & D_{\varphi}\big( \Lambda_{\varphi}(\theta) \big)(\dot{\varphi})+ D_{\varphi}\bigg(  \langle d\log(\v(\mu_\varphi)), \mu_\theta \rangle -  \tilde{\v}(\mu_\varphi)\bigg)(\dot{\varphi}) \\
=& -g_{\varphi}(\theta,dd^c\dot{\varphi}) -  g_{\varphi}(d\tilde{\v}(\mu_{\varphi})), d\dot{\varphi}) + \sum_{a=1}^r\mu_\theta^a g_{\varphi}\left(d\left(\frac{\v_{,a}(\mu_\varphi)}{\v(\mu_\varphi)}\right), d\dot{\varphi}\right) \\
=&- g_{\varphi}(\theta,dd^c\dot{\varphi}) - g_{\varphi}(d\tilde{\v}(\mu_{\varphi}), d\dot{\varphi}) - g(d\Lambda_{\varphi}(\theta),d\dot{\varphi}) + g(d\Lambda_{\varphi,\v}(\theta),d\dot{\varphi}) \\
&-  \sum_{a=1}^r \log(\v)_{,a}(\mu_\varphi)g_{\varphi}\left(d\mu^{a}_\theta,d\dot{\varphi}\right)\\
=& -g_{\varphi}(\theta,dd^c\dot{\varphi}) - g_{\varphi}(d\Lambda_{\varphi}(\theta),d\dot{\varphi}) - g_{\varphi}\big(\theta,d\log(\v(\mu_{\varphi})) \wedge d^c\dot{\varphi}\big)\\
&   + g_{\varphi}\left(d\left(\Lambda_{\varphi,\v}(\theta)-\tilde{\v}(\mu_{\varphi})\right), d\dot{\varphi}\right),\\
  \end{split}
\end{equation*}
where the third identity follows from 

\begin{equation*}
   d\Lambda_{\varphi, \v}(\theta)= d\Lambda_{\varphi}(\theta)+\sum_a d (\log(\v)_{,a}(\mu_\varphi)\mu_\theta^a),
\end{equation*}
while the last identity follows from the fact that

\[  \sum_{a=1}^r \frac{\uu_{,a}(\mu_\varphi)}{\uu(\mu_\varphi)}g_{\varphi}\left(d\mu^{a}_\theta,d\dot{\varphi}\right)=g_{\varphi}\big(\theta,d\log(\v(\mu_{\varphi})) \wedge d^c\dot{\varphi}\big),\]
(see e.g. \cite[Lemma 4]{Lah19}). Combining this with \eqref{comp:lichne} and using that $\varphi$ satisfy $S^{(r)}_\theta(\varphi)=0$ we get the desired identity. 
The last statement follows from Lemma \ref{lem-F-elliptic} and \eqref{define-lichne-pondéré}.
\end{proof}

As a consequence, standard arguments of elliptic theory allow us to conclude:

\begin{corollary}
The set $S_{1}$ is open.
\end{corollary}

\subsection{Existence of a solution for $t$ small}
The goal of this section is to show that $S_{1}$ is non-empty.
We recall that $r=\frac{(1-t)}{t}$.

\begin{theorem}\label{Thm-openess}
There exists $R:=R(\omega_0,\uu,\w)$  such that for any $r\geq R$ there exists $\varphi\in \mathcal{K}(X,\omega_0)^{\T}$ solving
\begin{equation}\label{eq-twisted-weighted}
\frac{1}{\v(\mu_\varphi)}\big({\mathrm{Scal}}_{\v}(\omega_{\varphi})-\w(\mu_{\varphi})\big) -r (\Lambda_{\varphi,\v}(\omega_0)-\tilde{\v}(\mu_\varphi))=0.
\end{equation}
\end{theorem}

The rest of the section is devoted to the proof of Theorem \ref{Thm-openess}, for which we use Hashimoto's arguments \cite{Hash} in the cscK case. The first step is to construct an approximate solution.

\subsubsection{Construction of an approximate solution} We now construct an approximate solution to \eqref{eq-twisted-weighted} for $r$ very big (or equivalently for $t$ very small):

\begin{lemma}\label{Lem-Appro-Sol}
 For each $j\in \mathbb{N}$, there exist $\varphi_1,\cdots\varphi_j\in C^\infty(X)^{\T}$ such that 
\begin{equation*}
\omega_j:=\omega_0+dd^c\left(\frac{\varphi_{1}}{r}+\cdots+\frac{\varphi_{j}}{r^{j}}\right),
\end{equation*}
satisfies 
\begin{equation*}
S^{(r)}_{\omega_0}(\omega_j)= r^{-j}f_{j,r},
\end{equation*}
for some $f_{j,r}\in C^\infty(X)^{\T}$ normalized by $\int_X f_{j,r}\v(\mu_{\omega_j})\omega_j^{[n]}=0$,  which is also bounded in $C^\infty(X)$ uniformly in $r$ (for $r>>1$).
\end{lemma}

\begin{proof}
Let $\omega_1:=\omega_0+dd^{c}(\varphi_1/r)$, for a certain $\varphi_1 \in \mathcal{C}^{\infty}(X)^\T$ that it will be suitably chosen in the following. Using Lemma \ref{l:line:op}, we expand $\mathcal{S}_{\omega_0}^{(r)}(\omega_1)$ in terms of $1/r$

\begin{align*}
\mathcal{S}_{\omega_0}^{(r)} (\omega_1)=&\mathcal{S}_{\omega_0}^{(r)} (\omega_0)+D_{\omega_0} \mathcal{S}_{\omega_0}^{(r)}\left(\frac{\varphi_1}{r}\right)+\cdots\\
=&\mathcal{S}_{\omega_0}^{(r)} (\omega_0)+\mathbb{H}_{\omega_{0},\v}^{\omega_0}(\varphi_1)+\mathcal{O}\left(\frac{1}{r}\right)\\
=&\frac{1}{\v(\mu_{\omega_0})}(\mathrm{Scal}_{\v}(\omega_0)-\w(\mu_{\omega_0}))+\mathbb{H}_{\omega_0,\v}^{\omega_0}(\varphi_1)+\mathcal{O}\left(\frac{1}{r}\right)\,
\end{align*}
where $\mathbb{H}_{\omega_0,\v}^{\omega_0}$ is the operator defined in \eqref{hash} for $\varphi=0$ and $\theta=\omega_0$. Also, observe that in the third identity we used that $ \Lambda_{\omega_0,\v} (\omega_0) -\tilde{\v}(\mu_{\omega_0})=0$. Using Lemma \ref{lem-F-elliptic} and standard elliptic theory, we can find $\varphi_1\in C^\infty(X)^{\T}$ such that
\begin{equation*}
\frac{1}{\v(\mu_{\omega_0})}(\mathrm{Scal}_{\v}(\omega_0)-\w(\mu_{\omega_0}))+\mathbb{H}_{\omega_0,\v}^{\omega_0}(\varphi_1)=0,\quad \int_X\varphi_1 \v(\mu_{\omega_0})\omega_0^{[n]}=0.
\end{equation*} 
Thus, we get $\omega_1:=\omega_0+dd^{c}(\varphi_1/r)$ such that
\begin{equation*}
\mathcal{S}^{(r)}_{\omega_0} (\omega_1)=r^{-1}f_{1,r}
\end{equation*} 
where $f_{1,r}$ is bounded in $C^\infty(X)^{\mathbb{T}}$ for all $r>0$, where $f_{1,r}$ are uniformly bounded in $C^\infty(X)^{\mathbb{T}}$ with respect to $r$ for $r\geq1$ and with zero mean value with respect to $\v(\mu_{\omega_1})\omega_1^{[n]}$. It is important to not that $\omega_1$ is still a K\"ahler form for $r$ large enough.
We now proceed by induction. Suppose we have constructed $\varphi_1,\cdots,\varphi_{j-1}$ such that for all $i=1,\cdots,j-1$ the $\T$-invariant K\"ahler metric
\begin{equation*}
\omega_i:=\omega_0+dd^c\left(\frac{\varphi_{1}}{r}+\cdots+\frac{\varphi_{i}}{r^{i}}\right),
\end{equation*}
satisfies
\begin{equation*}
\mathcal{S}^{(r)}_{\omega_0} (\omega_i)=r^{-i}f_{i,r}
\end{equation*}
where $f_{i,r}\in C^\infty(X)^{\T}$ are uniformly bounded in $C^\infty(X)$ with respect to $r$ for $r\geq1$  and $\int_X f_{i,r}\v(\mu_{\omega_i})\omega_{i}^{[n]}=0$. We prove the existence of the desired $\varphi_j\in C^\infty(X)^{\T}$. Denote $\omega_j=\omega_{j-1}+dd^c(r^{-j}\varphi_j)$. We have
\begin{align*}
\mathcal{S}^{(r)}_{\omega_0}(\omega_j)=& \mathcal{S}^{(r)}_{\omega_0}(\omega_{j-1})+ D_{\omega_{j-1}}\mathcal{S}^{(r)}_{\omega_0}(\frac{\varphi_{j}}{r^j})+\cdots\\
=& \mathcal{S}^{(r)}_{\omega_0}(\omega_{j-1})+\frac{1}{r^{j-1}}\mathbb{H}^{\omega_{0}}_{\omega_{j-1}}(\varphi_j)+\mathcal{O}\left(\frac{1}{r^{j}}\right)\\
=&\frac{1}{r^{j-1}}\left(f_{j-1,r}+\mathbb{H}^{\omega_{0}}_{\omega_{j-1}}(\varphi_j)\right)+\mathcal{O}\left(\frac{1}{r^j}\right).
\end{align*}
Using Lemma \ref{lem-F-elliptic} and standard elliptic theory, we can find the desired $\varphi_j\in C^\infty(X)^{\T}$. This completes the proof.
\end{proof}

In Lemma \ref{Lem-Appro-Sol} we have constructed a $\T$-invariant K\"ahler metric $\omega_j\in \alpha=[\omega_0]$ such that
\begin{equation*}
\frac{1}{\v(\mu_{\omega_j})}(\mathrm{Scal}_{\v}(\omega_j)-\w(\mu_{\omega_j})) -r (\Lambda_{\omega_j,\v}(\omega_0)-\tilde{\v}(\mu_{\omega_j}))= r^{-j}f_{j,r}
\end{equation*}
where $f_{j,r}\in C^\infty(X)^{\T}$ with $\int_X f_{j,r} \v(\mu_{\omega_j})\omega_j^{[n]} =0$. Since $\Delta_{\omega_j,\v}$ is elliptic and self-adjoint with respect to $\v(\mu_{\omega_j})\omega_j^{[n]}$ (see \cite[Appendix A]{DJL}), by elliptic theory we can find $G_{j,r}\in C^\infty(X)^{\T}$ such that
\begin{equation*}
\Delta_{\omega_j,\v}(G_{j,r})=f_{j,r}, \quad \int_XG_{j,r}\v(\mu_{\omega_j})\omega_{j}^{[n]}=0.
\end{equation*}
Also, for $p$ large enough, there exists a constant $C(p,\omega_j)$ such that
\[
\parallel G_{j,r}\parallel_{W^{2,p}(\omega_j^n)}\leq C(p,\omega_j)\parallel f_{j,r} \parallel_{W^{2,p-2}(\omega_j^n)}
\]
From Lemma \ref{Lem-Appro-Sol} we know that $\omega_j-\omega_0=O(r^{-1})$, hence 

\[\parallel G_{j,r}\parallel_{{W^{2,p}(\omega_0^n)}}\leq C(p,\omega_0)\parallel f_{j,r}\parallel_{{W^{2,p-2}(\omega_0^n)}}
\]
for a constant $C(p,\omega_0)$ depending only on $p,\omega_0$. It then follows that $\parallel G_{j,r}\parallel_{W^{2,p}(\omega_0^n)}$ can be bounded uniformly with respect to $r$ since $\parallel f_{j,r}\parallel_{{W^{2,p-2}(\omega_0^n)}}$ is bounded by a constant uniform relative to $r$. 
We consider 
\begin{equation}{\label{theta:j}}
    \theta_j:=\omega_0+r^{-(j+1)}dd^c G_{j,r},
\end{equation}
which is a K\"ahler form since $r$ is taken very large and $dd^c G_{j,r}$ is uniformly under control. We then get

\begin{equation}{\label{eq:twist_j}}
\begin{split}
\mathcal{S}^{(r)}_{\theta_j}(\omega_j)=&\frac{1}{\v(\mu_{\omega_j})}(\mathrm{Scal}_{\v}(\omega_j)-\w(\mu_{\omega_j})) -r\left(\Lambda_{\omega_j,\v}(\theta_j)-\tilde{\v}(\mu_{\omega_j})\right)\\
=&\left(\frac{1}{\v(\mu_{\omega_j})}(\mathrm{Scal}_{\v}(\omega_j)-\w(\mu_{\omega_j})) -r(\Lambda_{\omega_j,\v}(\omega_0)-\tilde{\v}(\mu_{\omega_j}))\right)-\frac{1}{r^{j}}\Delta_{\omega_j,\v}(G_{j,r})\\
=&\mathcal{S}^{(r)}_{\omega_0}(\omega_j)-\frac{1}{r^j}f_{j,r}\\
=&0.
\end{split}
\end{equation}

\subsubsection{Applying the implicit function theorem}{\label{s:IFT}} 

Let $\omega_j$ be the metric constructed in Lemma \ref{Lem-Appro-Sol}. We now recall a quantitative version of the implicit function theorem \cite[Theorem 5.3]{Fin}, which will be used to perturb $\omega_j$ to a genuine solution of \eqref{eq-twisted-weighted}.

\begin{theorem}\label{Quant-inv-thm}
Suppose that $B_1$, $B_2$ are Banach spaces,
$U\subset B_1$ is an open set containing the origin, and that
\begin{enumerate}
\item $T : U \to B_2$ is a differentiable map whose derivative at $0$, $D_0T$, is an isomorphism of
Banach spaces, with inverse $P$;
\item $\delta^\prime$ is the radius of the closed ball in $B_1$, centered at 0, on which $T-D_0T$ is Lipschitz, with constant $1/(2\parallel P\parallel_{\rm op})$;
\item $\delta:=\delta^\prime/(2\parallel P\parallel_{\rm op})$.
\end{enumerate}
Then whenever $y \in B_2$ satisfies $\parallel y-T (0)\parallel < \delta$, there exists $x \in U$ such that $T(x)=y$. Moreover, such an $x$ is unique subject to the constraint $\parallel x\parallel <\delta^\prime$.
\end{theorem}

We denote by $W^{2,k}(\omega_j^n)^{\T}_{0}$, the $\T$-invariant Sobolev completion of the space of smooth $\T$-invariant functions normalized such that $\int_X \varphi \,\v(\mu_{\omega_j})\omega_j^{[n]}=0$; we also let $\Omega^{1,1}(X)^{\T}$ denote the Sobolev completion w.r.t. the $W^{2,k}(\omega_j^n)$-norm of the space of real $(1,1)$-forms.\\
We apply Theorem \ref{Quant-inv-thm} with the Banach spaces $B_1:=\Omega^{1,1}(X)^{\T}\times W^{2,k+4}(\omega_j^n)^{\T}_{0}$, $B_2:=\Omega^{1,1}(X)^{\T}\times W^{2,k}(\omega_j^n)^{\T}_{0}$, and the map $T^{(r)}_{\v,\w}: U\to B_2$ defined on the open set $$U:=\{(\chi,\varphi)\in B_1\text{ }| \text{ }\omega_{j,\varphi}:=\omega_{j}+dd^c\varphi>0 \}$$ by the expression
\begin{align*}
\begin{split}
T^{(r)}(\chi,\varphi):=&\left(\theta_j+\chi,\mathcal{S}_{\theta_j+\chi}^{(r)}(\omega_{j,\varphi})\right)\\
=&\left(\theta_j+\chi,\frac{1}{\v(\mu_{\omega_{j,\varphi}})}(\mathrm{Scal}_{\v}(\omega_{j,\varphi})-\w(\mu_{\omega_{j,\varphi}}))-r\left(\Lambda_{\omega_{j,\varphi},\v}(\theta_j+\chi)-\tilde{\v}(\mu_{\omega_{j,\varphi}}) \right)\right).
\end{split}
\end{align*}
We identify the origin $0\in B_1$ with the base point $(\theta_j,\omega_j)$. Using \eqref{eq:twist_j} and Lemma \ref{l:line:op}, the linearization of $T^{(r)}$ at $0$ is given by
\begin{equation*}
(D_0T^{(r)})(\dot{\chi},\dot{\varphi})=\begin{pmatrix} 
1 & 0 \\
-r{\Lambda_{\omega_{j},\v}} &  -2 \mathbb{L}_{j,\v}+r\mathbb{H}_{j,\v}
\end{pmatrix} \begin{pmatrix} 
\dot{\chi}  \\
\dot{\varphi} 
\end{pmatrix},
\end{equation*}
where for convenience we note

\begin{equation}{\label{not}}
    \mu_j:=\mu_{\omega_j}\text{, } \Lambda_{j,\v}:=\Lambda_{\omega_j,\v}\text{, } \mathbb{L}_{j,\v}:=\mathbb{L}_{\omega_j,\v}\text{, } \text{and } \mathbb{H}_{j,\v}:=\mathbb{H}_{\omega_j,\v}^{\theta_j}.
\end{equation}
By Lemma \ref{l:line:op} we know that  $-2 \mathbb{L}_{j,\v}+r\mathbb{H}_{j,\v}$ is elliptic and self-adjoint with respect to $\v(\mu_j)\omega_j^{[n]}$. 
Thus, the kernel of $-2 \mathbb{L}_{j,\v}+r\mathbb{H}_{j,\v}$ is zero, which implies surjectivity (by the Fredholm alternative). It follows that $D_0T^{(r)}:B_1\to B_2$ is an isomorphism with inverse $P^{(r)}$ given by
\begin{equation*}
P^{(r)}=\begin{pmatrix} 
1 & 0 \\
\left(-2 \mathbb{L}_{j,\v}+r\mathbb{H}_{j,\v}\right)^{-1}\circ \left(-r\Lambda_{j,\v}\right) &  \left(-2 \mathbb{L}_{j,\v}+r\mathbb{H}_{j,\v}\right)^{-1} 
\end{pmatrix}.
\end{equation*}
We now want to compute the operator norm $\parallel P^{(r)}\parallel_{\rm op}$. We start by computing the operator norm of $\left(-2 \mathbb{L}_{j,\v}+r\mathbb{H}_{j,\v}\right)^{-1}$. By definition all eigenvalues of $\left(-2 \mathbb{L}_{j,\v}+r\mathbb{H}_{j,\v}\right)$ are negative. We denote by $\lambda_{1,j}<0$ the largest non-zero eigenvalue.
\begin{lemma}
 Then there exists a constant $C=C(\v,\omega_0)>0$ such that $\lambda_{1,j}<-Cr$ for $r$ large enough.
\end{lemma}

\begin{proof}
 Let $\varphi_{1,j}$ be an eigenfunction corresponding to $\lambda_{1,j}$. By Lemma \ref{lem-F-elliptic} and the definition of $\mathbb{L}_{j,\v}$ in \eqref{define-lichne-pondéré}, we have
\begin{align*}
\lambda_{1,j}=&\frac{\int_X\varphi_{1,j}\left(-2 \mathbb{L}_{j,\v}+r\mathbb{H}_{j,\v}\right)(\varphi_{1,j})\v(\mu_j)\omega_j^{[n]}}{\int_X |\varphi_{1,j}|^2\v(\mu_j)\omega_j^{[n]}}\\
=&\frac{-1}{\int_X|\varphi_{1,j}|^2 \v(\mu_j)\omega_j^{[n]}}\left(2\int_X |\nabla^{j,-}d\varphi_{1,j}|_{\omega_j}^2\v(\mu_{j})\omega_{j}^{[n]}+\int_X r|\nabla^j\varphi_{1,j}|_{\theta_j}^2\v(\mu_{j})\omega_{j}^{[n]}\right)\\
\leq&\frac{-r}{\int_X |\varphi_{1,j}|^2 \v(\mu_j)\omega_j^{[n]}}\int_X |\nabla^j\varphi_{1,j}|_{\theta_j}^2\v(\mu_{j})\omega_{j}^{[n]}\\
\leq& \frac{-C_{\v} r}{\int_X|\varphi_{1,j}|^2 \omega_j^{[n]}}\int_X |\nabla^j\varphi_{1,j}|_{\theta_j}^2\omega_{j}^{[n]},
\end{align*}
where $\nabla^j:=\nabla^{\omega_j}$ and $C_{\v}>0$ is a constant depending only on the $C^0$-bound for $\v$. Since $\omega_j-\omega_0=\mathcal{O}(r^{-1})$ and $\theta_j-\omega_0=\mathcal{O}(r^{-(j+1)})$, see \eqref{theta:j}, there is a constant $C_1=C_1(\omega_0)>0$ such that

\begin{equation*}
C_1\int_X |\nabla^j \varphi_{1,j}|_{\omega_j}^2\omega_{j}^{[n]}\leq\int_X |\nabla^j \varphi_{1,j}|_{\theta_j}^2\omega_{j}^{[n]}.
\end{equation*}
Hence,
\begin{equation*}
\lambda_{1,j}\leq \frac{-rC_{\v}C_1 }{\int_X|\varphi_{1,j}|^2 \omega_j^{[n]}}\int_X | \nabla^j\varphi_{1,j}|_{\omega_j}^2\omega_{j}^{[n]}.
\end{equation*} 
By the Poincar{\'e} inequality we have
\begin{equation*}
\int_X|\varphi_{1,j}|^2 \omega_j^{[n]}\leq C_3(\omega_j)\int_X|\nabla^j \varphi_{1,j}|_{\omega_j}^2\omega_{j}^{[n]},
\end{equation*}
and since $\omega_j-\omega_0=\mathcal{O}(r^{-1})$, we can find a constant $C_4=C_4(\omega_0)>0$ such that $C_4>C_3$ uniformly of all large enough $r>0$. It follows that $\lambda_{1,j}\leq -\left(\frac{C_{\v} C_1 }{C_4}\right)r$.
\end{proof}

Now, we compute the operator norm of $P^{(r)}$. To simplify notations we denote $K^{(r)}=\left(-2 \mathbb{L}_{j,\v}+r\mathbb{H}_{j,\v}\right)^{-1}$. Observe that $\lambda_{1,j}^{-1}<0$ is the smallest eigenvalue of $K^{(r)}$. In particular all eigenvalues of $K^{(r)}$ are dominated by $|\lambda_{1,j}^{-1}|$.\\
By the proof of the fundamental elliptic estimates \cite[Theorem 11.1 in Chapter 5]{Taylor23} and the expansions $\omega_j-\omega_0=\mathcal{O}(r^{-1})$, $\theta_j-\omega_0=\mathcal{O}(r^{-(j+1)})$, we infer that
there exists a constant $C=C(\omega_0,k)$ independent of $r$ such that

\begin{equation*}
\parallel K^{(r)}(\dot{\varphi})\parallel_{W^{2,k+4}(\omega_0^n)}\leq r C(\omega_0,k)(\parallel \dot{\varphi}\parallel_{W^{2,k}(\omega_0^n)}+\parallel K^{(r)}(\dot{\varphi})\parallel_{L^2(\omega_0^n)}).
\end{equation*}
 We compute
\begin{align*}
\parallel P^{(r)}\parallel_{\rm op}=&\sup_{\dot{\chi},\dot{\varphi}}\frac{\parallel \dot{\chi}\parallel_{W^{2,k}(\omega_j^n)}+\parallel -rK^{(r)}(\Lambda_{j,\v}(\dot{\chi}))+K^{(r)}(\dot{\varphi})\parallel_{W^{2,k+4}(\omega_j^n)}}{\parallel\dot{\chi}\parallel_{W^{2,k}(\omega_j^n)}+\parallel \dot{\varphi}\parallel_{W^{2,k}(\omega_j^n)}}\\
\leq &\sup_{\dot{\chi},\dot{\varphi}}\frac{\parallel\dot{\chi}\parallel_{W^{2,k}(\omega_j^n)}+\parallel rK^{(r)}\Lambda_{j,\v}(\dot{\chi}) \parallel_{W^{2,k+4}(\omega_j^n)}+\parallel K^{(r)}(\dot{\varphi})\parallel_{W^{2,k+4}(\omega_j^n)}}{\parallel\dot{\chi}\parallel_{W^{2,k}(\omega_j^n)}+\parallel\dot{\varphi}\parallel_{W^{2,k}(\omega_j^n)}}\\
\leq& 1+ \sup_{\dot{\chi},\dot{\varphi}}\frac{ \parallel rK^{(r)}\Lambda_{j,\v}(\dot{\chi}) \parallel_{W^{2,k+4}(\omega_0^n)}+\parallel K^{(r)}(\dot{\varphi})\parallel_{W^{2,k+4}(\omega_0^n)}}{\parallel\dot{\chi}\parallel_{W^{2,k}(\omega_0^n)}+\parallel\dot{\varphi}\parallel_{W^{2,k}(\omega_0^n)}}\quad \text{(we use }\omega_j=\omega+\mathcal{O}(r^{-1}))\\
\leq & 1+rC(1+|\lambda_{1,j}|^{-1})\sup_{\dot{\chi},\dot{\varphi}}\frac{r\parallel \Lambda_{j,\v}(\dot{\chi})\parallel_{W^{2,k}(\omega_0^n)}+\parallel \dot{\varphi}\parallel_{W^{2,k}(\omega_0^n)}}{\parallel\dot{\chi}\parallel_{W^{2,k}(\omega_0^n)}+\parallel\dot{\varphi}\parallel_{W^{2,k}(\omega_0^n)}}\\
\leq& 1+C r^2(1+(Cr)^{-1})(1+\parallel \Lambda_{\omega_0,\v}\parallel_{\rm op}),
\end{align*}
where we use $\omega_j-\omega_0=\mathcal{O}(r^{-1})$ in the last inequality. Thus, we obtain
\[
\parallel P^{(r)}\parallel_{\rm op}< Cr^2
\]
 where $C$ is a uniform constant for $r$. 
 
 We claim that, for $\nu\gg 1$, the operator $T^{(r)}-D_0T^{(r)}$ is Lipschitz on the ball centered at $0\in B_1$ of radius $\delta^\prime:=r^{-\nu}$ with constant $1/(2 \parallel P^{(r)}\parallel_{\rm op})$. Indeed, applying the mean value theorem to the 4th order operator $N^{(r)}:=T^{(r)}-D_0T^{(r)}$ we get 
 \begin{equation*}
\parallel N^{(r)}(\chi_1,\varphi_1)-N^{(r)}(\chi_0,\varphi_0)\parallel_{W^{2,k}(\omega_j^{n})}\leq \underset{t\in [0,1]}{\sup}\parallel D_{(\chi_t,\varphi_t)}N^{(r)}\parallel_{\rm op} \parallel (\chi_1-\chi_0,\varphi_1-\varphi_0)\parallel_{\Omega^{1,1}\times W^{2,k+4}(\omega_j^{n})}
 \end{equation*}
 where $(\chi_t,\varphi_t)=(t\chi_1+(1-t)\chi_0,t\varphi_1+(1-t)\varphi_0)$. We then observe that
 $$  D_{(\chi_t,\varphi_t)} N^{(r)}=  D_{(\chi_t,\varphi_t)}T^{(r)}-D_0T^{(r)}, $$

 $$
 \parallel (\chi_t ,\varphi_t)\parallel_{\Omega^{1,1}\times W^{2,k}(\omega_j^{n}) }\leq \parallel (\chi_1 ,\varphi_1)\parallel_{\Omega^{1,1}\times W^{2,k}(\omega_j^{n})}+ \parallel (\chi_0 ,\varphi_0)\parallel_{\Omega^{1,1}\times W^{2,k}(\omega_j^{n})}
 $$
 and that
 \[
      \parallel D_{(\chi_t,\varphi_t)}N^{(r)}\parallel_{\rm op}=\parallel  
       D_{(\chi_t,\varphi_t)}T^{(r)}-D_0T^{(r)}
      \parallel_{\rm op}\leq B  \parallel (\chi_t ,\varphi_t)\parallel_{\Omega^{1,1}\times W^{2,k+4}(\omega_j^{n})}
 \]
where $B$ is a uniform constant. It follows that
 \[
 \begin{split}
&\parallel N^{(r)}(\chi_1,\varphi_1)-N^{(r)}(\chi_0,\varphi_0)\parallel_{W^{2,k+4}(\omega_j^{n})}   \\
&\leq B \left( \parallel (\chi_1 ,\varphi_1)\parallel_{\Omega^{1,1}\times W^{2,k+4}(\omega_j^{n})}+ \parallel (\chi_0 ,\varphi_0)\parallel_{\Omega^{1,1}\times W^{2,k+4}(\omega_j^{n})}\right) \parallel (\chi_1-\chi_0,\varphi_1-\varphi_0)\parallel_{\Omega^{1,1}\times W^{2,k+4}(\omega_j^{n})}
 \end{split}
 \]
Choosing a ball $B_1$ around 0 of radius $\delta':=r^{-\nu}$ for some $\nu\gg 1$ the operator $(N^{(r)})_{\mid B_1}$ is then Lipschitz with constant $1/(2 \parallel P^{(r)}\parallel_{\rm op})$. 
Taking  
 
 \[\delta:=\frac{\delta^\prime}{2\parallel P^{(r)}\parallel_{\rm op}}>\frac{1}{2C}r^{-\nu-1},\] the inverse function (Theorem \ref{Quant-inv-thm}), holds in a ball in $B_2$ of radius $\delta$ centered at $T^{(r)}(0)$ (recall that we identified the origin $0$ in $B_1$ with $(\theta_j,\omega_j)$) given by
$$T^{(r)}(0)=(\theta_j,0).$$  
Using $\theta_j-\omega_0=\mathcal{O}(r^{-j-1})$, we obtain
\begin{equation*}
\parallel T^{(r)}(0)-(\omega_0,0)\parallel_{W^{2,k}(\omega_j)}\leq \tilde{C}r^{-j-1}<\delta,
\end{equation*}
for $j\gg \nu$ big enough. It follows that, for all large enough $r>0$, there exist $(\chi,\varphi)\in U\subset B_1$ such that $T^{(r)}(\chi,\varphi)=(\omega_0, 0)$, i.e.
\begin{equation*}
\begin{cases} \theta_j+\chi=\omega_0,\\ 
\frac{1}{\v(\mu_{\omega_{j,\varphi}})}(\mathrm{Scal}_{\v}(\omega_{j,\varphi}))-\w(\mu_{\omega_{j,\varphi}})-r\left(\Lambda_{\omega_{j,\varphi},\v}(\theta_j+\chi)-\tilde{\v}(\mu_{\omega_{j,\varphi}})\right)=0.
\end{cases}
\end{equation*} 
We conclude that $\varphi\in C^\infty(X)^{\T}$ using a bootstrapping argument.

\section{Closeness}{\label{s:close}}

\subsection{Weighted Aubin-Mabuchi functionals}

Following \cite{AJL}, 
we let ${\mathbf{E}}_{\v}$ denote the functional on $\mathcal{K}(X, \omega_0)^\T$,  defined by
\begin{equation}{\label{def:e}}
    (D_{\varphi} \mathbf{E}_\v)(\dot \varphi) = \int_X \dot \varphi  \v(\mu_{\varphi})\omega_{\varphi}^{[n]}, \qquad {\mathbf{E}}_\v(0)=0,
\end{equation}  
where $D_\varphi$ denotes the differential of the Fr\'echet manifold $\mathcal{K}(X,\omega_0)^\T$ at the point $\varphi$, see e.g.  \cite[Lemma 2.14]{BW}.
We define the space of normalized potentials

\begin{equation*}
   \Kn:=\K \cap {\mathbf{E}}^{-1}(0),
\end{equation*}
where $\mathbf{E}:=\mathbf{E}_1$. Also we introduce 

 $${J}_{\v}(\varphi):= \int_X \varphi \v(\mu_{0})\omega_0^{[n]} - \mathbf{E}_\v(\varphi)$$ and 
\begin{equation}{\label{I:v:fonc}}
    \begin{split}
    I_\v(\varphi):=\int_X \varphi \v(\mu_0)\omega_0^{[n]}- \int_X \varphi\v(\mu_\varphi) \omega_\varphi^{[n]}.
    \end{split}
\end{equation}

\begin{lemma}{\label{l:au:ma}}
The weighted Aubin-Mabuchi functionals satisfy the following relation

\begin{equation*}
 0\leq  \frac{1}{n+1} I_\v(\varphi) \leq J_\v(\varphi) \leq \frac{n}{n+1} I_\v(\varphi).
\end{equation*}

\end{lemma}

\begin{proof}
The proof is based on  Bernstein polynomial approximation and the semisimple principal construction of \cite{AJL}. \\
 Let $Y$ be a semi-simple principal $(X,\T)$-fibration over a basis $B^m$ such that the volume form $\tilde{\omega}_0^{[n+m]}$ of a compatible K\"ahler metric $\tilde{\omega}_0$ on $Y$ associate to a K\"ahler metric $\omega_0$ on $X$ (see \cite[Definition 5.3]{AJL}) restricts to the volume

\begin{equation}{\label{vol:comp}}
   \tilde{\omega}_0^{[m+n]}|_{X}= \mathrm{p}(\mu_{\omega_0})\omega_0^{[n]},
\end{equation}
on each fibers $X$, where $\mathrm{p}(x)$ is a weight of the form $\mathrm{p}(x):=\prod_{a=1}^{k}(\langle p_a,x\rangle+c_a)^{m_a}$. We observe that such fibration always exists as explained in \cite[Section 5]{AJL}. \\
Let $\tilde{\omega}_0$ be the compatible K\"ahler metric on $Y$ associated to $\omega_0$ on $X$. For any $\tilde{\varphi} \in \mathcal{K}(Y,\tilde{\omega}_0)^\T$, it is known that \cite[Proposition 4.2.1]{Gau}   we have 

\begin{equation*}
 0\leq  \frac{1}{n+1} I_1^Y(\tilde{\varphi}) \leq J^Y_1(\tilde{\varphi}) \leq \frac{n}{n+1} I_1^Y(\tilde{\varphi})
\end{equation*}
where $I_1^Y, J^Y_1$ are the Aubin-Mabuchi energies on $Y$.

On the other hands, by (37) we have that for any $\varphi \in \mathcal{K}(X,\omega_0)^\T \subset \mathcal{K}(Y,\tilde{\omega}_0)^\T$ 
\begin{equation*}
    I_1^Y(\varphi)= \mathrm{Vol}(B)I^X_{\mathrm{p}}(\varphi) \quad \text{ and }  J_1^Y(\varphi)= \mathrm{Vol}(B)J^X_{\mathrm{p}}(\varphi).
\end{equation*}
Then we get for any $\varphi \in \mathcal{K}(X,\omega_0)^\T \subset \mathcal{K}(Y,\tilde{\omega}_0)^\T$ 

\begin{equation*}
 0\leq  \frac{1}{n+1} I_\mathrm{p}^X(\varphi) \leq J^X_\mathrm{p}(\varphi) \leq \frac{n}{n+1} I_\mathrm{p}^X(\varphi),
\end{equation*}
where $\mathrm{p}$ is the polynomial weight associated to the fibration. Now, using the Bernstein polynomial approximation (see e.g. \cite{HS33}), we approximate \(\v\) by a sequence of polynomials $(P_k)_{k\geq1}$, where each $P_k$ is a linear combination of weights of the form of \(\mathrm{p}\) with positive coefficients. By linearity relative to the weight function we get
  \[
        0 \leq \frac{1}{n+1} I_{P_k}^X(\varphi) \leq J_{P_k}^X(\varphi) \leq \frac{n}{n+1} I_{P_k}^X(\varphi).
        \]    
The conclusions follows by passing to the limit in the above inequalities.        
\end{proof}
Let $(\varphi_t)_{t\geq0}$ be a path in $\mathcal{K}(X,\omega_0)^\T$ satisfying $\varphi_0=\varphi$ and $\frac{d \varphi_t}{dt}|_{t=0}=\dot{\varphi}$. A direct computation shows that 

\begin{eqnarray*}
    \frac{d}{dt}\left(\v(\mu_{\varphi_t})\omega_{\varphi_t}^{[n]}\right)\big|_{t=0}
    &=& \left(\sum_{a=1}^r\v_{, a}(\mu_{\varphi})  \frac{d}{dt}(\mu^a_{\varphi_t}) + \Delta_{\varphi_t}\dot{\varphi_t}\,\right) \omega_{\varphi_t}^{[n]}\big|_{t=0}\\
    &=& \left(\sum_{a=1}^r\v_{, a}(\mu_{\varphi})  \frac{d}{dt}(\mu_0^a+d^c\varphi_t(\xi_a)) + \Delta_{\varphi_t}\dot{\varphi_t}\,\right) \omega_{\varphi_t}^{[n]}\big|_{t=0}\\
    &=& \v(\mu_{\varphi})\Delta_{\varphi,\v}(\dot{\varphi})\, \omega_{\varphi}^{[n]}.
\end{eqnarray*}

By definition of $I_\v$ and $J_\v$ 
\begin{equation*}
    I_\v(\varphi) - J_\v(\varphi)= \mathbf{E}_\v(\varphi) - \int_X \varphi\v(\mu_\varphi) \omega_\varphi^{[n]}.
\end{equation*}
Thus, the definition of $\mathbf{E}$ \eqref{def:e} leads to

\begin{eqnarray}{\label{i-j}}
 \nonumber   D_\varphi(I_\v-J_\v)(\dot{\varphi})&=&  \int_X \dot \varphi  \v(\mu_{\varphi})\omega_{\varphi}^{[n]}- \frac{d}{dt} \left( \int_X \varphi\v(\mu_\varphi) \omega_\varphi^{[n]} \right)\big|_{t=0}\\
      &=& -\int_X \varphi \Delta_{\varphi,\v}(\dot{\varphi})\v(\mu_\varphi)\omega_\varphi^{[n]}.
\end{eqnarray}
Also, for any $(1,1)$-form $\theta$, we define $\J^{\theta}_{\v,\w}: \mathcal{K}(X,\omega_0)^\T \longrightarrow \R$  by 
\begin{equation}{\label{J-theta-fonc}}
(D_\varphi\J^{\theta}_{\v,\w})(\dot{\varphi}):=\int_X \dot{\varphi}\left(\Lambda_{\varphi,\v}(\theta) -\frac{\w(\mu_\varphi)}{\v(\mu_\varphi)}\right)\v(\mu_\varphi)\omega_{\varphi}^{[n]} ,\quad \J^{\theta}_{\v,\w}(0)=0,
\end{equation}
which is well-defined by \cite[Lemma 4]{Lah19} and \cite[Lemma 2.14]{BW}. We deduce the following key Lemma:

\begin{lemma}{\label{l:j:coercive}}
The functional $\J^{\omega_0}_{\v,\v\tilde{\v}}(\varphi) : \Kn \longrightarrow \R $ is coercive, i.e. there exists positive constants $A$ and $B$ such that 

\begin{equation*}
    \J^{\omega_0}_{\v,\v\tilde{\v}}(\varphi)  \geq A d_1(0,\varphi) -B,
\end{equation*}
where $d_1$ is the Darvas distance on the space of relative potential $\mathcal{K}(X,\omega_0)^\T$ \cite{Dar}, and where $\tilde{\v}$  is defined in \eqref{eq:tilde-v}.
\end{lemma}
It is worth noticing that the above Lemma is true only for the particular choice of $\theta=\omega_0$ and the weight $\w=\v\tilde{\v}$.

\begin{proof}

For any $(1,1)$-form $\rho$ with moment map, let $\J_{\v}^{\rho} : \mathcal{K}(X,\omega_0)^\T \longrightarrow \R$ with variation

\begin{equation}{\label{j-v}}
   (D_\varphi\J^{\rho}_{\v})(\dot{\varphi}):=\int_X \dot{\varphi}\Lambda_{\varphi,\v}(\rho)\v(\mu_\varphi)\omega_{\varphi}^{[n]} ,\quad \J^{\rho}_{\v}(0)=0,
\end{equation}
which is well-defined thanks to \cite[Lemma 4]{Lah19}. Hence $\J^{\omega_0}_{\v,\v\tilde{\v}}= \J^{\omega_0}_{\v} - \mathbf{E}_{\v\tilde{\v}}$. The last identity comes from the fact that $D_\varphi (\J^{\omega_0}_{\v,\v\tilde{\v}}) =D_\varphi (\J^{\omega_0}_{\v} - \mathbf{E}_{\v\tilde{\v}})$ and that all the functionals are normalized to be $0$ at $0$.\\
Using the definition of $\Lambda_{\varphi,\v}$, of $\Delta_{\varphi, \v}$ and the fact that $\mu_\varphi= \mu_0 +d^c \varphi$, we obtain 
  \[
   \begin{split}
       D_\varphi\J_{\v}^{\omega_0}(\dot\varphi)=&\int_X\dot{\varphi}\left(     \Lambda_\varphi (\omega_0)\v(\mu_\varphi)+\langle d\log \v(\mu_\varphi),\mu_0\rangle \v(\mu_\varphi) \right) \omega_\varphi^{[n]}\\
       =&\int_X\dot{\varphi}\left( (n-\Delta_\varphi \varphi)\v(\mu_\varphi)+\langle d\log\v(\mu_\varphi),\mu_\varphi-d^c\varphi\rangle)\v(\mu_\varphi) \right)\omega_\varphi^{[n]}\\
       =&\int_X\dot{\varphi}(n+\langle d\log\v(\mu_\varphi),\mu_\varphi\rangle)\v(\mu_\varphi)\omega_\varphi^{[n]} -\int_X\dot{\varphi}\Delta_{\varphi, \v}(\varphi)\v(\mu_\varphi)\omega_\varphi^{[n]}\\
       =&D_\varphi\mathbf{E}_{\v\tilde{\v}}(\dot\varphi)-\int_X\varphi\Delta_{\varphi,\v}(\dot\varphi)\v(\mu_\varphi)\omega_\varphi^{[n]}\\
       =&D_\varphi\mathbf{E}_{\v\tilde{\v}}(\dot\varphi)+ D_\varphi(I_\v-J_\v)(\dot{\varphi}),
   \end{split}
   \]
   where the last equality follows from \eqref{i-j}. The normalization $\J_\v^{\omega_0}(0)=\mathbf{E}_{\v\tilde{\v}}(0)=I_\v(0)=J_\v(0)=0$ yields
  
 \begin{equation*}
 \begin{split}
  \J_{\v,\v\tilde{\v}}^{\omega_0}(\varphi)=\J_\v^{\omega_0}(\varphi)-\mathbf{E}_{\v\tilde{\v}}(\varphi) = I_\v(\varphi)-J_\v(\varphi)\geq \frac{1}{n}J_\v(\varphi),
 \end{split}
\end{equation*} 
where for the last inequality we use Lemma \ref{l:au:ma}. Moreover by \cite[Lemma 6.4]{AJL}, $J_\v(\varphi) \geq \inf \v\, J_1(\varphi)$. We then conclude since $J_1$ and $d_1$ on $\Kn$ are comparable, as proved in \cite[Proposition 5.5]{DR}.
\end{proof}

We end this section with a lemma which will be very useful in the following:

\begin{lemma}\label{J and d1}
 For any $(1,1)$-form $\rho$ with moment map, any $\v>0$ and $\w$, we have that
 
\begin{equation*}
    |\J^{\rho}_{\v,\w}(\varphi)| \leq A d_1(0, \varphi) + B,
\end{equation*}
where $A$ and $B$ are positive constants.

\end{lemma}

\begin{proof}
By definition of $\J^{\rho}_{\v,\w}$ \eqref{J-theta-fonc}, we have the following decomposition

\begin{equation*}
    \J^{\rho}_{\v,\w}= \J^{\rho}_{\v} - \mathbf{E}_\w,
\end{equation*}
where $\J^{\rho}_{\v}$ is defined in \eqref{j-v} and $\mathbf{E}_\w$ in \eqref{I:v:fonc}. By \cite[Lemma 6.5 and eq. (55)]{AJL}, both $\J^{\rho}_{\v}$ and $\mathbf{E}_\w$ are controlled by $J(\varphi)$ and $\| \varphi \|_{L^1(\omega_0)}$. Hence, we conclude by \cite[Propositions 5.4 and 5.5]{DR}.
\end{proof}

\subsection{Closeness for time less than one}
In this section, we show that $S_{t_1}$ (defined in \eqref{set-solution}) is closed for any $t_1<1$ under the assumption that the weighted Mabuchi energy $\M_{\v,\w}$ is bounded from below. 

The weighted Mabuchi energy was introduced in \cite{Lah19}  and is defined by its variation 

\begin{equation*}
    D_{\varphi}\left(\mathbf{M}_{\v,\w}\right)(\dot{\varphi})=-\int_X \dot{\varphi}\big(\mathrm{Scal}_\v(\omega_\varphi)-\w(\mu_\varphi)\big)\v(\mu_\varphi)\omega_\varphi^{[n]}, \quad \mathbf{M}_{\v,\w}(0)=0.
\end{equation*}
The functional $\M_{\v, \w}$ admits a weighted Chen-Tian decomposition \cite[Theorem 5]{Lah19}:
\begin{align}
\begin{split}\label{Chen-Tian}
\M_{\v,\w}(\varphi)=\Ent_\v(\varphi)  -\J^{2\mathrm{Ric}(\omega_0)}_{\v,\w}(\varphi) - \int_X \log\v(\mu_0)\v(\mu_0)\omega_0^{[n]},
\end{split}
\end{align}
where $\Ent_\v(\varphi):= \int_X \log\left( \frac{\v(\mu_\varphi)\omega_\varphi^n}{\omega_0^n}\right)\v(\mu_\varphi)\omega_\varphi^{[n]}$ and $\J_{\v, \w}^{2\mathrm{Ric}(\omega_0)}$ is defined in \eqref{J-theta-fonc}.

\begin{prop}{\label{p:closed:t1}}
Suppose that the weighted Mabuchi functional $\M_{\v,\w}$ is bounded from below on $\mathcal{K}(X,\omega_0)^{\mathbb{T}}$, then $S_{t_1}$ is closed for any $t_1<1$.
\end{prop}

\begin{proof}
Suppose that there exists a sequence of solutions $\{\varphi_i\}_{i \in \mathbb{N}} \subset \mathcal{K}(X,\omega_0)^{\mathbb{T}}$ of (\ref{continuity-path-weithed}) at time $t_i$, with $t_i \rightarrow t_1 < 1$. We want to show that $t_1\in S$. To simply notations we drop the index $i$.\\
By definition, the twisted weighted Mabuchi functional

\begin{equation*}
    \M_{\v,\w}^{t}:= t \M_{\v,\w} + (1-t)\J^{\omega_0}_{\v,\v\tilde{\v}} 
\end{equation*}
is the Euler-Lagrange functional of the PDE \eqref{continuity-path-weithed}. Hence, a solution $\varphi_t$ of  \eqref{continuity-path-weithed} is a critical point $\M_{\v,\w}^{t}$.
Since $ \M_{\v,\w}$ and $\J^{\omega_0}_{\v,\v\tilde{\v}}$ are convex (in time) along weak geodesic segment (see \cite[Lemma 7, Corollary 4 and Theorem 4]{Lah20}), $\varphi_t$ is a global minima of $ \M_{\v,\w}^{t}$ on $\mathcal{K}(X,\omega_0)^\T$. In particular, for any solution $\varphi_t$ of \eqref{continuity-path-weithed} we have that

\begin{equation}{\label{bound:twi:mab}}
\M_{\v,\w}^{t}(\varphi_t)  \leq \M^t_{\v,\w}(0) \leq C,
\end{equation}
where $C$ does not depend on $t$. Thanks to the assumption, the above and Lemma \ref{l:j:coercive} we infer that $$C\geq \M_{\v,\w}^{t}(\varphi)\geq A(1-t)d_1(0, \varphi)-B, $$ and in particular $d_1(0, \varphi) \leq \frac{C_1}{1-t}$. By Lemma \ref{J and d1} there exists a constant $C'$ such that both $|\J^{\omega_0}_{\v,\v\tilde{\v}}(\varphi)|$ and  $|\J^{2\mathrm{Ric}(\omega_0)}_{\v,\w}(\varphi)|$ are bounded by $C'd_1(0,\varphi)$.
\begin{eqnarray*}
t\Ent_\v(\varphi) &\leq &  C +t|\J^{2\mathrm{Ric}(\omega_0)}_{\v,\w}(\varphi)|+ (1-t)|\J^{\omega_0}_{\v,\v \tilde{\v}}(\varphi)|+ \int_X \v(\mu_0) \log \v(\mu_0) \omega_0^{[n]}\\
&\leq& C_2 +C' td_1(0,\varphi) +C' (1-t)d_1(0,\varphi)\\
&\leq& C_2 +\frac{C'' t}{(1-t)}+C''
\end{eqnarray*}
i.e. $\Ent_\v(\varphi) \leq \frac{C_3}{t} +\frac{C_4}{(1-t)}$ for $t\in (0,1)$.
As showed in \cite[Lemma 4.1]{DJL} the weighted entropy $\Ent_\v$ is bounded if and only if the classical entropy $\Ent:=\Ent_1$ does. Hence, we conclude by applying the a priori-estimates proved in \cite{DJL}.
\end{proof}

\subsection{Closeness at time one}
We now prove that the set $S_1$ defined in \eqref{set-solution} is closed under the assumption that $\M_{\v,\w}$ is $\T^{\C}$-coercive.   Consider a sequence $(\varphi_i)_{i\geq0} \subset \mathcal{K}(X,\omega_0)^\T$ of solutions of \eqref{continuity-path-weithed} at $t=t_i$, $t_i\rightarrow1$.
Using the continuity of $\M_{\v,\w}$ and $\J_{\v,\v\tilde{\v}}^{\omega_0}$ along weak geodesic (see \cite{Lah20}), the following generalization of  \cite[Lemma 3.7]{CC21c} is straightforward.

\begin{lemma}{\label{l:minimizing:sequence}}
We have the followings

\begin{equation*}
    \begin{split}
        & \M_{\v,\w}^{t_i}({\varphi}_i) =\inf_{\varphi \in \mathcal{K}(X,\omega_0)^\T} \M_{\v,\w}^{t_i}(\varphi)\underset{t_i\to 1}{\longrightarrow}  \inf_{\varphi \in \mathcal{K}(X,\omega_0)^\T} \M_{\v,\w}(\varphi), \\
        & \M_{\v,\w}({\varphi}_i)\underset{t_i\to 1}{\longrightarrow} \inf_{\varphi \in \mathcal{K}(X,\omega_0)^\T} \M_{\v,\w}(\varphi),
        \\
        & (1-t_i)\J_{\v,\v\tilde{\v}}^{\omega_0}({\varphi}_i)\underset{t_i\to 1}{\longrightarrow} 0.
    \end{split}
\end{equation*}
\end{lemma}
From the  the $\T^{\C}$-coercivity of $\M_{\v,\w}$, Lemma \ref{l:j:coercive} and \eqref{bound:twi:mab}, we get an uniform upper bound of the $\T^{\mathbb{C}}$-relative $d_1$-distance

\begin{equation*}
        \sup_{i\in \mathbb{N}} d_{1,\mathbb{T}^{\mathbb{C}}}(0, \varphi_i) < \infty,
\end{equation*}

\noindent where $d_{1,\mathbb{T}^{\mathbb{C}}}(0,\varphi):= \inf_{\gamma \in \TC}d_1(0, \gamma \cdot \varphi)$ and for any normalized potential $\varphi \in \Kn$,   $\gamma \cdot \varphi$ is defined as the unique potential in $\Kn$  such that $\gamma^*\omega_{\varphi}=\omega_{\gamma \cdot \varphi}$. We also mention that the infimum in  $d_{1,\mathbb{T}^{\mathbb{C}}}(0,\varphi)$ is attained for an element $\gamma_0 \in \mathbb{T}^{\mathbb{C}}$. In other words, for every $\varphi \in \Kn$ there exists $\gamma_0 \in \mathbb{T}^{\mathbb{C}}$ such that

\begin{equation}{\label{inf:d1}}
    d_{1,\mathbb{T}^{\mathbb{C}}}(0,\varphi)=  d_{1}(0,\gamma_0 \cdot \varphi).
\end{equation}
This result was first obtained  in term of Aubin-Mabuchi functionals in \cite[Lemma 6.2]{BD87}, and in term of the $d_1$-distance in \cite[Proposition 6.8.]{DR}. We then deduce the following bound 

\begin{equation}{\label{bdd:d1:dist}}
    \sup_{i\in \mathbb{N}} d_1 (0, \tilde{\varphi}_i) < \infty, \qquad \tilde{\varphi}_i:= \gamma_i \cdot \varphi_i,
\end{equation}
where $\gamma_i\in \TC$ is given by \eqref{inf:d1}.

\noindent

\begin{lemma}{\label{l:T-indep}}
The weighted Mabuchi energy $\mathbf{M}_{\v,\w}$ is $\TC$-invariant if and only if the weighted Futaki invariant vanishes. In particular, $\mathbf{M}_{\v,\w}$ is $\mathbb{T}^{\mathbb{C}}$-invariant if it is bounded from below.
\end{lemma}

\begin{proof}
Let $\gamma \in \TC$. Consider a path $(\gamma_t)_{t \in [0,1]} \subset \TC$ satisfyng $\gamma_0= \mathrm{Id}$ and $\gamma_1=\gamma$ generate by the flow of a real holomorphic vector field $\xi \in \mathfrak{t}$. We deduce from \cite[Section 6]{Lah19} that, for any $\mathbb{T}$-invariant K\"ahler metric $\omega$, the derivative of $\M_{\v,\w}$ along $\gamma_t^*\omega$ is given by the weighted Futaki invariant $\mathbf{F}_{\v,\w}$

\begin{equation}{\label{deriv:mab}}
    \frac{d}{dt}\M_{\v,\w}(\gamma_t^*\omega)= \mathbf{F}_{\v,\w}(\ell),
\end{equation}
where $\ell(x):=\langle x, \xi \rangle + c$, $c \in \R$. In the above equation $\mathbf{F}_{\v,\w}$ is the weighted Futaki invariant, which is defined by 

\begin{equation}{\label{don:fut:gene}}
\mathbf{F}_{\v,\w}(\ell):=\int_X \big(\mathrm{Scal}_\v(\omega) - {\w} (\mu_{\omega})\big) \ell(\mu_{\omega}) \omega^{[n]}.
\end{equation}
The RHS of the above equality does not depend on the choice of $\omega \in [\omega_0]$ by \cite[Section 6]{Lah19}.
It follows from \eqref{deriv:mab} that $\mathbf{F}_{\v,\w} \equiv 0$ implies $\M_{\v,\w}(\gamma^*\omega)=\M_{\v,\w}(\omega)$.
The other implication follows similarly.
Also, from \eqref{deriv:mab} we have $\M_{\v,\w}(\gamma^*\omega)= t \mathbf{F}_{\v,\w}(\ell)+D$. If $\M_{\v,\w}$ is bounded from below, the only possibility is $\mathbf{F}_{\v,\w}(\ell)\equiv 0$.
\end{proof}

\begin{corollary}{\label{bound:ent}}
If $\mathbf{M}_{\v,\w}$ is $\mathbb{T}^{\mathbb{C}}$-invariant, then the entropy is uniformly bounded, i.e. $$\sup_i \Ent(\tilde{\varphi}_i) \leq C,$$ for a constant $C$ independent of $i$.
\end{corollary}

\begin{proof}
Using the $\mathbb{T}^{\mathbb{C}}$-invariance of $\mathbf{M}_{\v,\w}$, and combing Lemma \ref{l:minimizing:sequence} and \ref{l:T-indep}, we find that for any $\varepsilon>0$, there exists $i_0$ such that for any $i\geq i_0$ we have
$$\M_{\v,\w}(\tilde{\varphi}_i)=\M_{\v,\w}(\varphi_i)\leq  \M_{\v,\w}^{t_i}({\varphi_i}) +2\varepsilon \leq \M_{\v,\w}^{t_i}(0)+2\leq C,$$
where $C$ in independent of $i$. 
Now using the weighted Chen-Tian formula \eqref{Chen-Tian} together with Lemma \ref{J and d1} we find that

\begin{equation*}
\begin{split}
\Ent_\v(\tilde{\varphi}_i) &\leq  C +|\J^{2\mathrm{Ric}(\omega_0)}_{\v,\w}(\tilde{\varphi}_i)|+ \int_X \v(\mu_0) \log \v(\mu_0)\omega_0^{[n]}\\
&\leq C_1 +C_2 d_1(0,\tilde{\varphi}_i) \\
& \leq C_3,
\end{split}
\end{equation*}
where we use \eqref{bdd:d1:dist} for the last inequality.
We then infer that the entropy  $\Ent$ is bounded if and only if the weighted entropy $\Ent_\v$ is so \cite[Lemma 4.1]{DJL}.
\end{proof}
A direct adaptation of \cite[Section 3]{DJL} shows that a solution $\varphi_i$ of \eqref{continuity-path-weithed} satisfies

 \begin{equation*}
     \begin{cases}
     F_i &= \log\left( \v(\mu_{\varphi_i}) \frac{\omega_{\varphi_i}^n}{\omega^n_0}\right) \\
    \Delta_{\varphi_i,\v} F_i  &= \tilde{\w}_i(\mu_{\varphi_i}) + 2 \Lambda_{\varphi_i,\v} \left( \alpha_{0,i}\right),
     \end{cases}
\end{equation*}
where 
$$\tilde{\w}_i(x) := - \frac{\w(x)}{\v(x)}+ \frac{(1-t_i)}{t_i}\tilde{\v}(x),\qquad
        \alpha_{0,i} := \mathrm{Ric}(\omega_0 ) - \frac{(1-t_i)}{2t_i} \omega_0.$$
        For every $i\geq0$, let $\omega_i:=\gamma_i^*\omega_0$ and $h_i$ be such that $\omega_i= \omega_0 +dd^c h_i$ with the normalization $\sup_X h_i=0$. In particular, $\tilde{\varphi_i}=h_i+\gamma^* {\varphi_i}$.

\begin{lemma}
The potential $\tilde{\varphi}_i$ defined in \eqref{bdd:d1:dist} satisfies 

 \begin{equation*}
     \begin{cases}
     \tilde{F}_i &= \log\left( \v(\mu_{\tilde{\varphi}_i}) \frac{\omega_{\tilde{\varphi}_i}^n}{\omega^n_0}\right) \\
    \Delta_{\tilde{\varphi}_i,\v} \tilde{F_i}  &=   \tilde{\w}_i(\mu_{\tilde{\varphi}_i})+ 2 \Lambda_{\tilde{\varphi}_i,\v} ( \mathrm{Ric}(\omega_0 ) - \beta_i),
     \end{cases}
\end{equation*}
\noindent where $\beta_i=  \frac{1-t_i}{2t_i}  \omega_i$. Or equivalently,
 \begin{equation}\label{twisted cscK_eq}
     \begin{cases}
     \tilde{F}_i &= \log\left( \v(\mu_{\tilde{\varphi}_i}) \frac{\omega_{\tilde{\varphi}_i}^n}{\omega^n_0}\right) \\
    \Delta_{\tilde{\varphi}_i,\v} (\tilde{F}_i+\tilde{f}_i)  &=   \tilde{\w}_i(\mu_{\tilde{\varphi}_i})+ 2 \Lambda_{\tilde{\varphi},\v} \big(\alpha_{0,i}\big)
     \end{cases}
\end{equation}
where $\tilde{f}_i:=\frac{1-t_i}{2t_i}h_i$.
\end{lemma}

\begin{proof}
By definition $\gamma^*\omega_\varphi=\omega_{\tilde{\varphi}}$. It follows that $\mu_{\tilde{\varphi}}=\gamma^*\mu_{{\varphi}}$. Indeed,
$\omega_{{\varphi}}(\xi, \cdot) = - d\langle \mu_{{\varphi}}, \xi \rangle,$
gives $\omega_{\tilde{\varphi}}(\xi, \cdot)  = \gamma^* \omega_{{\varphi}}(\xi, \cdot) = -d\langle \gamma^*\mu_{{\varphi}}, \xi \rangle,$
which precisely means that $\mu_{\tilde{\varphi}}=\gamma^*\mu_{\varphi}$. In particular for each $i$, $\v(\mu_{\tilde{\varphi}_i})=\gamma^* \v(\mu_{\varphi_i})$ and

\begin{equation*}
    \begin{split}
        \tilde{F}_i = &\log\left( \v(\mu_{\tilde{\varphi}_i}) \frac{\omega_{\tilde{\varphi}_i}^n}{\omega^n_0}\right)= \log\left(  \frac{\gamma_i^*(\v(\mu_{\varphi_i})\omega_{\varphi_i}^n)}{\omega^n_0}\right) \\
        = &\log\left(  \frac{\gamma_i^*(\v(\mu_{\varphi_i})\omega_{\varphi_i}^n)}{\omega^n_{i}}\right)  + \log\left( \frac{\omega^n_i}{\omega_0^n} \right) \\
         = &\gamma_i^*\left(\log\left(  \frac{\v(\mu_{\varphi_i})\omega_{\varphi_i}^n}{\omega^n_0}\right)\right)  + \log\left( \frac{\omega_i^n}{\omega_0^n} \right) \\
        = &\gamma_i^*F_i  + \log\left( \frac{\omega^n_i}{\omega_0^n} \right).  
    \end{split}
\end{equation*}
On the other hand,

\begin{equation*}
    \begin{split}
        \Delta_{\tilde{\varphi}_i, \v} (\gamma_i^*F_i)=& \Delta_{\tilde{\varphi}_i}( \gamma^*_i F_i) + g_{\tilde{\varphi}_i}(d \log \v(\mu_{\tilde{\varphi}_i}), d( \gamma_i^*F_i)) \\
        =& \frac{dd^c \gamma_i^* F_i \wedge \omega^{[n-1]}_{\tilde{\varphi}_i}}{\omega^{[n]}_{\tilde{\varphi}_i}} + \frac{d(\log \v(\mu_{\tilde{\varphi}_i})) \wedge d^c \gamma_i^*F_i \wedge  \omega^{[n-1]}_{\tilde{\varphi}_i}}{\omega^{[n]}_{\tilde{\varphi}_i} } \\
        =& \gamma_i^*\left(\frac{dd^c F_i \wedge \omega^{[n-1]}_{\varphi_i}}{\omega^{[n]}_{\varphi_i}}\right) + \gamma^*_i\left( \frac{d \log \v(\mu_{\varphi_i}) \wedge d^cF_i\wedge  \omega^{[n-1]}_{\varphi_i}}{\omega^{[n]}_{\varphi_i} } \right) \\
        =&\gamma_i^*\left(\Delta_{\varphi_i, \v} F_i\right) \\
    =& \gamma_i^*\big( \tilde{\w}_i(\mu_{\varphi_i})+ 2 \Lambda_{\varphi_i,\v} (\alpha_{0,i}) \\  
    =&  \tilde{\w}_i(\mu_{\tilde{\varphi}_i})+ \gamma_i^*\left(  2 \Lambda_{{\varphi}_i,\v} \big( \mathrm{Ric}(\omega_0 ) - \frac{(1-t_i)}{2t_i}  \omega_0\big)\right) \\ 
   =& \tilde{\w}_i(\mu_{\tilde{\varphi}_i})+ 2 \Lambda_{\tilde{\varphi}_i,\v} ( \mathrm{Ric}(\omega_i ) -\beta_i ).
    \end{split}
\end{equation*}
    The last line follows from similar computations to the ones above. Then

\begin{equation*}
\begin{split}
    \Delta_{\tilde{\varphi}_i,\v} \tilde{F}_i =  & \Delta_{\tilde{\varphi}_i,\v} \gamma^*_i F_i +  \Delta_{\tilde{\varphi}_i,\v} \log\left(\frac{\omega^n_i}{\omega_0^n}\right) \\
    =  & \tilde{\w}_i(\mu_{\tilde{\varphi}_i})+ 2 \Lambda_{\tilde{\varphi}_i,\v} ( \mathrm{Ric}(\omega_i ) - \beta_i) +  \Lambda_{\tilde{\varphi}_i,\v} dd^c \log\left(\frac{\omega^n_i}{\omega_0^n}\right) \\
    = & \tilde{\w}_i(\mu_{\tilde{\varphi}_i})+ 2 \Lambda_{\tilde{\varphi}_i,\v} \left( \mathrm{Ric}(\omega_0\right) - \beta_i). \\
    \end{split}
\end{equation*}
\end{proof}

The rest of this section is devoted to the proof of the following theorem:
\begin{theorem}{\label{t:sol:t:1}}
Assume that $\M_{\v,\w}$ is $\T^{\mathbb{C}}$-coercive, then $S_{t_1}$ is closed. 
\end{theorem}
We follow the ideas in Chen-Cheng \cite{CC21c}.
By \cite[Lemma 3.12]{CC21c} and the initial arguments in  \cite[Proposition 3.13]{CC21c}, for any $p>1$, there exists $\varepsilon_p>0$ such that if $t_i \in (1-\varepsilon_p,1)$, one has

\begin{equation*}
    \int_X e^{-p\tilde{f}_i}\omega_0^{[n]} \leq C,
\end{equation*}
for a constant $C$ uniform in $i$ and $p$. Moreover, we can choose $\varepsilon_p$ small enough such that $e^{-\tilde{f}_i}$ is in $L^q$ for any $q>>1$. Hence, following \cite[Proposition 3.13]{CC21c}, we deduce:

\begin{prop}
For any $p > 1$, there exist $C>0$, and $\varepsilon_p' > 0$, such that for any $t_i \in (1 - \varepsilon_p', 1)$,
\[
\|\tilde{F_i} + \tilde{f}_i\|_{W^{1,2}_p} \leq C, \quad \|\Lambda_{\omega_0}(\omega_{\tilde{\varphi}})\|_{L^p(\omega_0^n)} \leq C,
\]
where $\varepsilon_p'$ depends only on $p$, $\omega_0$, and $C=C(p,\omega_0,\|\v\|_{C^0}, \|\tilde{\w_i}\|_{C^0}, \max |\alpha_{0,i}|_0, \Ent(\tilde{\varphi}_i))$. 
\end{prop}
Observe that $ \|\tilde{\w}_i\|_{C^0}$ is controlled by $ \|{\w}\|_{C^0},  \|{\v}\|_{C^0}$ and $\|\tilde{\v}\|_{C^0}$. Also, $\max |\alpha_{0,i}|_0$ is uniformly bounded by $\max |\mathrm{Ric}(\omega_0)|_0+|\omega_0|_0$ and $\Ent(\tilde{\varphi}_i)$ is uniformly bounded by Corollary \ref{bound:ent} if $\M_{\v,\w}$ is $\T^{\mathbb{C}}$-coercive. Then, up to a subsequence, the sequences $(\tilde{\varphi}_i)$, $(\tilde{F}_i+\tilde{f}_i)$ converge to a limit $\varphi_*$ and $F_*$ and by Sobolev embedding theorem

\begin{align}
    {\label{conv2}}    & \tilde{\varphi}_i \rightarrow \varphi_* \text{ in } C^{1,\alpha} \quad \text{and} \quad  dd^c\tilde{\varphi}_i \rightarrow dd^c\varphi_* \text{ weakly in } L^p, \\
       {\label{conv}} & \tilde{F}_i+\tilde{f}_i \rightarrow F_* \text{ in } C^{\alpha}\quad  \text{and}\quad  d(\tilde{F}_i+\tilde{f}_i) \rightarrow dF_* \text{  weakly in } L^p.
    \end{align}
By \cite[Lemma 3.14]{CC21c}, for any $1 \leq p < \infty$ and any $ 1 \leq k \leq n$,

\begin{equation*}
  \omega_{\tilde{\varphi}_i}^{[k]} \rightarrow  \omega_{\varphi_*}^{[k]} \quad \text{weakly in $L^p$}.
\end{equation*}
Moreover by \eqref{conv2} we have that  $\v(\mu_{\tilde{\varphi}_i}) \rightarrow \v(\mu_{\varphi_*})$ in $C^0$ since $\mu_{\tilde{\varphi}_i}= \mu_0 + d^c \tilde{\varphi}_i$. Similarly, $\tilde{\w}_i(\mu_{\tilde{\varphi}_i})\rightarrow -\frac{\w(\mu_{{\varphi}_*})}{\v(\mu_{{\varphi}_*})}$ in $C^0$. 
Hence

\begin{equation*}
 \v(\mu_{\tilde{\varphi}_i}) \omega_{\tilde{\varphi}_i}^{[k]} \rightarrow  \v(\mu_{\varphi_*}) \omega_{\varphi_*}^{[k]} \quad \text{weakly in $L^p$}.
\end{equation*}

\noindent Following the proof in \cite[Lemma 3.16]{CC21c} combined with Lemma \ref{l:j:coercive} and \ref{l:au:ma} we can also infer that for any $p$ we have
\begin{equation}\label{twist zero}
\|e^{-\tilde{f}_{i}}\|_{L^p(\omega_0^n)} \rightarrow 1.
\end{equation}

\begin{prop}{\label{p:weak:wcscK}}
The limit $\varphi_*$ is a weak $(\v,\w)$-cscK metric. More precisely,

 \begin{equation*}
      \v(\mu_{\varphi_*}) \omega_{\varphi_*}^n= e^{F_*} \omega_0^n
 \end{equation*}     
and 
\begin{equation*} 
\begin{split}
    \int_X \big(F_* dd^c u  - 2u \mathrm{Ric}(\omega_0)\big)&\wedge \v(\mu_{\varphi_*}) \omega_{\varphi_*}^{[n-1]} \\
    =&\int_X\left(-\frac{\w(\mu_{\varphi_*})}{\v(\mu_{\varphi_*})} +2\langle d\log(\v(\mu_{\varphi_*})), \mu_{\mathrm{Ric}(\omega_0)} \rangle \right)\v(\mu_{\varphi_*})\omega_{\varphi_*}^{[n]},
\end{split}
\end{equation*}
for any smooth function $u$ on $X$.
\end{prop}

\begin{proof}
As above
$$
 \v(\mu_{\tilde{\varphi}_i}) \omega_{\tilde{\varphi}_i}^{[k]} \rightarrow  \v(\mu_{\varphi_*}) \omega_{\varphi_*}^{[k]} \quad \text{weakly in $L^p$}.
$$
Also, $e^{\tilde{F}_i}= e^{\tilde{F}_i+\tilde{f}_i} e^{-\tilde{f}_i}$ and $\tilde{F}_i+\tilde{f}_i$ converges to $F_*$ in $L^p$ thanks to \eqref{conv}. By \eqref{twist zero} we can can deduce the first equality.\\
We now multiply the second equation in \eqref{twisted cscK_eq} by $u \in \mathcal{C}^{\infty}(X)^\T$  and we integrate w.r.t. to $\v(\mu_{\tilde{\varphi}_i})\omega_{\tilde{\varphi}_i}^{[n]}$. Integrating by parts gives

\begin{equation*}
    \int_X (\tilde{F}_i+\tilde{f}_i)\Delta_{\tilde{\varphi}_i,\v}(u) \,\v(\mu_{\tilde{\varphi}_i})\omega_{\tilde{\varphi}_i}^{[n]} = \int_X u\big(\tilde{\w}_i(\mu_{\tilde{\varphi}_i})+ 2 \Lambda_{\tilde{\varphi}_i,\v}((\alpha_{0,i})\big)\v(\mu_{\tilde{\varphi}_i})\omega_{\tilde{\varphi}_i}^{[n]}.
\end{equation*}

\noindent Then

\begin{equation}{\label{w:extremal:eq}}
\begin{split}
    \int_X (\tilde{F}_i+\tilde{f}_i)&\big(dd^cu - 2\alpha_{0,i}\big)\v(\mu_{\tilde{\varphi}_i})\omega_{\tilde{\varphi}_i}^{[n-1]} \\
   & = \int_X u\left(\tilde{\w}_i(\mu_{\tilde{\varphi}_i})+ 2\langle d\log \v(\mu_{\tilde{\varphi}_i}), \mu_{\alpha_{0,i}}-d^cu \rangle\right)\v(\mu_{\tilde{\varphi}_i})\omega_{\tilde{\varphi}_i}^{[n]}.
\end{split}    
\end{equation}
From the discussion above we know that $\v(\mu_{\tilde{\varphi}_i})\omega_{\tilde{\varphi}_i}^{[n]}$ converge to $\v(\mu_{\varphi_*})\omega_{\varphi_*}^{[n]}$ in $L^p$  and that $\mu_{\tilde{\varphi}_i}\rightarrow \mu_{\varphi_*}$ in $\mathcal{C}^0(X)$. We infer that $\mu_{\alpha_{0,i}}\rightarrow \mu_{\mathrm{Ric}(\omega_0)}$  since  $\mu_{\alpha_{0,i}}= \mu_{\mathrm{Ric}(\omega_0)}- \frac{(1-t_i)}{2t_i}\mu_{\omega_0}$. \\
Finally, as observed above, we have $\tilde{\w}_i(\mu_{\tilde{\varphi}_i})\rightarrow-\frac{\w(\mu_{{\varphi}_*})}{\v(\mu_{{\varphi}_*})}:= \tilde{\w}(\mu_{{\varphi}_*})$ in $C^0$. 
Hence

\begin{equation}
    \text{ RHS of \eqref{w:extremal:eq} } \longrightarrow \int_X u\big(\tilde{\w}(\mu_{{\varphi}_*})+ 2\langle d\log(\v(\mu_{\varphi_*})), \mu_{\mathrm{Ric}(\omega_0)} \rangle-d^cu\big)\v(\mu_{\varphi_*})\omega_{\varphi_*}^{[n]}. 
\end{equation}
For the LHS we already know that $\tilde{F}_i + \tilde{f}_i \rightarrow F_*$ in $L^p$. Hence
\begin{equation*}
    \text{ LHS of \eqref{w:extremal:eq} }\rightarrow   \int_X F_*\bigg(dd^cu   - 2\mathrm{Ric}(\omega_0)\bigg)\v(\mu_{\varphi_*}) \omega_{\varphi_*}^{[n-1]}.
\end{equation*}
\end{proof}

It follows from \eqref{conv2} that $\varphi_*$ belong to $W^{2,p}(X)$. For any $\xi \in \mathfrak{t}$, we have $\langle \mu_{\varphi_*},\xi\rangle= \langle \mu_{0},\xi\rangle+d^c\varphi(\xi)$. Hence, $\langle \mu_{\varphi_*},\xi\rangle$ belongs to $W^{1,p}(X)$. By Proposition \ref{p:weak:wcscK}, $\varphi_*$ satisfies $\omega_{\varphi_*}^{n}=e^{F_*-\log(\v(\mu_{\varphi_*}))}\omega_0^n$. Since $F_*$ also is in  $W^{1,p}(X)$, $F_*-\log(\v(\mu_{\varphi_*}))$ belongs too $W^{1,p}(X)$. Hence the same argument of Chen--Cheng \cite[Corollary 3.18]{CC21c}, based on the apriori estimates of Chen--He \cite{CH12} for complex Monge-Ampère equation with $W^{1,p}$ RHS apply.

\section{Applications}{\label{s:app}}

\subsection{Openness of existence in the space of weights}{\label{s:app:open}}

In this section, we generalize a theorem of Apostolov–Lahdili–Legendre \cite[Theorem 1.1]{ALL} for \(\v\)-solitons to any weighted cscK metric.  Let \(\mathcal{C}^{\infty}_{\log\text{-}\mathrm{c}}(P,\mathbb{R}_{>0})\) be the space of \(\log\)-concave functions in \(\mathcal{C}^{\infty}(P,\mathbb{R}_{>0})\). We define the space of weights such that the weighted Futaki invariant vanishes

\begin{equation*}
    F(X,[\omega_0]):= \{ (\v,\w) \in \mathcal{C}_{\log\text{-}\mathrm{c}}^{\infty}(P,\R_{>0}) \times \mathcal{C}^{\infty}(P,\R) \text{ s.t. } \mathbf{F}_{\v,\w} \equiv 0\},
\end{equation*}
where $\mathbf{F}_{\v,\w}$ is defined in \eqref{don:fut:gene}. We let
\begin{equation*}
    S(X,[\omega_0]):= \{ (\v,\w) \in \mathcal{C}_{\log\text{-}\mathrm{c}}^{\infty}(P,\R_{>0}) \times \mathcal{C}^{\infty}(P,\R) \text{ s.t. } \exists \, \omega \text{ } (\v,\w)\text{-cscK} \text{ in } [\omega_0]\}.
\end{equation*}

Observe that, since $(\v,\w)$-cscK metrics are minimizers of $\M_{\v,\w}$, it follows that $\M_{\v,\w}$ is bounded. Hence, by Lemma \ref{l:T-indep} the weighted Futaki invariant vanishes, and in particular

\begin{equation*}
    S(X,[\omega_0]) \subset F(X,[\omega_0]).
\end{equation*}

\begin{lemma}{\label{l:coerc}}
Suppose $\M_{\v,\w}$ is $\TC$-coercive. Then for any $(\v_{\varepsilon}, \w_{\varepsilon})$ in $ F(X,[\omega_0])$ close to $(\v,\w)$ for the $\mathcal{C}^{\infty}$-topology, $\M_{\v_\varepsilon,\w_\varepsilon}$ is $\TC$-coercive.
\end{lemma}

\begin{proof}
Let $f_\varepsilon:=\v_{\varepsilon}(\mu_\varphi)\frac{\omega^n_{\varphi}}{\omega_0^n}$ and $h:=\v(\mu_\varphi)\frac{\omega^n_{\varphi}}{\omega_0^n}$, so that $f_\varepsilon=h\frac{\v_{\varepsilon}(\mu_\varphi)}{\v(\mu_\varphi)}$. To simply notations, we write 

\[V:= \mathrm{Vol}_\v([\omega_0])=\int_X \v(\mu_\varphi) \omega_\varphi^{[n]} \quad \text{ and } \quad V_\varepsilon:=\mathrm{Vol}_{\v_\varepsilon}([\omega_0])= \int_X \v_\varepsilon(\mu_\varphi) \omega_\varphi^{[n]},\]
which does not depend on $\omega_\varphi \in [\omega_0]$ K\"ahler by \cite[Lemma 1]{Lah19}. A similar computation than the one in \cite[Lemma 4.1]{DJL} gives:

\begin{equation*}
    \begin{split}
        \Ent_{\v}(\varphi)= &\frac{1}{V}\int_X h \log h \omega_0^{[n]} - \log V\\
        = & \frac{1}{V} \int_X \frac{\v(\mu_\varphi)}{\v_{\varepsilon}(\mu_\varphi)} f_\varepsilon \log \left(\frac{\v(\mu_\varphi)}{\v_{\varepsilon}(\mu_\varphi)} f_\varepsilon\right) \,\omega_0^{[n]}  - \log V \\
        =&   \frac{1}{V} \int_X \frac{\v(\mu_\varphi)}{\v_{\varepsilon}(\mu_\varphi)} f_\varepsilon \log \left(\frac{\v(\mu_\varphi)}{\v_{\varepsilon}(\mu_\varphi)} \right) \,\omega_0^{[n]} + \frac{1}{V} \int_X \frac{\v(\mu_\varphi)}{\v_{\varepsilon}(\mu_\varphi)} f_\varepsilon \log  f_\varepsilon \,\omega_0^{[n]}- \log V
        \\
        \leq &  \frac{1}{V} \int_X \v(\mu_\varphi)\log \left(\frac{\v(\mu_\varphi)}{\v_{\varepsilon}(\mu_\varphi)} \right) \omega_\varphi^{[n]} 
        + \frac{1}{V} \sup_P\left(\frac{\v}{\v_{\epsilon}}\right) \int_{\{f_\varepsilon>1\}}  f_\varepsilon \log f_\varepsilon \,\omega_0^{[n]} - \log V \\
     \leq  & C_\varepsilon  + \frac{1}{V}\sup_P\left(\frac{\v}{\v_{\varepsilon}}\right) \int_{X}  f_\varepsilon \log  f_\varepsilon \,\omega_0^{[n]}+   \frac{1}{V}\sup_P\left(\frac{\v}{\v_{\varepsilon}}\right) \int_{\{f_\varepsilon \leq 1\}}  f_\varepsilon (-\log  f_\varepsilon) \,\omega_0^{[n]} -\log V \\
     \leq & C_\varepsilon + \frac{V_\varepsilon}{V} \sup_P\left(\frac{\v}{\v_{\epsilon}}\right) \Ent_{\v_{\epsilon}}(\varphi) +\frac{1}{V}\sup_P\left(\frac{\v}{\v_{\varepsilon}}\right) e^{-1} + \log{V_\varepsilon}{V}^{-1}.
        \end{split}
\end{equation*}
Hence

\begin{equation*}
   \Ent_{\v_{\varepsilon}}(\varphi) \geq -C_\varepsilon +  M_{\varepsilon}\Ent_\v(\varphi),
\end{equation*}
where 

\begin{equation*}
    M_{\varepsilon}:=\frac{V_\varepsilon}{V} \sup_P\left(\frac{\v}{\v_{\epsilon}}\right).
\end{equation*}
Observe that  $M_\varepsilon \rightarrow 1$ when $\v_{\varepsilon} \rightarrow \v $ smoothly. We consider the function $\J_{\v_{\varepsilon}}^{2\mathrm{Ric}(\omega_0)}$ defined in  \eqref{j-v}. By linearity of $\J_\v^{2\mathrm{Ric}(\omega_0)}$ with respect to $\v$ (see e.g. \cite[Lemma 6.12]{AJL}), we have that

\begin{equation*}
      - \J_{\v_{\varepsilon}}^{2\mathrm{Ric}(\omega_0)}(\varphi) = -M_{\varepsilon}\J_{\v}^{2\mathrm{Ric}(\omega_0)}(\varphi) - \J_{\v_{\varepsilon}-M_{\varepsilon}\v}^{2\mathrm{Ric}(\omega_0)}(\varphi).
\end{equation*}
Moreover, since $ -C \omega_0 \leq 2 \mathrm{Ric}(\omega_0) \leq C \omega_0 $, by \cite[(54) \& (55)]{AJL} we have

\begin{equation*}
    \J_{\v_{\varepsilon}-M_{\varepsilon}\v}^{2\mathrm{Ric}(\omega_0)}(\varphi) \leq C_{\varepsilon} \int | \varphi | \omega_{\varphi}^{[n]}
\end{equation*}
where $C_{\varepsilon}$ is equal (up to a multiplicative positive constant independent of $\varepsilon$ and $\varphi$) to 

\begin{equation*}
    \sup_{P}|\v_{\varepsilon}-M_\varepsilon \v| + \sup_{P \times P} | \langle d\v_{\varepsilon}-M_\varepsilon d\v, \cdot \rangle |.
\end{equation*}
 When \( \v_{\varepsilon} \) approaches \( \v \) in the smooth topology, we do have  that \( M_{\varepsilon} \) tends to \( 1 \), hence \( C_{\varepsilon} \) tends to \( 0 \). In the following, \( C_{\varepsilon} \) will denote several positive constants independent of \( \varphi \), which goes to \( 0 \) when \( \v_{\varepsilon} \) converges to \( \v \) in the smooth topology. Now, applying a result of Darvas \cite[Theorem 5.5]{Dar}, we obtain that

\begin{equation*}
   - \J_{\v_{\varepsilon}}^{2\mathrm{Ric}(\omega_0)}(\varphi)
\geq  - M_{\varepsilon} \J_{\v}^{2\mathrm{Ric}(\omega_0)}(\varphi) -  C_{\varepsilon} d_{1}(0,\varphi) - C_{\epsilon}.
\end{equation*}
Moreover, by \cite[Lemma 6.9]{AJL} and \cite[Theorem 5.5]{Dar},

\begin{equation*}
    \begin{split}
        \mathbf{E}_{\w_{\varepsilon}}(\varphi) =  &M_{\varepsilon} \mathbf{E}_{\w}(\varphi) + \mathbf{E}_{\w_{\varepsilon}-M_{\varepsilon}\w}(\varphi) \\
        \geq & M_{\varepsilon} \mathbf{E}_{\w}(\varphi) - \sup_P|\w_{\varepsilon}-M_{\varepsilon}\w| \int_C |\varphi|\omega_{\varphi}^{[n]}  \\
        \geq & M_{\varepsilon}\mathbf{E}_{\w}(\varphi) -C_{\varepsilon} d_{1}(0,\varphi),
    \end{split}
\end{equation*}
(see \eqref{I:v:fonc} for the definition of $\mathbf{E}_\w$). By definition of $ \J^{2\mathrm{Ric}(\omega_0)}_{\v,\w}$ we get
$$-\J^{2\mathrm{Ric}(\omega_0)}_{\v_\varepsilon,\w_\varepsilon}(\varphi)= -\J^{2\mathrm{Ric}(\omega_0}_{\v_\varepsilon} (\varphi)+ \mathbf{E}_{\w_\varepsilon}(\varphi) \geq  - M_{\varepsilon} \J_{\v, \w}^{2\mathrm{Ric}(\omega_0)}(\varphi) -  C_{\varepsilon} d_{1}(0,\varphi) - C_{\epsilon}.$$
Hence,
\begin{equation*}
\begin{split}
   \M_{\v_{\varepsilon},\w_{\varepsilon}}(\varphi) \geq &M_{\varepsilon}\M_{\v,\w}(\varphi) - C_{\varepsilon}d_{1}(0,\varphi) - C_{\epsilon}.
 \end{split}  
\end{equation*}
Now let $\gamma_0 \in \TC$ such that

\begin{equation*}
    d_1(0,\gamma_0 \cdot \varphi)= \inf_{ \gamma \in \TC} d_{1}(0,\gamma \cdot \varphi).
\end{equation*}
Since $(\v_{\varepsilon}, \w_{\varepsilon}) \in F(X,[\omega_0])$, $\M_{\v_{\varepsilon},\w_{\varepsilon}}$ is $\T^{\mathbb{C}}$-invariant (see Lemma \ref{l:T-indep}). The same conclusion holds for $(\v,\w)$. Hence 

\begin{equation*}
\begin{split}
  \M_{\v_{\varepsilon},\w_{\varepsilon}}( \varphi)=& \M_{\v_{\varepsilon},\w_{\varepsilon}}(\gamma_0 \cdot \varphi) \\
  \geq &M_{\varepsilon}\M_{\v,\w}(\gamma_0 \cdot \varphi) - C_{\varepsilon}d_{1}(0,\gamma_0 \cdot \varphi) - C_{\epsilon} \\
  =& M_{\varepsilon}\M_{\v,\w}(\varphi) - C_{\varepsilon} \inf_{ \gamma \in \TC}d_{1}(0, \varphi) - C_{\epsilon}\\
    \geq & M_{\varepsilon}\inf_{ \gamma \in \TC}d_{1}(0, \varphi)- C_{\varepsilon} \inf_{ \gamma \in \TC}d_{1}(0, \varphi) - C \\
   \geq & ( M_{\varepsilon}- C_{\varepsilon} )\inf_{ \gamma \in \TC}d_{1}(0, \varphi) - C, \\
 \end{split}  
\end{equation*}
which conclude the proof by choosing $\v_\varepsilon$ close enough to $\v$ for the $\mathcal{C}^{\infty}$-topology.

\end{proof}

\begin{corollary}{\label{c:weight:open}}
$S(X,[\omega_0])$ is open in $F(X,[\omega_0])$ for the $\mathcal{C}^{\infty}$-topology.
\end{corollary}

\begin{proof}
By \cite[Theorem 1]{AJL}, $\M_{\v,\w}$ is $\TC$-coercive. From Lemma \ref{l:coerc}, we deduce that $\M_{\v_\varepsilon,\w_\varepsilon}$ is $\TC$-coercive. We conclude thanks to Theorem \ref{t:exi}.
\end{proof}

\begin{remark}
The results above holds for the $(\mathcal{C}^1 \times \mathcal{C}^0)$-topology as well.
\end{remark}

\subsection{Weighted toric Yau-Tian-Dolnaldson correspondence}{\label{s:ytd}}

We consider a toric K\"ahler manifold $(X,\omega_0,\mathbb{T})$ with labeled Delzant polytope $(P, \textbf{L})$ \cite{VA4, TD}, where $\textbf{L} =(L_j)_{j=1}^k$ is the collection of non-negative affine linear functions defining $P$, with $dL_j$ being primitive elements of the lattice of circle subgroups of $\mathbb{T}$.\\
Following \cite{SKD, Lah19, LLS}, given $\v \in \mathcal{C}^{\infty}(P,\R_{>0})$ and  $\w \in \mathcal{C}^{\infty}(P,\R)$, we introduce the toric weighted Donaldson--Futaki invariant

\begin{equation}{\label{define-futaki}}
  \mathcal{F}_{\v,\w}(f):=  2\int_{\partial P} f\v d\sigma -  \int_P  f  \w  dx,
\end{equation}
defined for all continuous functions $f$ on $P$, where $d\sigma$ is the induced measure on each face $F_i \subset \partial P$ by letting $dL_i \wedge d \sigma = -dx$, $dx$ being the Lebesgue measure on $P$. The toric weighted Futaki invariant is related to the one defined in \eqref{don:fut:gene}. \\
We denote by $\mathcal{CV}^{\infty}(P)$ the set of continuous convex functions on $P$ which are smooth in the interior $P^0$.  We fix $x_0 \in P^0$ and consider the following  normalization
 \begin{equation}{\label{normalized-function-polytope}}
     \mathcal{CV}^{\infty}_*(P):=\{ f \in \mathcal{CV}^{\infty}(P) \text{ } | \text{ } f(x) \geq f(x_0)=0 \}.
 \end{equation}
 We denote by $f^*$ the linear projection of $f\in \mathcal{CV}^{\infty}(P)$ to  $\mathcal{CV}^{\infty}_*(P)$.
 \smallskip
 
 We now recall the definition of stability introduced in \cite{Jub23, LLS}, which is a generalization of the one in \cite{SKD}:

\begin{define}{\label{uniform-K-stable}}
A labeled Delzant polytope $(P,\mathbf{L})$ is  $(\v,\w)$-uniformly K-stable if there exists $\lambda > 0$ such that

\begin{eqnarray}{\label{uniform-equation}}
\mathcal{F}_{\v,\w}(f) \geq \lambda \| f^*\|_{L_1(P)},
\end{eqnarray}
\noindent for all  $f \in \mathcal{CV}^{\infty}(P)$.
\end{define}

\subsubsection{Existence via coercivity}

We recall the following fact proved in \cite[Proposition 7.3 and (7.17)]{Jub23}:

\begin{prop}{\label{p:stable:propre}}
Suppose that the labeled moment polytope $(P,\mathbf{L})$ is $(\v,\w)$-uniformly K-stable. Then the weighted Mabuchi energy $\M_{\v,\w}$ of $(X,\omega_0,\T)$ is $\mathbb{T}^{\mathbb{C}}$-coercive.
\end{prop}

It then follows straightforwardly from Theorem \ref{t:exi} that

\begin{corollary}{\label{c:ytd}}
Let $(X,\omega_0,\T)$ be a compact toric K\"ahler manifold with labeled Delzant polytope $(P,\textbf{L})$. Suppose that $\v$ is $\log$-concave. Then there exists a $(\v,\w)$-cscK metric in $[\omega_0]$ if and only if $(P,\mathbf{L})$ is uniformly K-stable.
\end{corollary}
For completeness, we mention that the direction ``existence implies stability" is shown in \cite{LLS}.

\subsubsection{Existence via estimates}{\label{s:ytd2}}

We now present an alternative proof of Corollary \ref{c:ytd}, which only relies on the estimates provided in \cite{DJL}. We follow a beautiful idea in \cite{LLS23} in the cscK case. 

Let \(\mathcal{S}(P,\mathbf{L})\) be the space of symplectic potentials, i.e. the space of continuous convex functions on \(P\), smooth in the interior \(P^0\), which satisfy Guillemin's boundary conditions defined in \cite{Gui}. It is well known that each Kähler potential defines a Kähler structure on \(X\) \cite{SKD, Gui}. We denote by \(\mathrm{Scal}_\v(u) : P \longrightarrow \mathbb{R}\) the weighted scalar curvature of the Kähler structure associated with the symplectic potential \(u\); see \cite[eq. (60)]{Lah19} for an explicit expression of \(\mathrm{Scal}_\v(u)\).  

Let \(u_0\) be the Guillemin Kähler potential. Given $\v_0>0$ a fixed weight, we consider the continuity path

\begin{equation}{\label{toric:cont:path}}
    \mathrm{Scal}_{\v_t}(u_t)=\w_t, \qquad t\in[0,1], u_t \in \mathcal{S}(P,\mathbf{L}),
\end{equation}
where $\w_t:=t \w +(1-t) \mathrm{Scal}_{\v_0}(u_0)$ and $\v_t:= t\v + (1-t)\v_0$. We let

\begin{equation*}
    S:=\{  t \in [0,1] \text{ } |  \text{ } \exists \, u_{t} \in \mathcal{S}(P,\mathbf{L})\text{ solution of }\eqref{toric:cont:path}\}.
\end{equation*}
We start with the following key observation:

\begin{lemma}{\label{l:stable}}
Suppose that $(P,\mathbf{L})$ is $(\v,\w)$-uniformly K-stable. Then $(P,\mathbf{L})$ is $(\v_t,\w_t)$-uniformly K-stable for any $t\in[0,1]$.
\end{lemma}

\begin{proof}
The symplectic potential $u_0$  trivially solve
\begin{equation}{\label{cont:path:tor}}
    \mathrm{Scal}_{\v_0}(u_0)=\w_0, \qquad u_0 \in \mathcal{S}(P,\mathbf{L}).
\end{equation}
Hence by \cite{LLS}, $(P,\mathbf{L})$ is $(\v_0,\mathrm{Scal}_{\v_0}(u_0))$ uniformly K-stable. We let $\lambda_0$ and $\lambda_1$ be the constants in the stability condition in \eqref{uniform-equation} corresponding to the weights  $(\v_0,\mathrm{Scal}_{\v_0}(u_0))$ and $(\v,\w)$ respectively.\\
Let $f \in \mathcal{CV}^{\infty}(P)$. Since $P$  is uniformly K-stable with respect to $(\v,\mathrm{Scal}(u_0))$ and $(\v,\w)$ we have
\begin{equation*}
\begin{split}
    \mathcal{F}_{\v_t,\w_t}(f)  =&(1-t) \mathcal{F}_{\v,\mathrm{Scal}_{\v_0}(u_0)}(f) + t \mathcal{F}_{\v,\w}(f) \\
        \geq & \lambda_t \int_P f^* dx,
\end{split}    
\end{equation*}
where $\lambda_t:= t \lambda_1 + (1-t) \lambda_0$.
\end{proof}
Following \cite{SKD, Jub23}, we let $\M_{\v,\w}^P : \mathcal{S}(P,\mathbf{L})\longrightarrow \mathbb{R}$ be defined as

\begin{equation*}
    \M_{\v,\w}^P(u)= \frac{1}{(2\pi)^n} \M_{\v,\w}(\varphi_u),
\end{equation*}
where $\varphi_u \in \mathcal{K}(X,\omega_0)^\T$ is the K\"ahler potential corresponding the symplectic potential $u$ via the Legendre transform (see \cite{SKD, Gui}).

\begin{prop}{\label{p:toric:ent}}
Suppose that $(P,\mathbf{L})$ is $(\v,\w)$-uniformly K-stable. Then the entropy $\Ent: \mathcal{K}(X,\omega_0)^\T \longrightarrow \R$ is bounded along the K\"ahler potentials $\varphi_{{u_t}} \in \mathcal{K}(X,\omega_0)^\T$ corresponding to a symplectic potential $u_t \in \mathcal{S}(P,\mathbf{L})$.
\end{prop}

\begin{proof}
Using the explicit formula for $\M_{\v,\w}^P$ in \cite[Section 7.4]{Jub23}, we find that

\begin{equation*}
    \M_{\v_t,\w_t}^P= t\M_{\v,\w}^P + (1-t)\M_{\v_0,\mathrm{Scal}_{\v_0}(u_0)}^P.
\end{equation*}
Since $u_t$ minimize $ \M_{\v_t,\w_t}^P$ (by \cite[Proposition 7.7]{Jub23}), using the above we find that

\begin{equation*}
    \begin{split}
       \M_{\v_t,\w_t}^P(u_t) \leq \M_{\v,\w}^P(0) + \M_{\v_0,\mathrm{Scal}_{\v_0}(u_0)}^P(0)\leq  C,
    \end{split}
\end{equation*}
where $C>0$ is a constant independent of $t$ and $u_t$. By the weighted Chen-Tian \eqref{Chen-Tian} formula we find

\begin{equation*}
\begin{split}
    \Ent_{\v_t}(\varphi_{u_t})=& \M_{\v_t,\w_t}(\varphi_{u_t}) + \J^{2\mathrm{Ric}(\omega_0)}_{\v_t,\w_t}(\varphi_t)+\int_X \log \v(\mu_0) \v(\mu_0) \omega_0^n \\
    = &(2\pi)^n \M^{P}_{\v_t,\w_t}(u_t) + \J^{2\mathrm{Ric}(\omega_0)}_{\v_t,\w_t}(\varphi_{u_t}) +\int_X \log \v(\mu_0) \v(\mu_0) \omega_0^n\\
    \leq &  C+ \J^{2\mathrm{Ric}(\omega_0)}_{\v_t,\w_t}(\varphi_{u_t}) \\
     \leq &C + Cd_1(0,\varphi_{u_t}) \\
    \leq  & C_1 + C_1\int_P |u_t| dx \leq C_2.
\end{split}    
\end{equation*}
The second inequality follows from \cite[Lemma 6.12 and 6.5]{AJL}, as well as \cite[Theorem 3 and Corollary 4.14]{Dar}. The third inequality is proved in \cite[eq. (66)]{apo}. Finally, the last equality follows from Lemma \ref{l:stable} and the fact that \(\int_P |u_t| \, dx\) is bounded, given that \((P,\mathbf{L})\) is \((\v_t,\w_t)\)-uniformly K-stable. The conclusion follows from the fact that the standard entropy \(\Ent\) is bounded if and only if the weighted entropy \(\Ent_\v\) is bounded (see \cite[Lemma 4.1]{DJL}).
\end{proof}

\begin{lemma}
$S$ is non-empty, open and closed. 
\end{lemma}

\begin{proof}
$S$ is clearly non-empty as $u_0$ is a solution at $t=0$. \\
Via the moment map $\mu_{\varphi_{u_t}} :X \longrightarrow P$, the equation $\mathrm{Scal}_{\v_t}(u_t)-\w_t=0$ writes on $X$ (up to a factor $\v_t(\mu_{\varphi_{u_t}})$) as

\begin{equation*}
\frac{1}{\v_t(\mu_{\varphi_{u_t}})}(\mathrm{Scal}_{\v_t}(\varphi_{u_t})- \w_t(\mu_{\varphi_{u_t}}))=0.
\end{equation*}
By \eqref{comp:lichne}, the linearization of the LHS at $\varphi_{{u_t}}$ is equal to the weighted Lichnerowicz operator

\begin{equation*}
    D_{\varphi_{u_t}}\left(\frac{1}{\v_t(\mu_{\varphi_{u_t}})}\left(\mathrm{Scal}_{\v_t}({\varphi_{u_t}})- \w_t(\mu_{{\varphi_{u_t}}})\right)\right)(\dot{\varphi})= \mathbb{L}_{\varphi_{u_t},\v_t}(\dot{\varphi}).
\end{equation*}
Since \(\mathbb{L}_{\varphi,\v}\) is elliptic and self-adjoint with respect to \(\v(\mu_\varphi)\omega_\varphi^{[n]}\) \cite[Lemma B1]{Lah19}, standard arguments from elliptic PDE theory show the openness in the space of Kähler potentials. By the well-known correspondence between Kähler and symplectic potentials \cite{SKD, Gui} we get the openness for $S$.

For the closeness we use the same arguments in the proof of \cite[Theorem 4.3]{LLS23} (which work thanks to Proposition \ref{p:toric:ent}) and the estimates in \cite{DJL}.
\end{proof}

\subsection{Weighted cscK problem on semisimple principal fibration}{\label{s:fibration}}

We briefly recall the setting of a semisimple principal fibration as in \cite{AJL}.  

Let \((X, \omega_0)\) be a compact Kähler \(n\)-manifold endowed with a Hamiltonian isometric action of an \(r\)-dimensional torus \(\T\). Let \(\pi_B : Q \longrightarrow B\) be a principal \(\T\)-bundle over a \(2m\)-dimensional manifold \((B,\omega_B) = \prod_{a=1}^{k} (B_a, \omega_{B_a})\), where each \((B_a, \omega_{B_a})\), \(a = 1, \ldots, k\), is a compact cscK polarized Kähler \(2m_a\)-manifold, such that there exists a connection 1-form \(\theta \in \Omega^1(Q, \mathfrak{t})\) with curvature

\begin{equation}\label{def:connection}
    d\theta = \sum_{a=1}^{k} \pi_B^*(\omega_{B_a}) \otimes p_a,
\end{equation}  
where \(p_a \in \mathfrak{t}\).  We consider the action of \(\T\) generated by elements of the form \(-\xi_i^X + \xi_i^Q\) on \(X \times Q\) and define  
\[
Y := (X\times Q)/\T.
\]  
It is shown in \cite[Section 5]{AJL} that \(\pi_B : Y \longrightarrow B\) is an holomorphic fibration over \(B\). Moreover, thanks to \cite[eq. (25)]{AJL}, we have the following identification:

\begin{equation}{\label{iden:func}}
    \mathcal{C}^{\infty}(Y)^\T = \mathcal{C}^{\infty}(X \times B)^\T.
\end{equation}  
Also, from \cite[Section 5]{AJL}, any Kähler metric \(\omega\) on \(X\) induces a Kähler metric \(\tilde{\omega}\) on \(Y\) such that the restriction \((Y_b, \tilde{\omega}_b)\) of \((Y, \tilde{\omega})\) to a fiber over a point \(b \in B\) is \(\T\)-equivariantly Kähler isomorphic to \((X, \omega)\). Such a metric \(\tilde{\omega}\) on \(Y\) is the \emph{compatible Kähler metric associated with} \(\omega\).  From now on, $\omega$ will denote a K\"ahler metric on $X$ and $\tilde{\omega}$ the compatible one on $Y$.
\smallskip

We recall the following fundamental formula for the scalar curvature of a compatible Kähler metric (see \cite[Lemma 5.8]{AJL}):

\begin{lemma}\label{l:Y-Scal}   Let $\v>0$ be a weight function on $P$. Then
\begin{equation*}
\frac{\mathrm{Scal}_\v(\tilde \omega)}{\v(\mu_\omega)} = \frac{\mathrm{Scal}_{\mathrm{p} \v}(\omega) }{(\mathrm{p}\v )(\mu_{{\omega}})}+ \v(\mu_\omega)q(\mu_{\omega}) 
\end{equation*}
with $\mathrm{p}(x)= \prod_{a=1}^k(\langle p_a, x\rangle + c_a)^{m_a}$ and $\mathrm{q}(x)=\sum_{a=1}^k \frac{\mathrm{Scal}(\omega_{B_a})}{\langle p_a, x\rangle + c_a}$,
where $c_a$ are constant such that $\langle p_a, x\rangle + c_a>0$.
\end{lemma}

\begin{lemma}\label{l:operators-tilde} Let $\psi$ be a $\T$-invariant smooth function on $Y$ -seen as a $\T$-invariant function on $X \times B$ via \eqref{iden:func}. For any $x\in X$, let $\psi_x$ be the smooth function on $B$ associate to $\psi$; and for any $b\in B$, let $\psi_b$ be the smooth function on $X$ associate to $\psi$. We then have 
\smallskip
\begin{itemize}
\item[(i)] $\Delta_{\tilde \omega,\v}^Y \psi = \Delta^X_{\omega,\mathrm{p}\v } \psi_b + \Delta^B_{x} \psi_x$;
\smallskip
\item[(ii)] $\mathbb{L}^Y_{\tilde \omega,\v} \psi  = \mathbb{L}^{X}_{\omega, \mathrm{p}\v } \psi_b  + \mathbb{L}^B_x \psi_x + \Delta^B_x\left(\Delta^X_{\omega,\mathrm{p}\v}\psi_b \right)_x +  \Delta^X_{\omega,\mathrm{p}\v}\left(\Delta^B_{x} \psi_x \right)_{b} 
                  + \sum_{a=1}^k Q_a(x)\Delta^{B_a}_{\omega_{B_a}} \psi_x$;     
                  \smallskip
\item[(iii)] $\left({\mathbb{H}}^{\tilde \theta}_{\tilde \omega, \v}\right)^Y \psi  = \left({{\mathbb{H}}}^{\theta}_{\omega, \mathrm{p}\v }\right)^X \psi_b  + \sum_{a=1}^k P_a(x)\Delta^{B_a}_{\omega_{B_a}} \psi_x$.
\end{itemize}
\end{lemma}
Here $\Delta^B_x$ denotes the Laplacian with respect to $\omega_{B,x}:= \sum_{a=1}^k (\langle p_a, x\rangle + c_a)\pi_B^* (\omega_{B_a})$ on $B$. \\

This lemma was proved in \cite[Lemma A.3]{AJL} when $\v=1$. We provide a proof in the general case for reader's convenience.

\begin{proof}
By \cite[Lemma A.3]{AJL}
\begin{equation*}
\Delta_{\tilde \omega}^Y \psi = \Delta^X_{\omega,\mathrm{p}} \psi_b + \Delta^B_{x} \psi_x.
\end{equation*}
Hence we only need to compute the weighted part of the Laplacian, i.e. $g_{\tilde{\omega}}(d_Y \psi, d_Y \log(\v(\mu_\omega)))$. By \eqref{iden:func}, $d_Y \psi= d_X \psi_b + d_B \psi_x$. Let $\theta{^r}:=\theta_1 \wedge \dots \wedge \theta_r$, where $\theta$ is defined in \eqref{def:connection} and $\theta_i:= \theta(\xi_i)$ with $(\xi_1,\dots, \xi_r)$ a fixed basis of $\mathfrak{t}$.
Also,

\begin{equation*}
    \begin{split}
       g_{\tilde{\omega}}(d_Y \psi_b, d_Y \log(\v(\mu_\omega)))
       &= \frac{d_Y \psi_b \wedge d^c_Y \log(\v(\mu_\omega)) \wedge \tilde \omega^{[m+n-1]}\wedge \theta^{ r}}{\tilde \omega^{[m+n]}\wedge \theta^{ r}} \\
       &=  \frac{d_X \psi_b \wedge d^c_X \log(\v(\mu_\omega)) \wedge \omega^{[n-1]}\wedge \mathrm{p}(\mu_{\omega}) (\pi_B^* \omega_B)^{[m]}\wedge \theta^{ r}}{\omega^{[n]}\wedge \mathrm{p}(\mu_{\omega})(\pi_B^* \omega_B)^{[m]} \wedge \theta^{ r}} \\
       &= g_{\omega}(d_X \psi_b, d_X \log \v(\mu_\omega)).
    \end{split}
\end{equation*}
We refer to \cite[eq. (27)]{AJL} for the first equality. The term involving $d_B \psi_x$ can be computed similarly. The proof of $(i)$ concludes simply using the definition of $\Delta^Y_{\tilde{\omega}, \v}$. \\
For $(ii)$ we first recall that
\begin{equation}\label{liche:decompo}
\begin{split}
\mathbb{L}^Y_{\tilde{\omega},\v}(\psi)=&\frac{1}{2}(\Delta^Y_{\tilde{\omega},\v})^{2}(\psi)+\left(d^Y\right))^*_{\tilde{\omega},\v}\left(\mathrm{Ric}_{\tilde{\omega},\v}((d_Y^{c}\psi)^{\sharp})\right), \\
\end{split}
\end{equation}
where $\sharp$ stands for the riemannian duality between $TY$ and $T^*Y$ by using the K\"ahler metric $\tilde{\omega}$, see \cite[Lemma A.2]{AJL}. By $(i)$ the first term of the RHS of the above equality writes as

\begin{equation*}
  (\Delta^Y_{\tilde{\omega},\v})^{2}(\psi) = (\Delta^X_{{\omega},\mathrm{p}\v})^{2}(\psi_b) + (\Delta^B_x)^{2}(\psi_x)  + \Delta^B_x\left(\Delta^X_{\omega,\mathrm{p}\v}\psi_b) \right)_x +  \Delta^X_{\omega,\mathrm{p}\v}\left(\Delta^B_{x} \psi_x \right)_{b}.
\end{equation*}
For the second term of \eqref{liche:decompo} we follow \cite[eq. (79)]{AJL}, and using Lemma \ref{l:Y-Scal} together with the fact that $\mathrm{Scal}_\v(\tilde{\omega})=\v(\mu_{\tilde{\omega}})\Lambda_{{\tilde{\omega}},\v}(\mathrm{Ric}_\v(\tilde{\omega}))$, we obtain that

\begin{equation*}
\begin{split}
      \left(d_Y\right)^{*,\tilde{\omega},\v}\left(\mathrm{Ric}_{\tilde{\omega},\v}\left((d_Y^{c}\psi)^{\sharp,\tilde{\omega}}\right)\right)=  & \left(d^X\right)^{*,\omega,\mathrm{p}\v}\left(\mathrm{Ric}_{\omega,\mathrm{p}\v }\left(\left(d_X^{c}\psi_b\right)^{\sharp,\omega}\right)\right)\\
      &+   \left(d^B\right)^{*,\omega_B}\left(\mathrm{Ric}_{\omega_B}\left(\left(d_B^{c}\psi_x\right)^{\sharp,\omega_B}\right)\right) 
      + \sum_{a=1}^k R_a(\mu_\omega) \Delta^{B_a}_{\omega_{B_a}}(\psi_x),
\end{split}      
\end{equation*}
where $\left(d_Y\right)^{*,\tilde{\omega},\v}$ denote the adjoint operator of $d_Y$ with respect to $\v(\mu_{\tilde{\omega}})\tilde{\omega}^{[n+m]}$ and similarly for $\left(d_X\right)^{*,\omega,\mathrm{p}\v}$ and $d_B^{*,\omega_B}$. Combining the above identities we get $(ii)$; finally
$(iii)$ follows directly from the case when $\v=1$ (\cite[Lemma A.3]{AJL}).
\end{proof}

We can now state the main result of this section which extends \cite[Theorem 2]{AJL}: while the latter concerns the existence of extremal metrics on $Y$ ($\v=1$ and $\w=$affine extremal function), here we deal with the general case of weighted extremal metrics.

\begin{corollary}{\label{c:fibra:sec}}  Let $(\v,\w)$ be weights on $P$ with $\v>0$, $\log$-concave. The following are equivalent:
\begin{itemize}
\item[(i)] $Y$ admits a $(\v,\w)$-extremal metric in $[\tilde{\omega}_0]$;
\item[(ii)] The weighted Mabuchi energy $\M^Y_{\v,\w}$ of $(Y,[\tilde{\omega}_0],\T)$  is $\TC$-coercive;
\item[(iii)] $X$ admits a  $\T$-invariant $(\mathrm{p}\v, \tilde \w)$-cscK metric in the K\"ahler class $[\omega_0]$ (or equivalently $\M^X_{\mathrm{p}\v ,\tilde{\w}}$ is $\TC$-coercive),  where
\[\mathrm{p}(x)=\prod_{a=1}^k(\langle p_a, x\rangle+ c_a)^{m_a}, \qquad \tilde \w(x)=  \w(x)-\sum_{a=1}^k \frac{\mathrm{Scal}(\omega_{B_a})}{\left(\langle p_a, x \rangle + c_a\right)}  ;\].
\item[(iv)] The weighted Mabuchi energy $\M^X_{\mathrm{p}\v,\tilde{\w}}$ of $(X,\T,[\omega_0])$  is $\TC$-coercive.
\end{itemize}
\smallskip
Moreover, suppose that $(X,\omega_0,\T)$ is toric with Delzant polytope $P$. Then the later conditions are equivalent to:

\begin{itemize}
\item[(v)] $P$ is $(\mathrm{p}\v ,\tilde{\w})$ uniformly $K$-stable.
\end{itemize}
\end{corollary}

\begin{proof}
The equivalence $(i)$ \(\Leftrightarrow\) $(ii)$ and $(iii)$ \(\Leftrightarrow\) $(iv)$ follow from Theorem \ref{t:exi}. The equivalence $(ii)$ \(\Leftrightarrow\) $(iii)$ is a straightforward adaptation of \cite[Theorem 6.1]{Jub23} and \cite[Theorem 2]{AJL}, combined with Lemma \ref{l:operators-tilde} and Theorem \ref{t:exi}. Lastly, when \((X,\T)\) is toric, the implication $(v)$ \(\Rightarrow\) $(iii)$ follows from Proposition \ref{p:stable:propre}, while $(iii)$ \(\Rightarrow\) $(v)$ is established in Corollary \ref{c:ytd}.
\end{proof}

\subsection{A sufficient condition for existence on toric fibration}


$(X,\omega_0,\T)$ be a toric K\"ahler manifold of dimension $n$. We recall the cone decomposition used in \cite{DJ, ZZ}. Let \(F_j\) be the facet of \(P\) defined by \(L_j\). 
For each \(j=1,\cdots,d\), let \(P_j\) be the cone with basis \(F_j\) and vertex \(x_0\) as illustrated in Figure~\ref{figure-cone}.
For a function \(f\in \mathcal{C}^\infty(P)\), we denote by \(d_xf\) its differential at \(x\in P\). We then have

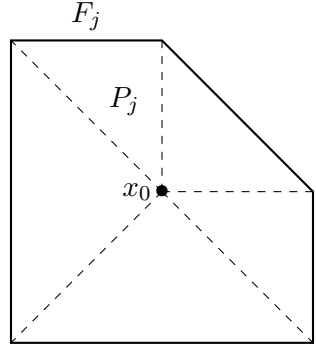
\begin{figure}
\label{figure-cone}
\centering
\caption{The cone decomposition}
\begin{tikzpicture}  
\draw (0,0) node{$\bullet$};
\draw (0,0) node[left]{\(x_0\)};
\draw [thick] (-2,-2) -- (2,-2) -- (2,0) -- (0,2) -- (-2,2) -- cycle;
\draw (-1,2) node[above]{$F_j$};
\draw (-0.5,1.2) node{$P_j$};
\draw [dashed] (0,0) -- (-2,-2);
\draw [dashed] (0,0) -- (2,-2);
\draw [dashed] (0,0) -- (2,0);
\draw [dashed] (0,0) -- (0,2);
\draw [dashed] (0,0) -- (-2,2);
\end{tikzpicture}
\end{figure}

\begin{theorem}[\cite{DJ}]{\label{prop-combinatorial}}
Assume that for all $j=1, \dots, d$, for all \(x \in P_j\), 
\begin{equation}{\label{combinatorial1}}
    \frac{1}{L_j(x_0)}\left(\v(x)( n+1) + d_x \v(x-x_0) \right) - \frac{\w(x)}{2} \geq 0.
\end{equation}
Then $(P,\textbf{L})$ is $(\v,\w)$-uniformly K-stable. \\
In particular, when $X$ is Fano and $[\omega_0]= t c_1(X)$ the condition 
\begin{equation*}
    \frac{1}{t}\left(\v(x)( n+1) + d_x \v(x-x_0) \right) - \frac{\w(x)}{2} \geq 0 \qquad \forall x \in P,
\end{equation*}
implies that $(P,\textbf{L})$ is $(\v,\w)$-uniformly K-stable. 
\end{theorem}


We then arrive at
\begin{corollary}{\label{c:effective:cond}}
Let $(Y,\tilde{\omega}_0, \T)$ be a semisimple principal fibration with toric fiber $(X, \omega_0,\T)$. Assume that the weights $(\mathrm{p}\v, \tilde{\w})$ satisfy the conditions in Theorem \ref{prop-combinatorial} and that $\v$ is $\log$-concave. Then there exists a $(\v,\w)$-cscK metric in $[\tilde{\omega}_0]$.
\end{corollary}
Here the weight $\mathrm{p}$ is the one introduced in Lemma \ref{l:Y-Scal}.
Observe that the case $\mathrm{p}=1$ and $\tilde{\w}=\w$ corresponds to the case when $B$ is a point and $(Y,\tilde{\omega}_0, \T)=(X, \omega_0,\T)$.

\bibliographystyle{plain}
\bibliography{biblio.bib}
\bigskip

 \noindent{\sc Universit\`a di Roma Tor Vergata, \\Via della Ricerca Scientifica 1,
Rome, Italy.\\
\tt {dinezza@mat.uniroma2.it}}
\bigskip
  
 \noindent{\sc Sorbonne Université, Université Paris Cité, CNRS, IMJ-PRG,\\ F-75005 Paris, France,\\
 \tt {simon.jubert@imj-prg.fr}}
\bigskip

\noindent{\sc Campus Saint-Jean, University of Alberta, 8406, 91 Street, Edmonton (AB),
\\ Canada T6C 4G9,\\
 \tt {lahdili.abdellah@gmail.com}}
\bigskip

\end{document}